\documentclass[12pt]{amsart}
\usepackage{amsmath} 
\usepackage{amsfonts} 
\usepackage{amssymb} 
\usepackage{amsthm} 
\usepackage{enumerate}
\usepackage{pstricks}
\usepackage{pstcol,pst-fill,pst-grad}
\numberwithin{equation}{section}

\theoremstyle{plain}
\newtheorem{theorem}{Theorem}[section]
\newtheorem*{theorem*}{Theorem} 
\newtheorem*{maintheorem}{Main Theorem}

\newtheorem{lemma}{Lemma}[section]
\newtheorem*{lemma*}{Lemma} 
\newtheorem{corollary}{Corollary}[section]
\newtheorem*{corollary*}{Corollary} 
\newtheorem{consequence}{Consequence}[section]
\newtheorem*{consequence*}{Consequence} 
\newtheorem{proposition}{Proposition}[section]
\newtheorem*{proposition*}{Proposition} 

\newtheorem*{conjecture*}{Conjecture} 

\theoremstyle{definition}
\newtheorem{definition}{Definition}[section]
\newtheorem*{definition*}{Definition} 
\newtheorem{remark}{Remark}[section]
\newtheorem*{remark*}{Remark} 
\newtheorem*{remarks*}{Remarks} 
\newtheorem*{notations}{Notations} 

\newtheorem*{question*}{Question} 
\newtheorem*{questions*}{Questions} 

\newtheorem*{example*}{Example} 
\newtheorem*{examples*}{Examples} 

\newtheorem*{exercise*}{Exercise} 
\newtheorem*{exercises*}{Exercises} 

\theoremstyle{plain}

\newtheorem*{thm*}{Théorème} 

\newtheorem*{lemme*}{Lemme} 

\newtheorem*{cor*}{Corollaire} 

\newtheorem*{csq*}{Conséquence} 

\newtheorem*{prop*}{Proposition} 

\newtheorem*{conj*}{Conjecture} 

\theoremstyle{definition}

\newtheorem*{déf*}{Définition} 

\newtheorem*{rem*}{Remarque} 
\newtheorem*{rems*}{Remarques} 

\newtheorem*{ex*}{Exemple} 
\newtheorem*{exs*}{Exemples} 

\newtheorem*{exo*}{Exercice} 
\newtheorem*{exos*}{Exercices} 

\theoremstyle{remark}

\newcommand{\disp}{\displaystyle}

\renewcommand{\a}{\alpha}
\renewcommand{\b}{\beta}
\renewcommand{\c}{\gamma}
\newcommand{\C}{\Gamma}
\renewcommand{\d}{\delta}
\newcommand{\D}{\Delta}
\newcommand{\e}{\varepsilon}
\newcommand{\f}{\varphi}
\newcommand{\F}{\Phi}

\renewcommand{\k}{\kappa}
\renewcommand{\l}{\lambda}
\renewcommand{\L}{\Lambda}

\newcommand{\n}{\nu}
\newcommand{\om}{\omega}
\newcommand{\Om}{\Omega}
\newcommand{\p}{\psi}
\renewcommand{\P}{\Psi}
\renewcommand{\r}{\rho}
\newcommand{\s}{\sigma}
\renewcommand{\S}{\Sigma}
\renewcommand{\t}{\theta}

\newcommand{\Ups}{\Upsilon} 
\newcommand{\x}{\xi}

\newcommand{\cD}{\mathcal{D}}

\newcommand{\cF}{\mathcal{F}}

\newcommand{\cL}{\mathcal{L}}

\newcommand{\cT}{\mathcal{T}}
\newcommand{\cU}{\mathcal{U}}
\newcommand{\cV}{\mathcal{V}}

\newcommand{\gF}{\mathfrak{F}}

\newcommand{\gS}{\mathfrak{S}}

\newcommand{\imp}{\Longrightarrow} 
\newcommand{\con}{\Longleftarrow} 

\newcommand{\st}{~|~} 

\newcommand{\Oset}{\mbox{\large $\varnothing$}} 

\newcommand{\inc}{\subset} 
\newcommand{\setmin}{\raisebox{0.45ex}{\scriptsize $\smallsetminus$}} 

\renewcommand{\to}{\longrightarrow} 
\newcommand{\map}[5]{\begin{array}{rcl} #1 #2 & \to & #3 \\ #4 & \longmapsto & #5 \end{array}} 
\newcommand{\I}[1]{\mathfrak{I}_{\! {\mbox{\tiny $#1$}}}} 
\newcommand{\rest}[2]{#1_{\mathbf{|} #2}} 

\newcommand{\act}{\! \cdot \!} 

\renewcommand{\fam}[3]{\left( #1_{#2} \right)_{\! #2 \in #3}} 

\newcommand{\flowR}[2]{\left( #1^{#2} \right)_{\! #2 \in \RR}} 
 
\newcommand{\card}[1]{\mathrm{card} \left( #1 \right)} 

\newcommand{\NN}{\mathbf{N}} 
\newcommand{\ZZ}{\mathbf{Z}} 
\newcommand{\RR}{\mathbf{R}} 

\newcommand{\Rn}[1]{\mathbf{R}^{#1}} 
\newcommand{\Sn}[1]{\mathbf{S}^{#1}} 
\newcommand{\Hn}[1]{\mathbf{H}^{#1}} 
\newcommand{\Tn}[1]{\mathbf{T}^{#1}} 

\renewcommand{\leq}{\leqslant} 
\renewcommand{\geq}{\geqslant} 

\let\oldint\int
\renewcommand{\int}[4]{\oldint_{\! #1}^{#2} \!\!\! #3 \, \mathrm{d} #4} 

\newcommand{\LL}[1]{\mathrm{L}^{#1}} 

\newcommand{\goes}{\rightarrow} 
 
\renewcommand{\bar}{\overline} 

\renewcommand{\dim}[1]{\mathrm{dim} \left( #1 \right)} 
 
\let\olddet\det
\renewcommand{\det}[1]{\olddet{\! \left( #1 \right)}} 
 
\renewcommand{\O}{\mathrm{O}} 

\newcommand{\norm}[1]{\left\| #1 \right\|} 
\newcommand{\opnorm}[1]{\left| \! \left| \! \left| #1 \right| \! \right| \! \right|} 

\newcommand{\scal}[2]{\left\langle #1 , #2 \right\rangle} 

\newcommand{\vol}{\mathrm{vol}} 

\newcommand{\Tg}[3]{T_{\! #1} #2 \! \cdot \! #3} 

\newcommand{\pder}[2]{\frac{\partial #1}{\partial #2}} 
\newcommand{\ppder}[3]{\frac{\partial^{2} #1}{\partial #2 \partial #3}} 
 
\newcommand{\ed}[1]{\mathrm{d} #1} 

\newcommand{\Cl}[1]{\mathrm{C}^{#1}} 

\newcommand{\bC}{\partial C}
\newcommand{\bD}{\partial \cD}
\newcommand{\sF}{\mathsf{F}}
\newcommand{\iP}{\mathnormal{\P}}
\renewcommand{\vol}[2]{\mathrm{vol}_{#1} \! \left( #2 \right)}
\newcommand{\Vol}[2]{\mathrm{Vol}_{#1} \! \left( #2 \right)}
\renewcommand{\inc}{\subseteq}

\begin{document}

\title[]{Some smooth Finsler deformations \\ 
of hyperbolic surfaces}

\author{Bruno Colbois}
\address{Bruno Colbois, 
Universit\'{e} de Neuch\^{a}tel, 
Institut de math\'{e}matique, 
Rue \'{E}mile Argand 11, 
Case postale 158, 
CH--2009 Neuch\^{a}tel, 
Switzerland}
\email{bruno.colbois@unine.ch}

\author{Florence Newberger}
\address{Florence Newberger, 
Department of Mathematics, 
California State University, 
Long Beach, 
Long Beach, CA, 
USA}
\email{fnewberg@csulb.edu}

\author{Patrick Verovic}
\address{Patrick Verovic, 
UMR 5127 du CNRS \& Universit\'{e} de Savoie, 
Laboratoire de math\'{e}matique, 
Campus scientifique, 
73376 Le Bourget-du-Lac Cedex, 
France}
\email{verovic@univ-savoie.fr}

\date{\today}


\begin{abstract}
    
Given a closed hyperbolic Riemannian surface, the aim of the present paper is to describe 
an explicit construction of \emph{smooth} deformations of the hyperbolic metric into Finsler metrics 
that are \emph{not Riemannian} and whose properties are such that the classical Riemannian results 
about entropy rigidity, marked length spectrum rigidity and boundary rigidity all \emph{fail} 
to extend to the Finsler category. 

\end{abstract}

\maketitle

\bigskip
\bigskip


\section{Introduction}

In this paper, we construct Finsler metrics on hyperbolic surfaces, proving
that certain recent Riemannian rigidity results fail to extend to the Finsler
category. Recall that a Finsler metric $\cF$ on a manifold $M$ is a continuous
function $\cF : TM \to \RR$ such that $\cF(p , \cdot)$ is a norm on $T_{p}M$ for any $p \in M$.
If in addition $\cF$ is $\Cl{\infty}$ on $TM \setmin \{ 0 \}$ (the tangent bundle minus the zero section), 
then $\cF$ is said to be \emph{smooth}.
In that case, $\cF$ is called \emph{strongly convex} iff for any $(p , v) \in TM \setmin \{ 0 \}$, 
the symmetric bilinear form $\frac{\partial^{2} \cF^{2}}{\partial v^{2}}(p , v)$ on $T_{p}M$ 
is positive definite.

\smallskip

Note that a Riemannian metric $g$ on $M$ gives rise to an associated
Finsler metric $\cF$ defined by $\cF(p , v) = (g(p) \! \act \! (v , v))^{1 / 2}$.
We will then say that $\cF$ is Riemannian.

\medskip 

Perhaps the most significant Riemannian rigidity results are the minimal entropy
rigidity theorems of Katok for surfaces and of Besson, Courtois, and Gallot in higher dimensions, 
due to their many applications
(see for example the surveys \cite{BCG95} and \cite{Cro02}). We begin by stating
these two theorems and one of their consequences, relevant to our work (for a more
general version in higher dimensions, see~\cite{BCG96}).

\medskip 

Given a Finsler metric $F$ on a simply connected manifold $\tilde{M}$, 
the volume growth entropy of $F$, denoted by $h(F)$, 
is the asymptotic exponential growth rate of $F$-balls in $\tilde{M}$, \emph{i.e.}, 
\begin{equation*}
   h(F) := \limsup_{R \goes +\infty} \frac{1}{R} \log{(\Vol{F}{B_{F}(x , R)})} \in [0 , +\infty] 
\end{equation*}
for an arbitrary $x \in \tilde{M}$, where $B_{F}(x , R)$ is the open ball of radius $R$ in
$\tilde{M}$ about $x$ with respect to $F$, 
and $\textrm{Vol}_{F}$ denotes the Holmes-Thompson volume on $\tilde{M}$ associated with $F$
(see Section~\ref{sec:entropy} for the definition). 

\smallskip 

By extending this definition, if $\cF$ is a Finsler metric on a \emph{compact} manifold $M$ 
whose universal cover is $\tilde{M}$, the upper limit above is actually a limit, and 
the volume growth entropy $h(\cF)$ of $\cF$ is defined to be equal to that $h(F)$ 
of the lift $F$ to $\tilde{M}$ of the metric $\cF$. 

\bigskip 

\begin{theorem}[Katok, \cite{Kat82}] \label{thm:Katok} 
   Let $(S , g_{0})$ be a closed hyperbolic (Riemannian) surface, 
   and let $g$ be a Riemannian metric on $S$. 
   Then, denoting by $\mathrm{vol}$ the usual Riemannian volume, we have
   
   \begin{enumerate}
      \item $h(g_{0})^{\! 2} \, \vol{g_{0}}{S} \leq h(g)^{\! 2} \, \vol{g}{S}$, and 
      
      \smallskip 

      \item $h(g_{0})^{\! 2} \, \vol{g_{0}}{S} = h(g)^{\! 2} \, \vol{g}{S}$ 
      if and only if $g$ is hyperbolic. 
\end{enumerate}
\end{theorem}

\medskip 

Later, Besson, Courtois and Gallot extended the first part (inequality) of this result 
to higher dimensions and obtained something different for the second part (rigidity):  

\begin{theorem}
[Besson--Courtois--Gallot, \cite{BCG95} and~\cite{BCG96}] \label{thm:BCG} 
   Let $(M , g_{0})$ be a closed $n$-dimensional Riemannian locally symmetric space of negative curvature 
   with $n \geq 3$, and let $(N , g)$ be a closed negatively curved Riemannian manifold homotopy 
   equivalent to $(M , g_{0})$. 
   Then, denoting by $\mathrm{vol}$ the usual Riemannian volume, we have
   
   \begin{enumerate}
      \item $h(g_{0})^{\! n} \, \vol{g_{0}}{M} \leq h(g)^{\! n} \, \vol{g}{N}$, and 
      
      \smallskip 

      \item $h(g_{0})^{\! n} \, \vol{g_{0}}{M} = h(g)^{\! n} \, \vol{g}{N}$ 
      if and only if $(N , g)$ is homothetic to $(M , g_{0})$.
\end{enumerate}
\end{theorem}

\bigskip 

An important corollary of these results is the boundary rigidity of negatively
curved symmetric spaces (see~\cite{Cro02} for a general discussion of boundary
rigidity). Let $(M , g_{0})$ denote a compact connected Riemannian manifold with
a non-empty boundary $\partial M$, and let $d_{g_{0}}$
denote the induced metric on $\partial M$ by the distance function on $M$ associated with $g_{0}$. 
Such a manifold is called boundary rigid if and only if for any compact connected Riemannian 
manifold $(N , g)$ with non-empty boundary $\partial N$, any \emph{metric} isometry 
$(\partial M , d_{g_{0}}) \to (\partial N , d_{g})$ extends to a \emph{smooth} 
isometry $(M ,g_{0}) \to (N , g)$, where $d_{g}$ is defined the same way as $d_{g_{0}}$. 

\begin{theorem}[\cite{Cro02}, Corollary~6.3] \label{thm:Riemann-boundary-rigid}
   Any bounded domain in a Riemannian symmetric space of negative curvature 
   has a closure that is boundary rigid.
\end{theorem}

\medskip 

Another Riemannian rigidity result is related to the marked length spectrum. 
The marked length spectrum of a Finsler manifold $(M , \cF)$ is the map 
that assigns to each free homotopy class of $M$ the 
$\cF$-length of a shortest closed parameterized curve $[0 , 1] \to M$ 
(thus a $\cF$-geodesic since it is locally $\cF$-length minimizing) 
in that free homotopy class. 

\begin{theorem}[\cite{CroSha98}, Theorem 1.1. See also~\cite{Cro02}, Theorem~8.2] 
   Let $g_{0}$ be a negatively curved Riemannian metric on a closed manifold $M$.  
   Let $(g_{\l})_{\l \in (-\e , \e)}$ be a smooth variation of $g_{0}$ through Riemannian metrics 
   on $M$ 
   such that for each $\l \in (-\e , \e)$, the marked length spectrum of $g_{\l}$ 
   is the same as that of $g_{0}$. 
   Then $g_{\l}$ is isometric to $g_{0}$ for all $\l \in (-\e , \e)$.
\end{theorem}

\bigskip 

We now state our main result.

\begin{maintheorem}
   Let $\cD$ denote an open ball in the two-dimensional hyperbolic space $(\Hn{2} , g_{0})$. 
   Then there exist $\e > 0$ and a continuous function 
   $\F : (-\e , \e) \times T\Hn{2} \to \RR$ that is $\Cl{\infty}$ 
   on $(-\e , \e) \times (T\Hn{2} \setmin \{ 0 \})$
   and such that for each $\l \in (-\e , \e)$, we have 
   
   \begin{enumerate}      
      \item $\sF_{\l}(\cdot) := \F(\l , \cdot)$ is a \emph{smooth} 
      strongly convex Finsler metric on $\Hn{2}$, 

      \smallskip 

      \item $\sF_{0}$ is associated with $g_{0}$, and $\sF_{\l}$ is \emph{not Riemannian} 
      whenever $\l \neq 0$, 

      \smallskip 

      \item $d_{\sF_{\l}} = d_{\sF_{0}}$, where $d_{\sF_{0}}$ and $d_{\sF_{\l}}$ 
      are the metrics induced on $\bD$ 
      by the distance functions on $\Hn{2}$ associated respectively with $\sF_{0}$ and $\sF_{\l}$, 

      \smallskip 

      \item every two points in $\cD$ can be joined by a geodesic of $\sF_{\l}$ 
      whose image is contained in $\cD$, 

      \smallskip 

      \item $\sF_{\l}(x , u) = \sF_{0}(x , u)$ for all $x \in \Hn{2} \setmin \cD$ and $u \in T_{x}\Hn{2}$, 

      \smallskip 

      \item $\sF_{\l}$ has no conjugate points. 
\end{enumerate}
\end{maintheorem}

\medskip 

It is important to point out that the greatest difficulty in the proof of this theorem 
is the \emph{smoothness} property we expect from our family of Finsler metrics. 
Indeed, how to get a Finsler metric on an open disc in $\Hn{2}$ 
that induces a given distance function on the boundary of that disk 
is a construction that is well known (see~\cite{Arc94}). 
However, we want here to construct a family of Finsler metrics on an open disk in $\Hn{2}$ 
that all induce on the boundary of that disk the same distance function as that 
induced by the hyperbolic Riemannian metric $g_{0}$ on $\Hn{2}$, and that extend 
to $g_{0}$ outside the disk in a \emph{smooth} way, this latter point being not something classical. 
Moreover, we expect the extended family $\fam{F}{\l}{(-\e , \e)}$ to be also 
$\emph{smooth}$ with respect to the real parameter $\l$ since this may be useful 
in studying the behaviour of some invariants associated with these Finsler metrics 
by differentiating them with respect to $\l$.  

\bigskip

As corollaries to our Main Theorem, we get that \emph{all} of the Riemannian
rigidity results stated above \emph{fail} to extend to the Finsler category for \emph{surfaces}. 
In particular, Katok's rigidity result about closed hyperbolic surfaces 
(see Theorem~\ref{thm:Katok}, point~(2)) does \emph{not} hold any longer for Finsler metrics. 
By the way, it is interesting to note that Besson, Courtois and Gallot 
conjectured in their paper~\cite{BCG95} (page~630) that Theorem~\ref{thm:BCG} 
should remain true in the Finslerian context. 

\medskip

\begin{corollary} \label{cor:boundary}
   Let $(S , g_{0})$ be a closed hyperbolic surface. 
   Then there exist a domain $\Om$ in $S$ with non-empty boundary $\partial \Om$,
   a number $\e > 0$ and a continuous function 
   $\iP : (-\e , \e) \times TS \to \RR$ that is $\Cl{\infty}$ on $(-\e , \e) \times (TS \setmin \{ 0 \})$
   and such that for each $\l \in (-\e , \e)$, we have 
   
   \begin{enumerate}      
      \item $\cF_{\l}(\cdot) := \iP(\l , \cdot)$ is a \emph{smooth} strongly convex Finsler metric on $S$, 

      \smallskip 

      \item $\cF_{0}$ is associated with $g_{0}$, and $\cF_{\l}$ is \emph{not Riemannian} 
      whenever $\l \neq 0$, 

      \smallskip 

      \item $d_{\cF_{\l}} = d_{\cF_{0}}$, where $d_{\cF_{0}}$ and $d_{\cF_{\l}}$ 
      are the metrics induced on $\partial \Om$ by the distance functions 
      on $S$ associated respectively with $\cF_{0}$ and $\cF_{\l}$.
   \end{enumerate}
\end{corollary}

\medskip

\begin{corollary} \label{cor:spectrum}
   Let $(S , g_{0})$ be a closed hyperbolic surface. 
   Then there exist $\e > 0$ and a continuous function 
   $\iP : (-\e , \e) \times TS \to \RR$ that is $\Cl{\infty}$ on $(-\e , \e) \times (TS \setmin \{ 0 \})$
   and such that for each $\l \in (-\e , \e)$, we have 
   
   \begin{enumerate}      
      \item $\cF_{\l}(\cdot) := \iP(\l , \cdot)$ is a \emph{smooth} strongly convex Finsler metric on $S$, 

      \smallskip 

      \item $\cF_{0}$ is associated with $g_{0}$, and $\cF_{\l}$ is \emph{not Riemannian} 
      whenever $\l \neq 0$, 

      \smallskip 

      \item the marked length spectrum of $\cF_{\l}$ is equal to that of $\cF_{0}$.
   \end{enumerate}
\end{corollary}

\medskip

\begin{corollary} \label{cor:entropy}
   Let $(S , g_{0})$ be a closed hyperbolic surface. 
   Then there exist $\e > 0$ and a continuous function 
   $\iP : (-\e , \e) \times TS \to \RR$ that is $\Cl{\infty}$ on $(-\e , \e) \times (TS \setmin \{ 0 \})$
   and such that for each $\l \in (-\e , \e)$, we have 
   
   \begin{enumerate}      
      \item $\cF_{\l}(\cdot) := \iP(\l , \cdot)$ is a \emph{smooth} strongly convex Finsler metric on $S$, 

      \smallskip 

      \item $\cF_{0}$ is associated with $g_{0}$, and $\cF_{\l}$ is \emph{not Riemannian} 
      whenever $\l \neq 0$, 

      \smallskip 

      \item $h(\cF_{\l})^{\! 2} \, \Vol{\cF_{\l} \!}{S} = h(\cF_{0})^{\! 2} \, \Vol{\cF_{0} \!}{S}$,
      where $\mathrm{Vol}$ denotes the Holmes-Thompson volume.
   \end{enumerate}
\end{corollary}

\bigskip
\bigskip
\bigskip


\section{Proofs of the Corollaries} \label{sec:entropy}

In this section, we will explain how the Main Theorem implies Corollaries~\ref{cor:boundary},
\ref{cor:spectrum} and \ref{cor:entropy}. 

Let $(\Hn{2} , g_{0})$ be the two-dimensional hyperbolic space and let $F_{0}$ 
be the (smooth strongly convex) Finsler metric associated with $g_{0}$.  
Let $\C$ be a discrete cocompact subgroup of $g_{0}$-isometries acting properly discontinuously on
$\Hn{2}$ without fixed points. 
Then $S = \Hn{2} / \C$ is a closed surface endowed with the quotient hyperbolic metric $g_{0}$ 
and the projection map $\pi : \Hn{2} \to S$ is a Riemannian covering.
Finally, let $\cD$ be an open ball in $(\Hn{2} , g_{0})$ such that $\rest{\pi}{\bar{\cD}}$ is injective, 
where $\bar{\cD}$ stands for the closure of $\cD$ in $(\Hn{2} , g_{0})$. 

\smallskip 

\begin{notations}  
For any Finsler metric $\cF$ on $S$, we will denote by $F$ the lift of $\cF$ to $\Hn{2}$ 
and by $D_{F}$ the distance function on $\Hn{2}$ associated with $F$. 
The $\cF$-length (respectively $F$-length) of curves will be denoted by $L_{\cF}$ 
(respectively $L_{F}$), and the induced metric on $\bD$ by $D_{F}$ will be denoted by $d_{F}$. 
For $x \in \Hn{2}$ and $R > 0$, let $B_{F}(x , R)$ denote the open ball about $x$ in
$\Hn{2}$ of radius $R$ with respect to $D_{F}$.
In addition, given $x , y \in \Hn{2}$, any $D_{F}$-distance minimizing 
geodesic $[0 , 1] \to \Hn{2}$ connecting $x$
to $y$ will be denoted by $[x , y]_{F}$ (it is to be noticed that $\cF$ is complete since 
$S$ is closed, and thus $F$ is complete too).
\end{notations}

\bigskip

Corollary~\ref{cor:boundary} will be a straightforward consequence of the following lemma:  

\begin{lemma} \label{lem:quotient}
   There exist $\e > 0$ and a continuous function 
   $\iP : (-\e , \e) \times TS \to \RR$ that is $\Cl{\infty}$ on $(-\e , \e) \times (TS \setmin \{ 0 \})$
   and such that for each $\l \in (-\e , \e)$, we have 
   
   \begin{enumerate}      
      \item $\cF_{\l}(\cdot) := \iP(\l , \cdot)$ is a smooth strongly convex Finsler metric on $S$, 

      \smallskip 

      \item $\cF_{0}$ is associated with $g_{0}$, and $\cF_{\l}$ is not Riemannian whenever $\l \neq 0$, 

      \smallskip 

      \item $d_{F_{\l}} = d_{F_{0}}$, 

      \smallskip 

      \item every two points in $\cD$ can be joined by a geodesic of $F_{\l}$ 
      whose image is contained in $\cD$, 

      \smallskip 

      \item $F_{\l}(x , u) = F_{0}(x , u)$ for all $x \in \Hn{2} \setmin \C(\cD)$ and $u \in T_{x}\Hn{2}$, 

      \smallskip 

      \item $F_{\l}$ has no conjugate points.      
   \end{enumerate}
\end{lemma}

\medskip

\begin{proof}~\\ 
\indent Roughly speaking, the proof will first consist in spreading out 
the family $\fam{\sF}{\l}{(-\e , \e)}$ 
obtained in the Main Theorem all over $\Hn{2}$ by the deck transformations 
of the universal covering $\pi : \Hn{2} \to S$. 
Then, we will get a new family $\fam{F}{\l}{(-\e , \e)}$ of Finsler metrics on $\Hn{2}$ 
that are invariant under the group $\C$, which will make it possible to consider their 
quotients on the surface $S$. 
This will finally give rise to a family $\fam{\cF}{\l}{(-\e , \e)}$ of Finsler metrics on $S$, 
each of them being equal to the Riemannian hyperbolic metric $g_{0}$ 
outside a small topological disk.   

\smallskip 

Let $\e > 0$ and $\F : (-\e , \e) \times T\Hn{2} \to \RR$ as given in the Main Theorem. 

\smallskip 

Since $\rest{\pi}{\bar{\cD}}$ is injective, $\bar{\cD}$ is compact 
and $\pi : \Hn{2} \to S$ is a covering map,
there exists an open set $U$ in $\Hn{2}$ such that $\bar{\cD} \inc U$ and 
$\rest{\pi}{U}$ is still injective.
So, as $\F(\l , (x , u)) = \sF_{0}(x , u)$ for all $\l \in (-\e , \e)$, 
$x \in \Hn{2} \setmin \cD$ and $u \in T_{x}\Hn{2}$
by property~(5) in the Main Theorem,
we can define the new continuous function $\widehat{\F} : (-\e , \e) \times T\Hn{2} \to \RR$
by setting 
$\widehat{\F}(\l , (\c(x) , \Tg{x}{\c}{u})) = \sF_{\l}(x , u)$ for all $\c \in \C$ if $x \in U$ 
and $u \in T_{x}\Hn{2}$, and 
$\widehat{\F}(\l , (x , u)) = \sF_{0}(x , u)$ if $x \in \Hn{2} \setmin \C(U)$ and $u \in T_{x}\Hn{2}$. 

\smallskip 

Note that this definition makes sense since we have $\c(U) \cap U = \Oset$ 
for all $\c \in \C$ with $\c \neq \I{\Hn{2}}$ 
(indeed, if $\c \in \C$ and $x_{0} \in U$ are such that $\c(x_{0}) \in U$, 
then necessarily $\c(x_{0}) = x_{0}$
by injectivity of $\rest{\pi}{U}$, and hence $\c = \I{\Hn{2}}$ since $\C$ has no fixed points).

\smallskip 

This function $\widehat{\F}$ is then $\Cl{\infty}$ on $(-\e , \e) \times (T\Hn{2} \setmin \{ 0 \})$
and for each $\l \in (-\e , \e)$ satisfies 
  
\begin{enumerate}[$(i)$] 
   \item $\widehat{\F}(\l , \cdot)$ is a smooth strongly convex Finsler metric on $\Hn{2}$, 
   
   \item $\widehat{\F}(0 , \cdot)$ is associated with $g_{0}$, and $\widehat{\F}(\l , \cdot)$ 
   is not Riemannian whenever $\l \neq 0$. 
\end{enumerate}

\smallskip 

Since $\widehat{\F}(\l , (\c(x) , \Tg{x}{\c}{u})) = \widehat{\F}(\l , (x , u))$ 
for all $\l \in (-\e , \e)$, $(x , u) \in T\Hn{2}$ 
and $\c \in \C$ by construction,
the quotient function $\iP : (-\e , \e) \times TS \to \RR$ 
given by $\iP(\l , T\pi(x , u)) = \widehat{\F}(\l , (x , u))$ is well defined and immediately 
satisfies points~(1) and~(2) of Lemma~\ref{lem:quotient} thanks to points $(i)$ and $(ii)$ above. 
Furthermore, points~(3) to~(6) in the Main Theorem automatically yield points~(3) to~(6) 
of Lemma~\ref{lem:quotient}. 
\end{proof}

\bigskip

\begin{proof}[Proof of Corollary~\ref{cor:boundary}]~\\ 
\indent Choose $\Om = \pi(\cD)$ and apply Lemma~\ref{lem:quotient}.
\end{proof}

\bigskip

On the other hand, keeping in mind that $d_{F_{0}}$ denotes the induced metric on $\bD$ 
by the distance function $D_{F_{0}}$ on $\Hn{2}$ associated with $F_{0}$, 
Corollaries~\ref{cor:spectrum} and \ref{cor:entropy} will need the following lemma:

\begin{lemma} \label{lem:distance}
   Let $\cF$ be a Finsler metric on $S$ such that we have
   
   \begin{enumerate}
      \item $d_{F} = d_{F_{0}}$, and 

      \smallskip 

      \item $F(x , u) = F_{0}(x , u)$ for all $x \in \Hn{2} \setmin \C(\cD)$ and $u \in T_{x}\Hn{2}$. 
   \end{enumerate}
   Then $D_{F}(x , y) = D_{F_{0}}(x , y)$ for all $x , y \in \Hn{2} \setmin \C(\cD)$. 
\end{lemma}

\medskip

\begin{proof}~\\ 
\indent Fix $x , y \in \Hn{2} \setmin \C(\cD)$
and let us consider a $F$-distance minimizing geodesic $[x , y]_{F}$ connecting $x$ to $y$. 

We will construct a curve $\s$ also connecting $x$ to $y$ such that
$L_{F_{0}}(\s) = L_{F}([x , y]_{F})$, and hence conclude that
$D_{F_{0}}(x , y) \leq L_{F_{0}}(\s) = L_{F}([x , y]_{F}) = D_{F}(x , y)$. 
A similar argument will give the reverse inequality, proving the lemma.

\smallskip

As the image of $[x , y]_{F}$ is compact,
it intersects only a finite number $N \geq 0$ of connected components 
of the open set $\C(\cD) = \pi^{-1}(\pi(\cD))$ in $\Hn{2}$. 

\smallskip

If $N = 0$, \emph{i.e.}, if $[x , y]_{F}$ does not enter $\C(\cD)$, 
then hypothesis~(2) implies $L_{F}([x , y]_{F}) = L_{F_{0}}([x , y]_{F})$, 
hence we may take $\s = [x , y]_{F}$ and obtain the result.

\smallskip

Suppose $N \geq 1$, and let $C_{1}$ be the first connected component of $\C(\cD)$ that
$[x , y]_{F}$ enters. Let $e_{1}$ be the point at which $[x , y]_{F}$ enters
$C_{1}$ the first time, and $o_{1}$ be the point at which $[x , y]_{F}$ leaves $C_{1}$ the last time.   
Similarly let $C_{2}$ be the first connected component of
$\C(\cD)$ met by the geodesic $[o_{1} , y]_{F}$, define $e_{2}$ to be the point at
which $[o_{1} , y]_{F}$ enters $C_{2}$ the first time, and $o_{2}$ to be the point at which
$[o_{1} , y]_{F}$ leaves $C_{2}$ the last time. 
Continuing in this fashion, we define finite sequences 
$(C_{i})_{1 \leq i \leq k}$, $(e_{i})_{1 \leq i \leq k}$ 
and $(o_{i})_{1 \leq i \leq k}$ by induction, where $k \in \{ 1, \ldots , N \}$ is such that  
the image of $[o_{k} , y]_{F}$ does not intersect $\C(\cD)$ 
(see Figure~\ref{fig:lem-distance}).

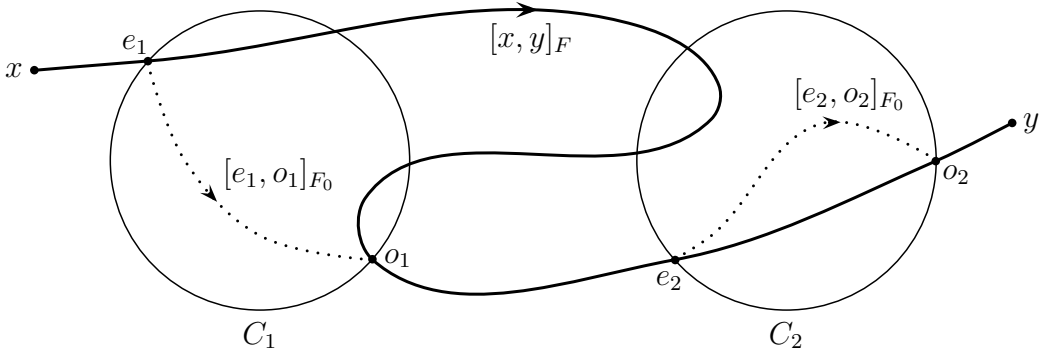
\begin{figure}[h] 


\begin{pspicture}(15,5) 
   \pscircle[linewidth=0.02](4,2.5){2} 
   \pscircle[linewidth=0.02](11,2.5){2} 
   \uput{0.15}[270](4,0.5){$C_{1}$} 
   \uput{0.15}[270](11,0.5){$C_{2}$} 
   \psdots[dotstyle=*,dotscale=0.9](1,3.7)(14,3)(2.51,3.823)(5.495,1.187)(9.52,1.177)(12.985,2.49) 
   \uput{0.15}[180](1,3.7){$x$} 
   \uput{0.15}[0](14,3){$y$} 
   \uput{0.1}[120](2.5,3.823){$e_{1}$} 
   \uput{0.15}[5](5.5,1.177){$o_{1}$} 
   \uput{0.15}[260](9.5,1.177){$e_{2}$} 
   \uput{0.1}[315](13,2.5){$o_{2}$} 
   \pscurve[linewidth=0.04](1,3.7)(2.5,3.823)(9,4.3)(10.1,3.5)(10,3.05)
   (5.4,2)(5.5,1.177)(9.5,1.177)(13,2.5)(14,3) 
   \psline[linewidth=0.08]{c->}(7.56,4.505)(7.7,4.505) 
   \uput{0.17}[270](7.6,4.5){$[x , y]_{F}$} 
   \pscurve[linewidth=0.04,linestyle=dotted](2.5,3.823)(3,2.5)(4,1.5)(5.5,1.177) 
   \psline[linewidth=0.06]{c->}(3.338,2.05)(3.43,1.95) 
   \uput{0.15}[45](3.4,2){$[e_{1} , o_{1}]_{F_{0}}$} 
   \pscurve[linewidth=0.04,linestyle=dotted](9.5,1.177)(10,1.5)(11.5,3)(13,2.5) 
   \psline[linewidth=0.06]{c->}(11.58,3.005)(11.72,3.01) 
   \uput{0.17}[100](12,3){$[e_{2} , o_{2}]_{F_{0}}$} 
\end{pspicture}  
   \caption{\label{fig:lem-distance} Proof of Lemma~\ref{lem:distance}} 
\end{figure}

\smallskip 

Now we have
$$
L_{F}([x , y]_{F}) = L_{F}([x , e_{1}]_{F}) + L_{F}([e_{1} , o_{1}]_{F}) 
+ L_{F}([o_{1} , e_{2}]_{F}) + \cdots + L_{F}([e_{k} , o_{k}]_{F}) + L_{F}([o_{k} , y]_{F}).
$$

\smallskip 

The $F$-length and the $F_{0}$-length of each segment lying entirely
outside $\C(\cD)$ are already equal by hypothesis~(2). 
In particular, we have
$$
L_{F}([o_{i} , e_{i + 1}]_{F}) = L_{F_{0}}([o_{i} , e_{i + 1}]_{F})
$$
for each $1 \leq i \leq k - 1$ (in case $k \geq 2$), as well as
$$
L_{F}([x , e_{1}]_{F}) = L_{F_{0}}([x , e_{1}]_{F})
$$
and
$$
L_{F}([o_{k} , y]_{F}) = L_{F_{0}}([o_{k} , y]_{F}).
$$

\smallskip 

Furthermore, for each $i \in \{ 1 , \ldots , k \}$, 
observing that $e_{i}$ and $o_{i}$ lie in $\partial C_{i}$, 
we have
\begin{equation*}
   L_{F_{0}}([e_{i} , o_{i}]_{F_{0}}) = D_{F_{0}}(e_{i} , o_{i}) 
   = 
   d_{F_{0}}(e_{i} , o_{i}) = d_{F}(e_{i} , o_{i}) 
   = 
   D_{F}(e_{i} , o_{i}) = L_{F}([e_{i} , o_{i}]_{F}) 
\end{equation*}
by hypothesis~(1). 

\smallskip 

Thus, if $\s$ is the curve
$$
\s = [x , e_{1}]_{F} \# [e_{1} , o_{1}]_{F_{0}} \# [o_{1} , e_{2}]_{F} 
\# \cdots \#
[e_{k - 1} , o_{k - 1}]_{F_{0}} \# [o_{k - 1} , e_{k}]_{F} \# [e_{k} , o_{k}]_{F_{0}} \# [o_{k} , y]_{F},
$$
where $\#$ is the concatenation operator,
then $L_{F_{0}}(\s) = L_{F}([x , y]_{\cF})$, as desired.

\smallskip

The same argument reversing the roles of ${F_{0}}$ and $F$ shows that
$D_{F}(x , y) \leq D_{F_{0}}(x , y)$, 
and hence we have $D_{F}(x , y) = D_{F_{0}}(x , y)$, completing the proof.
\end{proof}

\bigskip

\begin{proof}[Proof of Corollary~\ref{cor:spectrum}]~\\ 
\indent Let $\e > 0$ and $\iP : (-\e , \e) \times TS \to \RR$ as given by Lemma~\ref{lem:quotient}. 

\smallskip 

Fix $\l \in (-\e , \e)$ and consider a free homotopy class $\S$ of $S$ that is \emph{not} trivial.
As in the proof of Lemma~\ref{lem:quotient}, let $U$ be an open set in $\Hn{2}$ such that $\bar{\cD} \inc U$ 
and $\rest{\pi}{U}$ is injective. 

\medskip

If $\s_{\l} : [0 , 1] \to S$ is a closed curve of shortest $\cF_{\l}$-length within $\S$, 
the image $\s_{\l}([0 , 1])$ of $\s_{\l}$ can not entirely lie in $\pi(\cD)$. 
Indeed, if this were the case, the image of the curve $f^{-1} \circ \s_{\l}$
would be included in $\cD$, where $f : (U , F_{\l}) \to (\pi(U) , \cF_{\l})$ is the isometry
induced by $\pi$ and $F_{\l}$ is the lift to $\Hn{2}$ of $\cF_{\l}$. 
But $f^{-1} \circ \s_{\l}$ is contractible to a point in $\cD$ 
and thus $\s_{\l}$ would be contractible to a point in $\pi(\cD)$,
which is not possible since $\S$ is not trivial.
So, there exists a point $p_{0}$ in the image of $\s_{\l}$ that is not in $\pi(\cD)$.

\medskip

Let $x_{0} \in \Hn{2}$ such that $\pi(x_{0}) = p_{0}$ and denote by $\bar{\s}_{\l}$ the unique lift
of $\s_{\l}$ to $\Hn{2}$ starting at $x_{0}$. 

\smallskip 

If $y_{0} \in \pi^{-1}(p_{0})$ is the end point of $\bar{\s}_{\l}$, 
then we have $x_{0} , y_{0} \in \Hn{2} \setmin \C(\cD)$, 
and therefore
\begin{equation} \label{equ:distance}
   D_{F_{\l}}(x_{0} , y_{0}) = D_{F_{0}}(x_{0} , y_{0})
\end{equation}
by Lemma~\ref{lem:distance}.

Now, we have 
\begin{equation} \label{equ:length}
   L_{F_{\l}}(\bar{\s}_{\l}) = L_{\cF_{\l}}(\s_{\l})
\end{equation}
since $\pi : (\Hn{2} , F_{\l}) \to (S , \cF_{\l})$ is a local isometry. 
But this implies that $\bar{\s}_{\l}$ is a distance minimizing geodesic 
connecting $x_{0}$ to $y_{0}$ in $(\Hn{2} , F_{\l})$
because if this were not the case, there would be a $F_{\l}$-geodesic $[x_{0} , y_{0}]_{F_{\l}}$ such that
$L_{F_{\l}}([x_{0} , y_{0}]_{F_{\l}}) < L_{F_{\l}}(\bar{\s}_{\l})$ and hence
$L_{F_{\l}}([x_{0} , y_{0}]_{F_{\l}}) < L_{\cF_{\l}}(\s_{\l})$ by Equation~\ref{equ:length}.
Therefore, we would get $L_{\cF_{\l}}(\pi \circ [x_{0} , y_{0}]_{F_{\l}}) < L_{\cF_{\l}}(\s_{\l})$,
which is not possible since $\pi \circ [x_{0} , y_{0}]_{F_{\l}}$ is a closed curve that belongs to $\S$ 
(indeed, as $[x_{0} , y_{0}]_{F_{\l}}$ is homotopic to $\bar{\s}_{\l}$ 
in the simply connected space $\Hn{2}$ 
with fixed ends $x_{0}$ and $y_{0}$, the closed curve $\pi \circ [x_{0} , y_{0}]_{F_{\l}}$ 
is homotopic to $\s_{\l}$ in $S$
with fixed base point $p_{0}$).

So, we have $L_{F_{\l}}(\bar{\s}_{\l}) = D_{F_{\l}}(x_{0} , y_{0})$, 
and hence
\begin{equation} \label{equ:length-dist}
   L_{\cF_{\l}}(\s_{\l}) = D_{F_{0}}(x_{0} , y_{0})
\end{equation}
by Equations~\ref{equ:distance} and \ref{equ:length}.

\medskip 

Next, if $\bar{\s}_{0} : [0 , 1] \to \Hn{2}$ is a $D_{F_{0}}$-distance minimizing geodesic 
connecting $x_{0}$ to $y_{0}$,
the curve $\s_{0} := \pi \circ \bar{\s}_{0}$ is a closed curve that belongs to $\S$ 
(same reasoning as above for $\pi \circ [x_{0} , y_{0}]_{F_{\l}}$) 
with $L_{F_{0}}(\bar{\s_{0}}) = L_{\cF_{0}}(\s_{0})$ 
since $\pi : (\Hn{2} , F_{0}) \to (S , \cF_{0})$ is a local isometry. 
Thus, we get 
\begin{equation} \label{equ:comparison-1}
   L_{\cF_{\l}}(\s_{\l}) = L_{\cF_{0}}(\s_{0})
\end{equation}
by Equation~\ref{equ:length-dist}. 

\medskip 

On the other hand, starting with is a closed curve $\tau_{0} : [0 , 1] \to S$ 
of shortest $\cF_{0}$-length within $\S$
and using exactly the same steps as previously (reversing the roles of $\cF_{\l}$ and $\cF_{0}$), 
there exists a closed curve $\tau_{\l} : [0 , 1] \to S$ that belongs to $\S$ and satisfies 
\begin{equation} \label{equ:comparison-2}
   L_{\cF_{0}}(\tau_{0}) = L_{\cF_{\l}}(\tau_{\l}).
\end{equation}

\medskip 

Finally, Equations~\ref{equ:comparison-1} and~\ref{equ:comparison-2} yield  
$$
L_{\cF_{\l}}(\s_{\l}) \leq L_{\cF_{\l}}(\tau_{\l}) 
= 
L_{\cF_{0}}(\tau_{0}) \leq L_{\cF_{0}}(\s_{0}) = L_{\cF_{\l}}(\s_{\l}),
$$
and therefore $L_{\cF_{\l}}(\s_{\l}) = L_{\cF_{0}}(\tau_{0})$. 

\smallskip 

This proves Corollary~\ref{cor:spectrum}.
\end{proof}

\medskip

\begin{remark*}
It is to be noticed here that the proof of Corollary~\ref{cor:spectrum} shows that 
for each non-trivial free homotopy class $\S$ for $S$ and each $\l \in (-\l_{0} , \l_{0})$, 
there is a \emph{unique} (up to reparameterization) closed curve $\Tn{1} := \RR / \ZZ \to S$ 
in $\S$ of shortest $\cF_{\l}$-length. 

\medskip

\begin{proof}~\\ 
\indent Let $\s : \Tn{1} \to S$ and $\tau : \Tn{1} \to S$ be closed curves 
in $\S$ of (the same) shortest $\cF_{\l}$-length (thus $\cF_{\l}$-geodesics) and prove they are equal
up to a translation in $\Tn{1}$.

\smallskip

If there were $p_{0} \in \s(\Tn{1}) \setmin \pi(\cD)$ and $p_{1} \in \tau(\Tn{1}) \setmin \pi(\cD)$ 
with $p_{0} \notin \tau(\Tn{1})$ and $p_{1} \notin \s(\Tn{1})$,
then the same reasoning as in the proof of Corollary~\ref{cor:spectrum}
would lead to the existence of closed curves $\s_{0} : \Tn{1} \to S$ and $\tau_{0} : \Tn{1} \to S$
in $\S$ of shortest $\cF_{0}$-length with $p_{0} \in \s_{0}(\Tn{1})$ and $p_{1} \in \tau_{0}(\Tn{1})$.
As $p_{0} \neq p_{1}$, we would get $\s_{0} \neq \tau_{0}$, which is not possible
since it is well known there is a unique (up to reparameterization) closed curve
$\Tn{1} \to S$ in $\S$ of shortest $\cF_{0}$-length since $\cF_{0}$ is Riemannian hyperbolic.

\smallskip

So, we have $\s(\Tn{1}) \setmin \pi(\cD) \inc \tau(\Tn{1})$ 
or $\tau(\Tn{1}) \setmin \pi(\cD) \inc \s(\Tn{1})$, 
and this implies there exist $t_{0} \in \Tn{1}$ and a neighborhood $\cV$ of $0$ in $\Tn{1}$ such that
$\s(t) = \tau(t + t_{0})$ for all $t \in \cV$.
Thus $\s(t) = \tau(t + t_{0})$ for all $t \in \Tn{1}$ since $\s$ 
and $\Tn{1} \ni t \longmapsto \tau(t + t_{0}) \in S$ are $\cF_{\l}$-geodesics. 
\end{proof} 
\end{remark*}

\bigskip

Last, we discuss the volume growth entropy, considered with respect to the
Holmes-Thompson volume, which we now define. Let $(M , \cF)$ be an $n$-dimensional
Finsler manifold. For each $p \in M$, let $B_{p}^{\cF} \! M := \{ v \in T_{p}M \st \cF(p , v) < 1 \}$
be the unit open ball in $T_{p}M$ for the norm $\cF(p , \cdot)$,
and $\left( B_{p}^{\cF} \! M \right)^{\! \circ}$
its dual set in $T_{p}^{*}M$ 
(recall that for any set $X$ in a finite dimensional vector space $V$, 
we have $X^{\circ} := \{ \f \in V^{*} \st \f(v) \leq 1 \ \mbox{for all} \ v \in X \} \inc V^{*}$).
Then define the unit $\cF$-ball co-tangent bundle 
\begin{eqnarray*}
   \left( B^{\cF} \! M \right)^{\! \circ} 
   & := & 
   \bigcup_{p \in M} \{ p \} \! \times \! \left( B_{p}^{\cF} \! M \right)^{\! \circ} \\
   & = & 
   \left\{ (p , \f) \in T^{*}M \st \f(v) \leq 1 
   \ \ \mbox{for all} \ \ v \in B_{p}^{\cF} \! M \right\} \inc T^{*}M.
\end{eqnarray*}
\indent Let $\om$ be the canonical symplectic form on $T^{*}M$ given by $\om := \ed \a$,
where $\a$ is the Liouville $1$-form on $T^{*}M$, and 
$$
\Om := \frac{1}{n!} \underbrace{\om \wedge \cdots \wedge \om}_{n \ \mbox{\tiny times}}
$$ 
be the canonical volume form on $T^{*}M$.

\medskip 

For any Borel subset $A \inc M$, define the Holmes-Thompson volume of $A$ by
$$
\Vol{\cF}{A} := \frac{1}{C_{n}} \oldint_{\rest{\left( B^{\cF} \! M \right)^{\! \circ}}{A}} \!\!\! \Om,
$$ 
where 
$\rest{\left( B^{\cF} \! M \right)^{\! \circ}}{A} 
:= 
\left\{ (p , \f) \in \left( B^{\cF} \! M \right)^{\! \circ} \st p \in A \right\} \inc T^{*}M$
and $C_{n}$ is the volume of the unit open ball in $n$-dimensional Euclidean space.

\medskip 

The Holmes-Thompson volume generalizes the Riemannian volume in the sense that if
$\cF$ is the Finsler metric associated with a Riemannian metric $g$ on $M$, 
then $\mathrm{Vol}_{\cF} = \mathrm{vol}_{g}$.
Note that for a Finsler manifold there is another choice of volume that generalizes Riemannian volume 
called the Busemann volume which corresponds to the Hausdorff measure (see \cite{BBI01}, page 192). 
It is to be mentionned that partial results concerning the entropy rigidity question have been
obtained using the Busemann volume (see~\cite{Vero99} and~\cite{BolNew01}). 
Moreover, to get a taste of the difference between these two notions of volume in Finsler geometry, 
one may have a look at \cite{AlvBer05}. 

\medskip

\begin{remark}
~
\begin{enumerate}
   \item If $\cF^{*} : T^{*}M \to \RR$ is the dual Finsler metric of $\cF$ defined by
   $$
   \cF^{*}(p , \f) := \max{\{ \f(v) \st v \in T_{p}M \ \ \mbox{and} \ \ \cF(p , v) = 1 \}},
   $$
   then, for each $p \in M$, we have 
   $$
   \left( B_{p}^{\cF} \! M \right)^{\! \circ} 
   = 
   B_{p}^{\cF^{*} \!} \! M := \{ \f \in T_{p}^{*}M \st \cF^{*}(p , \f) < 1 \} \inc T_{p}^{*}M. 
   $$

   \medskip 

   \item Given any Riemannian metric $g$ on $M$, we have the formula 
   $$
   \Vol{\cF}{A} 
   = 
   \frac{1}{C_{n}} \int{A}{}{\ \vol{g^{*}(p) \!}
   {\! \left( B_{p}^{\cF} \! M \right)^{\! \circ}} \!\!}{\vol{g \!}{p}}
   $$
   for any Borel subset $A \inc M$, where $g^{*}(p)$ is the dual scalar product of $g(p)$ on $T_{p}^{*}M$
   and $\mathrm{vol}_{g^{*}(p)}$ its associated Haar measure (see \cite{BurIva02}). 

   \medskip 

   \item In case the Finsler metric $\cF$ is \emph{smooth} and \emph{strongly convex}, 
   let $\Om_{\cF}$ be the symplectic volume form on $TM \setmin \{ 0 \}$ associated with $\cF$ 
   defined as the pullback of the canonical volume form on $T^{*}M$ by the Legendre transform 
   $TM  \setmin \{ 0 \} \to T^{*}M$ induced by $\cF$
   (this map is a local diffeomorphism since $\cF$ is strongly convex). \\
   Then we can write
   $$
   \Vol{\cF}{A} 
   = 
   \frac{1}{C_{n}} \oldint_{(TM  \backslash \{ 0 \}) \cap \rest{B^{\cF} \! M}{A}} \!\!\! \Om_{\cF},
   $$  
   where $\rest{B^{\cF} \! M}{A} 
   := \{ (p , v) \in TM \st p \in A \ \mbox{and} \ \cF(p , v) < 1 \} \inc TM$.    
\end{enumerate}
\end{remark}

\bigskip

We shall now use the Main Theorem together with the following key result by Ivanov 
to prove Corollary~\ref{cor:entropy}. 

\begin{theorem}[Ivanov, \cite{Iva01}] \label{thm:Ivanov}
   Let $\D$ be an open Euclidean disk in $\Rn{2}$, 
   and consider \emph{smooth strongly convex} Finsler metrics $\gF_{0}$ and $\gF$ on $\Rn{2}$. 
   Assume that \emph{every} two points in $\D$ can be joined \emph{within} $\bar{\D}$ 
   by a \emph{unique} (up to reparametrization) geodesic of $\gF_{0}$ and by a geodesic of $\gF$. 
   Then, if $d_{\gF}(x , y) \geq d_{\gF_{0} \!}(x , y)$
   for all $x , y \in \partial \D$, we have $\Vol{\gF}{\D} \geq \Vol{\gF_{0} \!}{\D}$, 
   where $d_{\gF_{0}}$ and $d_{\gF}$ are the metrics induced on $\partial \D$ 
   by the distance functions on $\Rn{2}$ associated respectively with $\gF_{0}$ and $\gF$. 
\end{theorem}

\medskip 

From this, we get

\begin{consequence} \label{csq:Ivanov}
   Let $\gF$ be a smooth strongly convex Finsler metric on $\Hn{2}$ without conjugate points and 
   such that every two points in $\cD$ can be joined by a geodesic of $\gF$ 
   whose image is contained in $\cD$. 
   Then, if $d_{\gF} = d_{F_{0}}$, we have $\Vol{\gF}{\cD} = \Vol{F_{0} \!}{\cD}$. 
\end{consequence}

\medskip

\begin{proof}~\\ 
\indent First of all, the hyperbolic Finsler metric $F_{0}$ is smooth and strongly convex, 
and every two points in $\Hn{2}$ can be joined by a unique (up to reparametrization) geodesic of $F_{0}$. 
Therefore, since $\cD$ is an open ball in $(\Hn{2} , F_{0})$, the unique (up to reparametrization) 
$F_{0}$-geodesic joining two points in $\cD$ has its image contained in $\cD$. 
Thus, $d_{\gF} \geq d_{F_{0}}$ yields $\Vol{\gF}{\cD} \geq \Vol{F_{0} \!}{\cD}$ 
by Theorem~\ref{thm:Ivanov}. 

\smallskip 

On the other hand, since the smooth strongly convex Finsler metric $\gF$ has no conjugate points, 
every two points in $\cD$ can \emph{actually} be joined within $\cD$ 
by a \emph{unique} (up to reparametrization) geodesic of $\gF$. 
So, using $d_{\gF} \leq d_{F_{0}}$, we have $\Vol{\gF}{\cD} \leq \Vol{F_{0} \!}{\cD}$ 
still by Theorem~\ref{thm:Ivanov}. 

\smallskip 

Conclusion: $\Vol{\gF}{\cD} = \Vol{F_{0} \!}{\cD}$. 
\end{proof}

\bigskip

Let us now prove the following two independent lemmas: 

\begin{lemma} \label{lem:inclusions}
   Let $\cF$ be a Finsler metric on $S$ such that 
   
   \begin{enumerate}
      \item $d_{F} = d_{F_{0}}$, 

      \smallskip 

      \item $F(x , u) = F_{0}(x , u)$ for all $x \in \Hn{2} \setmin \C(\cD)$ and $u \in T_{x}\Hn{2}$.  
   \end{enumerate}
   Let $c > 0$ be a constant such that $\mathrm{diam}_{F}(\cD) \leq c$ 
   and $\mathrm{diam}_{F_{0}}(\cD) \leq c$. 
   Then for all $x \in \Hn{2} \setmin \C(\cD)$ and $R > c$, we have 
   $$
   B_{F_{0}}(x , R - c) \inc B_{F}(x , R) \inc B_{F_{0}}(x , R + c).
   $$
\end{lemma} 

\medskip 

and 

\medskip 

\begin{lemma} \label{lem:volumes}
   Let $\cF$ be a Finsler metric on $S$ such that 
   
   \begin{enumerate}
      \item $\Vol{F}{\cD} = \Vol{F_{0} \!}{\cD}$, 

      \smallskip 

      \item $F(x , u) = F_{0}(x , u)$ for all $x \in \Hn{2} \setmin \C(\cD)$ and $u \in T_{x}\Hn{2}$.  
   \end{enumerate}
   Let $c > 0$ be a constant such that $\mathrm{diam}_{F}(\cD) \leq c$ 
   and $\mathrm{diam}_{F_{0}}(\cD) \leq c$. 
   Then for all $x \in \Hn{2} \setmin \C(\cD)$ and $R > 2 c$, we have 
   $$
   \Vol{F}{B_{F_{0}}(x , R - c)} \leq \Vol{F_{0} \!}{B_{F_{0}}(x , R)} \leq \Vol{F}{B_{F_{0}}(x , R + c)}.
   $$   
\end{lemma}

\bigskip

\begin{proof}[Proof of Lemma~\ref{lem:inclusions}]~\\ 
\indent Let $x \in \Hn{2} \setmin \C(\cD)$ and $y \in \Hn{2}$. 

\smallskip 

We will show
$$
D_{F_{0}}(x , y) - c \leq D_{F}(x , y) \leq D_{F_{0}}(x , y) + c,
$$
which immediately implies the result.

\medskip 

Let us fix distance minimizing geodesics $[x , y]_{F}$ and $[x , y]_{F_{0}}$ 
for $F$ and $F_{0}$ respectively connecting $x$ to $y$.

\smallskip 

By Lemma~\ref{lem:distance}, we know that if $y \in \Hn{2} \setmin \C(\cD)$, 
then $D_{F}(x , y) = D_{F_{0}}(x , y)$, and the inequalities above hold. 

\medskip 

So, suppose $y$ is in a connected component $C$ of $\C(\cD)$ and let $e_{F_{0}}$ and $e_{F}$ be the
points at which respectively $[x , y]_{F_{0}}$ and $[x , y]_{F}$ enter
$C$ the first time ($e_{F_{0}}$ and $e_{F}$ lie in $\bC$, see Figure~\ref{fig:lem-inclusions}).

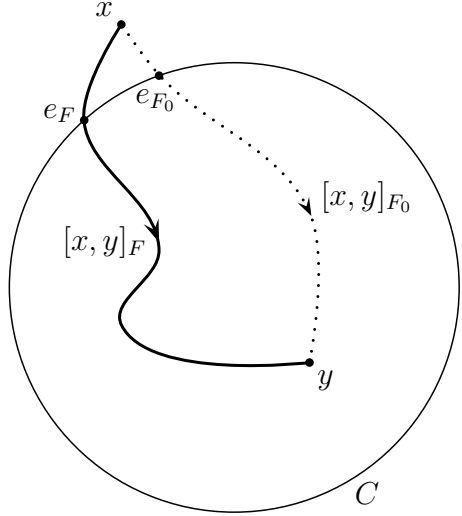
\begin{figure}[h] 

\begin{pspicture}(6,7.2) 
   \pscircle[linewidth=0.02](3,3){3} 
   \uput{0.1}[0](4.5,0.264){$C$} 
   \psdots[dotstyle=*,dotscale=0.9](1.5,6.5)(4,2)(1.005,5.226)(2,5.818) 
   \uput{0.15}[135](1.5,6.5){$x$} 
   \uput{0.15}[315](4,2){$y$} 
   \uput{0.1}[160](1,5.236){$e_{F}$} 
   \uput{0.2}[265](2,5.828){$e_{F_{0}}$} 
   \pscurve[linewidth=0.04](1.5,6.5)(1,5.236)(2,3.5)(1.5,2.5)(4,2) 
   \psline[linewidth=0.08]{c->}(1.943,3.75)(2,3.6) 
   \uput{0.15}[180](1.98,3.6){$[x , y]_{F}$} 
   \pscurve[linewidth=0.04,linestyle=dotted](1.5,6.5)(2,5.828)(4,4)(4,2) 
   \psline[linewidth=0.06]{c->}(3.975,4.05)(4.025,3.95) 
   \uput{0.15}[30](4,4){$[x , y]_{F_{0}}$} 
\end{pspicture}  
   \caption{\label{fig:lem-inclusions} Proof of Lemma~\ref{lem:inclusions}} 
\end{figure}

\smallskip 

As $\mathrm{diam}_{F_{0}}(\cD) \leq c$ and $\mathrm{diam}_{F}(\cD) \leq c$, 
we have $D_{F_{0}}(e_{F_{0}} , y) \leq c$ and $D_{F}(e_{F} , y) \leq c$. 
Furthermore, by Lemma~\ref{lem:distance} and since $x$, $e_{F_{0}}$ and $e_{F}$ 
all lie outside $\C(\cD)$, we have 
$D_{F}(x , e_{F}) = D_{F_{0}}(x , e_{F})$ and $D_{F}(x , e_{F_{0}}) = D_{F_{0}}(x , e_{F_{0}})$. 
Thus
$$
D_{F}(x , y) \leq D_{F}(x , e_{F}) + D_{F}(e_{F} , y) \leq D_{F_{0}}(x , e_{F}) + c
$$
and
$$ 
D_{F_{0}}(x , y) \leq D_{F_{0}}(x , e_{F_{0}}) + D_{F_{0}}(e_{F_{0}} , y) \leq D_{F}(x , e_{F_{0}}) + c.
$$
\indent We conclude that
$$
D_{F_{0}}(x , y) - c \leq D_{F}(x , y) \leq D_{F_{0}}(x , y) + c,
$$
as claimed.
\end{proof}

\bigskip

\begin{proof}[Proof of Lemma~\ref{lem:volumes}]~\\ 
\indent Let $x \in \Hn{2}$ and $R > 2 c$. 

\smallskip 

By assumption~(1), we have $\Vol{F}{\c(\cD)} = \Vol{F_{0} \!}{\c(\cD)}$ 
for each $\c \in \C$ since $\C$ is a group of isometries for both $F$ and $F_{0}$.

\medskip 

On the other hand, by assumption~(2),
Borel sets in $\Hn{2}$ not intersecting $\C(\cD)$ have the same Holmes-Thompson volume 
with respect to $F_{0}$ as with respect to $F$ 
(the boundary of $\C(\cD)$ being a set of zero measure for both Holmes-Thompson volumes). 

\smallskip 

Let $\cU$ be the union of the connected components of $\C(\cD)$ 
that intersect $\partial B_{F_{0}}(x , R)$. 
Then
\begin{align*}
   \Vol{F}{B_{F_0}(x , R) \setmin \cU} & = \Vol{F_{0} \!}{B_{F_{0}}(x , R) \setmin \cU} \\
   \intertext{and}
   \Vol{F}{B_{F_{0}}(x , R) \cup \cU} & = \Vol{F_{0} \!}{B_{F_{0}}(x , R) \cup \cU}.
\end{align*}

\smallskip 

Since for each connected component $\c(\cD)$ ($\c \in \C$) of $\C(\cD)$ we have
$\mathrm{diam}_{F_{0}}(\c(\cD)) \leq c$ and $\mathrm{diam}_{F}(\c(\cD)) \leq c$, 
one gets
$$
B_{F_{0}}(x , R - c) 
\inc 
B_{F_{0}}(x , R) \setmin \cU \inc B_{F_{0}}(x , R)
\inc
B_{F_{0}}(x , R) \cup \cU
\inc 
B_{F_{0}}(x , R + c),
$$
and hence
$$
\Vol{F}{B_{F_{0}}(x , R - c)} \leq \Vol{F_{0} \!}{B_{F_{0}}(x , R)} \leq \Vol{F}{B_{F_{0}}(x , R + c)}.
$$
\end{proof}

\bigskip

Before proving Corollary~\ref{cor:entropy}, we need the following lemma about the volume growth entropy:  

\begin{lemma} \label{lem:entropy}
   Let $\cF$ be a Finsler metric on $S$ such that 
   
   \begin{enumerate}
      \item $d_{F} = d_{F_{0}}$, 

      \smallskip 

      \item $\Vol{F}{\cD} = \Vol{F_{0} \!}{\cD}$, and 

      \smallskip 

      \item $F(x , u) = F_{0}(x , u)$ for all $x \in \Hn{2} \setmin \C(\cD)$ and $u \in T_{x}\Hn{2}$. 
   \end{enumerate}
   Then $h(F) = h(F_{0})$. 
\end{lemma}

\medskip

\begin{proof}~\\ 
\indent Choose $x \in \Hn{2} \setmin \C(\cD)$ and $R > 2 c$. 

\smallskip 

By Lemma~\ref{lem:inclusions}, we have
$$
B_{F_{0}}(x , R - c) \inc B_{F}(x , R) \inc B_{F_{0}}(x  ,R + c).
$$
Therefore
$$
\Vol{F}{B_{F_{0}}(x , R - c)} \leq \Vol{F}{B_{F}(x , R)} \leq \Vol{F}{B_{F_{0}}(x , R + c)},
$$
and hence, by Lemma~\ref{lem:volumes},
$$
\Vol{F_{0} \!}{B_{F_{0}}(x , R - 2 c)} 
\leq 
\Vol{F}{B_{F}(x , R)} 
\leq 
\Vol{F_{0} \!}{B_{F_{0}}(x , R + 2 c)}.
$$

\smallskip 

But
$$
h(F_{0}) = \lim_{R \goes +\infty} \frac{1}{R} \log{(\Vol{F_{0} \!}{B_{F_{0}}(x , R - 2 c)})}
= 
\lim_{R \goes +\infty} \frac{1}{R} \log{(\Vol{F_{0} \!}{B_{F_{0}}(x , R + 2 c)})}.
$$
Thus $h(F) = h(F_{0})$, as desired.
\end{proof}

\bigskip

We are now able to prove Corollary~\ref{cor:entropy}.

\medskip

\begin{proof}[Proof of Corollary~\ref{cor:entropy}]~\\ 
\indent Let $\e > 0$ and $\iP : (-\e , \e) \times TS \to \RR$ as given by Lemma~\ref{lem:quotient}. 

\smallskip 

Fixing $\l \in (-\e , \e)$, the smooth strongly convex Finsler metric $F_{\l}$ on $\Hn{2}$ 
satisfies points~(4) and~(6) of Lemma~\ref{lem:quotient}, 
and thus the hypotheses of Consequence~\ref{csq:Ivanov}. 
Furthermore, as it satisfies point~(3) in Lemma~\ref{lem:quotient}, we then get 
$\Vol{F_{\l} \!}{\cD} = \Vol{F_{0} \!}{\cD}$ by Consequence~\ref{csq:Ivanov}. 

\smallskip 

So, by point~(5) in Lemma~\ref{lem:quotient}, $F_{\l}$ satisfies all the three hypotheses 
of Lemma~\ref{lem:entropy}. 
Therefore, according to this latter lemma, $h(F_{\l}) = h(F_{0})$, 
or equivalently $h(\cF_{\l}) = h(\cF_{0})$. 

\smallskip 

Now, let $U$ be an open set in $\Hn{2}$ such that $\bar{\cD} \inc U$ 
and $\rest{\pi}{U}$ is injective as in the proof of Lemma~\ref{lem:quotient}. 
Then $\pi$ induces isometries from $(U , F_{\l})$ onto $(\pi(U) , \cF_{\l})$ 
and from $(U , F_{0})$ onto $(\pi(U) , \cF_{0})$, 
which yield $\Vol{F_{\l} \!}{\cD} = \Vol{\cF_{\l} \!}{\pi(\cD)}$ and 
$\Vol{F_{0} \!}{\cD} = \Vol{\cF_{0} \!}{\pi(\cD)}$. 

Hence $\Vol{\cF_{\l} \!}{\pi(\cD)} = \Vol{\cF_{0} \!}{\pi(\cD)}$. 

\smallskip 

On the other hand, point~(5) in Lemma~\ref{lem:quotient} implies that 
$\cF_{0}(p , v) = \cF_{0}(p , v)$ for all $p \in S \setmin \pi(\cD)$ and $v \in T_{p}S$. 
Thus $\Vol{\cF_{\l} \!}{S \setmin \pi(\cD)} = \Vol{\cF_{0} \!}{S \setmin \pi(\cD)}$, and finally 
$\Vol{\cF_{\l} \!}{S} = \Vol{\cF_{0} \!}{S}$. 

Conclusion: $h(\cF_{\l})^{\! 2} \, \Vol{\cF_{\l} \!}{S} = h(\cF_{0})^{\! 2} \, \Vol{\cF_{0} \!}{S}$. 

\smallskip 

This finishes the proof of Corollary~\ref{cor:entropy}. 
\end{proof}

\bigskip
\bigskip
\bigskip 


\section{Proof of the Main Theorem} \label{sec:main-thm}

Throughout all this section, 
denote by $\scal{\cdot}{\cdot}$ the usual scalar product in $\Rn{2}$ and
$|\cdot|$ its associated norm. 
Let $\Hn{2} := \{ p \in \Rn{2} \st |p| < 1 \} \inc \Rn{2}$
endowed with the Klein metric $g_{0}$ that is given by
\begin{equation} \label{equ:Klein} 
   g_{0}(p) \act (v , w) 
   := 
   \frac{\scal{v}{w}}{1 - |p|^{2}} + \frac{\scal{p}{v} \scal{p}{w}}{(1 - |p|^{2})^{2}}
\end{equation}
for all $p \in \Hn{2}$ and $v , w \in T_{p}\Hn{2} = \Rn{2}$.

\smallskip 

Thus $(\Hn{2} , g_{0} )$ is a model of
the hyperbolic plane where images of the geodesics are \emph{affine segments}. 

\smallskip 

For each $r \in (0 , 1]$, let 
$$
\cD(r) := \{ p \in \Rn{2} \st |p| < r \} \inc \Hn{2}.
$$

\smallskip 

Finally, fix an arbitrary $R \in (0 , 1)$, let $\cD := \cD(R)$, 
and denote by $d_{g_{0}}$ the induced metric on $\bD$ by the distance function on $\Hn{2}$ 
associated with $g_{0}$. 

\bigskip
\bigskip 

\subsection{Arcostanzo's construction} \label{ssec:Arcostanzo} 
~

\medskip 

In~\cite{Arc94}, Arcostanzo gives conditions on a distance $d$ on $\bD$ and a set
$\gS$ of parameterized curves $\c : [0 , 1] \to \bar{\cD}$ in such a way that there exists a
Finsler metric $\sF$ on $\cD$ whose associated distance on $\cD$ extends to a distance on $\bar{\cD}$ 
that induces the metric $d$ on $\bD$ and such that $\{ \rest{\c}{(0 , 1)} \st \c \in \gS \}$ 
coincides with the set of maximal geodesics of $\cF$ after reparametrization by $(0 , 1)$. 
We will state this result precisely in the specific case when the
distance on $\bD$ is $d_{g_{0}}$, though more general results are established in~\cite{Arc94}.

\medskip 

We begin by giving Arcostanzo's conditions on a set of parameterized curves.

\begin{definition} \label{def:admissible}
   A set $\gS$ of parameterized curves $\c : [0 , 1] \to \Rn{2}$ is said to be admissible for $\cD$ 
   if and only if the following properties hold:

   \begin{enumerate}
      \item each $\c \in \gS$ is $\Cl{\infty}$, regular, injective, and satisfies $\c((0 , 1)) \inc \cD$; 

      \smallskip 

      \item for each $\c \in \gS$, we have $\c(0) , \c(1) \in \bD$; 

      \smallskip 

      \item for any $p , q \in \bar{\cD}$ with $p \neq q$, 
      there exists a unique $(\c , t_{0} , t_{1}) \in \gS \times [0 , 1] \times [0 , 1]$
      such that $p = \c(t_{0})$ and $q = \c(t_{1})$ with $t_{0} < t_{1}$; 

      \smallskip 

      \item for any $p \in \cD$ and $v \in T_{p}\cD = \Rn{2}$ with $v \neq 0$, 
      there exists a unique $(\c , t) \in \gS \times (0 , 1)$ 
      such that $p = \c(t)$ and $\c'(t)$ is parallel to $v$ with the same direction.
   \end{enumerate}
\end{definition}

\medskip 

For Arcostanzo's construction to yield a Finsler metric and not just a distance
on $\cD$, a certain amount of regularity is required about the way the end points
$\c(0) , \c(1) \in \bD$ depend on the parameterized curve $\c \in \gS$. 

\smallskip 

More precisely, given $\gS$ an admissible set of parameterized curves for $\cD$,
for each $x \in \bD$ and $p \in \cD$, there is a unique $\c \in \gS$ 
such that $x = \c(0)$ and $p \in \c([0 , 1])$ according to point~(3) in Definition~\ref{def:admissible}.
Setting $\s(x , p) := \c(1)$, we then get a map $\s : \bD \times \cD \to \bD$
we will call the `end point map' associated with $\gS$ (see Figure~\ref{fig:end-point-1}).

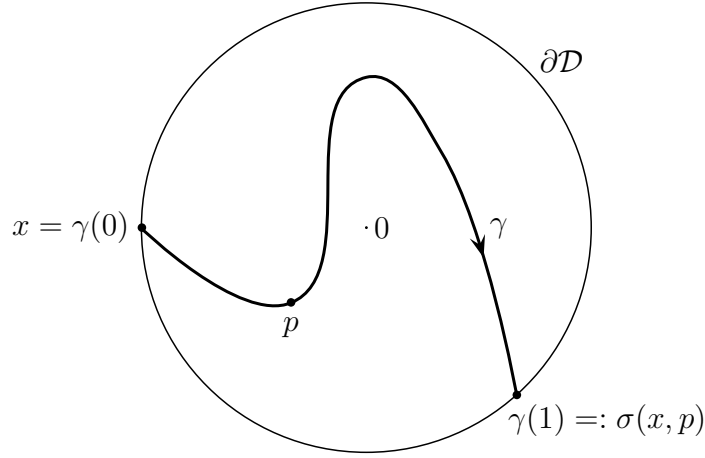
\begin{figure}[h] 

\begin{pspicture}(8,6.5) 
   \psdots[dotstyle=*,dotscale=0.3](4,3) 
   \uput{0.1}[0](4,3){$0$} 
   \pscircle[linewidth=0.02](4,3){3} 
   \uput{0.1}[45](6.236,5){$\bD$} 
   \psdots[dotstyle=*,dotscale=0.9](1.015,3)(3,2.005)(6,0.775) 
   \uput{0.15}[180](1,3){$x = \c(0)$} 
   \uput{0.2}[270](3,2){$p$} 
   \uput{0.12}[310](6,0.764){$\c(1) =: \s(x , p)$} 
   \pscurve[linewidth=0.04](1,3)(3,2)(4,5)(5,4)(6,0.764) 
   \psline[linewidth=0.08]{c->}(5.5153,2.75)(5.559,2.6) 
   \uput{0.2}[50](5.525,2.7){$\c$} 
\end{pspicture}  
   \caption{\label{fig:end-point-1} The `end point map' $\s$} 
\end{figure}

\medskip 

\begin{remark*}
For any $x , y \in \bD$ and $p \in \cD$, we obviously have: $\s(x , p) = y \iff \s(y , p) = x$.
\end{remark*}

\medskip 

\begin{definition} \label{def:property-C}
   An admissible set of parameterized curves $\gS$ for $\cD$ satisfies Arcostanzo's Property~(C) 
   if and only if 
   
   \begin{enumerate}
      \item the associated `end point map' $\s : \bD \times \cD \to \bD$ is $\Cl{\infty}$, and 

      \smallskip 

      \item for every $(x , p) \in \bD \times \cD$ and every $v \in T_{p}\cD = \Rn{2}$ with $v \neq 0$,
      we have the equivalence
      $$
      \pder{\s}{p}(x , p) \act v = 0 
      \iff 
      v \ \mbox{and} \ \c'(t) \ \mbox{are parallel vectors},
      $$
      where $(\c , t)$ is the unique element in $\gS \times (0 , 1)$ 
      such that $x = \c(0)$ and $p = \c(t)$ according to point~(3) in Definition~\ref{def:admissible}
      (with $t_{0} = 0$ and $t_{1} = t$).
   \end{enumerate}
\end{definition}

\medskip 

\begin{remark} \label{rem:property-C}
Point~(2) in Definition~\ref{def:property-C} can be reformulated in another way: 

\smallskip 

For every $(x , p) \in \bD \times \cD$ and every $v \in T_{p}\cD = \Rn{2}$ with $v \neq 0$,
we have the equivalence
$$
\pder{\s}{p}(x , p) \act v = 0 
\iff 
(x = e^{-}(p , v) \ \ \mbox{or} \ \ x = e^{+}(p , v)),
$$
where $e^{-}(p , v) := \c(0)$ and $e^{+}(p , v) := \c(1)$
if $\c$ denotes the unique parameterized curve in $\gS$ given by point~(4) 
in Definition~\ref{def:admissible}.
\end{remark}

\medskip 

Arcostanzo points out, for example, that the set of maximal geodesics in
$\bar{\cD}$ (reparameterized by $[0 , 1]$) of a negatively curved Riemannian metric on an open
neighborhood of $\bar{\cD}$ is admissible for $\cD$ 
and satisfies Property~(C) (see~\cite{Arc94}, page~242).

\bigskip

We can now state Arcostanzo's result:

\begin{theorem} [\cite{Arc94}, Th\'{e}or\`{e}me~2, page~243] \label{thm:Arcostanzo}
   Let $\gS$ be an admissible set of parameterized curves for $\cD$ that satisfies Property~(C).
   Then there exists a \emph{unique} Finsler metric $\sF$ on $\cD$ 
   whose associated distance function on $\cD$ 
   extends to a distance on $\bar{\cD}$ that induces the metric $d_{g_{0}}$ on $\bD$ 
   and such that $\{ \rest{\c}{(0 , 1)} \st \c \in \gS \}$ 
   coincides with the set of maximal geodesics of $\sF$ after reparametrization by $(0 , 1)$. 
   Its precise formula is given by 
   $$
   \sF(p , v) 
   = 
   \frac{1}{4} \! \int{\bD}{}{\ \ppder{d_{g_{0}}}{x}{y}(x , \s(x , p)) \! 
   \left| \pder{\s}{p}(x , p) \act v \right| \!\!}{x}
   $$
   for any $(p , v) \in T\cD = \cD \times \Rn{2}$, 
   where $|\cdot|$ stands for the canonical Euclidean norm on $\Rn{2}$ 
   and $\mathrm{d} x$ denotes the canonical measure on the Euclidean circle $\bD = R \Sn{1}$. \\
   Moreover, this Finsler metric is smooth.
\end{theorem}

\medskip 

\begin{remark} \label{rem:Arcostanzo}
~
\begin{enumerate}
   \item Given $x_{0} = R e^{i t_{0}}$ and $y_{0} = R e^{i s_{0}}$ in $\bD = R \Sn{1}$ 
   such that $x_{0} \neq y_{0}$,
   the partial derivative $\disp \ppder{d_{g_{0}}}{x}{y}(x_{0} , y_{0})$
   is defined as to be equal to $\disp \ppder{f}{t}{s}(t_{0} , s_{0}) \in \RR$,
   where $f(t , s) = d_{g_{0}}(R e^{i t} , R e^{i s})$ for all $t , s \in \RR$.

   \medskip

   \item Arcostanzo's result applies here since $g_{0}$ is the hyperbolic metric and an easy computation 
   then shows that $\disp \ppder{d_{g_{0}}}{x}{y}(x , y) > 0$ for any $x , y \in \bD$ with $x \neq y$.

   \medskip

   \item By uniqueness of $\sF$ in Theorem~\ref{thm:Arcostanzo}, if we choose $\gS$ to be the set 
   of maximal geodesics of $g_{0}$ in $\bar{\cD}$ after reparametrization by $[0 , 1]$ 
   (whose images are the chords of the Euclidean circle $\bD$), 
   then we get that $\sF$ equals the restriction to $T\cD$ of the Finsler metric $F_{0}$ on $\Hn{2}$ 
   associated with $g_{0}$. 
   
   \medskip
   
   \item If we choose $\gS$ to be the set of maximal geodesics in $\bar{\cD}$ 
   (reparameterized by $[0 , 1]$) of a negatively curved Riemannian metric 
   on an open neighborhood of $\bar{\cD}$ (the set $\gS$ is then admissible for $\cD$ 
   and satisfies Property~(C) as shown by Arcostanzo in~\cite{Arc94}, page~242), 
   the unique Finsler metric on $\cD$ we get by Theorem~\ref{thm:Arcostanzo} is \emph{not} Riemannian. 
   
   \medskip

   \item The existence of a unique Finsler metric $\sF$ on $\cD$ 
   given by the formula in Theorem~\ref{thm:Arcostanzo}
   still holds without the assumption that the admissible set of parameterized curves $\gS$ 
   for $\cD$ has Property~(C), but in that case $\sF$ is \emph{not necessarily} reversible nor smooth.
\end{enumerate}
\end{remark}

\medskip 

\begin{remark} \label{rem:smooth}
Although it is not written in~\cite{Arc94}, the fact that $\sF$ is a \emph{smooth} Finsler metric on $\cD$ 
can be proved as follows.

\medskip 

\begin{proof}~\\ 
\indent Consider the map $\Ups : \cD \times \Rn{2} \times \bD \to \Rn{2}$ 
defined by 
$$
\Ups((p , v) , x) 
:= 
\frac{1}{4} \ppder{d_{g_{0}}}{x}{y}(x , \s(x , p)) \pder{\s}{p}(x , p) \act v.
$$

\smallskip 

Since $d_{g_{0}}$ is $\Cl{\infty}$ on $(\bD \times \bD) \setmin \{ (x , x) \st x \in \bD \}$ and 
$\s$ is $\Cl{\infty}$ on $\bD \times \cD$ (point~(1) in Definition~\ref{def:property-C}) 
which satisfies $\s(x , p) \neq x$ for all $(x , p) \in \bD \times \cD$, the positive function 
$\disp (x , p) \mapsto \ppder{d_{g_{0}}}{x}{y}(x , \s(x , p))$ is $\Cl{\infty}$ on $\bD \times \cD$, 
and therefore $\Ups$ is $\Cl{\infty}$. 

\smallskip 

This implies in particular that $\Ups$ is continuous, 
thus the function $\sF : T\cD = \cD \times \Rn{2} \to \RR$ in Theorem~\ref{thm:Arcostanzo} 
given by $\disp \sF(p , v) = \int{\bD}{}{|\Ups((p , v) , x)| \!}{x}$ is well defined and continuous. 

\smallskip 

On the other hand, for any $((p , v) , x) \in \cD \times (\Rn{2} \setmin \{ 0 \}) \times \bD$, 
the vector $\disp \pder{\s}{p}(x , p) \act v \in \Rn{2}$ vanishes 
iff $x = e^{-}(p , v) \ \mbox{or} \ x = e^{+}(p , v)$ (point~(2) in Definition~\ref{def:property-C}). 
So, given $(p , v) \in \cD \times (\Rn{2} \setmin \{ 0 \})$, the differential 
$\disp \pder{|\Ups|}{(p , v)}((p , v) , x) \in \LL{}(\Rn{4} , \RR)$ 
exists for all $x \in \bD \setmin \{ e^{-}(p , v) , e^{+}(p , v) \}$ 
and writes 
\begin{eqnarray*} 
   \pder{|\Ups|}{(p , v)}((p , v) , x) \act (w , \x) 
   & = & 
   \pder{|\Ups|}{p}((p , v) , x) \act w + \pder{\Ups}{v}((p , v) , x) \act \x \\ 
   & = & 
   \frac{\scal{\pder{\Ups}{p}((p , v) , x) \act w}{\Ups((p , v) , x)} 
   + \scal{\Ups(p , \x , x)}{\Ups((p , v) , x)}}{|\Ups((p , v) , x)|}
\end{eqnarray*}
for every $w , \x \in \Rn{2}$ (notice that $\Ups$ is linear with respect to $v$). 

\smallskip 

Therefore, if we fix $\tau > 0$ and take $0 < |v| \leq \tau$ 
together with $|w| \leq 1$ and $|\x| \leq 1$, we get 
\begin{eqnarray*}
   \left| \pder{|\Ups|}{(p , v)}((p , v) , x) \act (w , \x) \right| 
   & \leq & 
   \left| \pder{\Ups}{p}((p , v) , x) \act w \right| + |\Ups((p , \x) , x)| \\ 
   & & 
   \quad (\mbox{by Cauchy-Schwarz inequality}) \\   
   & \leq & 
   \tau \opnorm{\pder{\Ups}{p}((p , \cdot) , x)} + \norm{\Ups((p , \cdot) , x)}, 
\end{eqnarray*}
where $\norm{\cdot}$ and $\opnorm{\cdot}$ are respectively the operator norms 
on $\LL{}(\Rn{2} , \Rn{2})$ and $\LL{}_{2}(\Rn{2} \times \Rn{2} , \Rn{2})$ 
(bilinear maps from $\Rn{2} \times \Rn{2}$ to $\Rn{2}$).

\smallskip 

Since $\Ups$ is $\Cl{\infty}$, the functions 
$\disp (x , p) \mapsto \norm{\Ups((p , \cdot) , x)}$ 
and $\disp (x , p) \mapsto \opnorm{\pder{\Ups}{p}((p , \cdot) , x)}$ 
are continuous on $\bD \times \cD$. 
So, given any $r \in (0 , R)$, the compactness of $\bD \times \bar{\cD(r)}$ implies  
that there exist positive constants $\L_{1}$ and $\L_{2}$ such that
$$
\norm{\Ups((p , \cdot) , x)} \leq \L_{1} 
\quad \mbox{and} \quad 
\opnorm{\pder{\Ups}{p}((p , \cdot) , x)} \leq \L_{2}
$$
for all $(x , p) \in \bD \times \cD(r)$. 

\smallskip 

Conclusion: for any $p \in \cD(r)$, $v \in \Rn{2}$ such that $0 < |v| \leq \tau$, 
$x \in \bD \setmin \{ e^{-}(p , v) , e^{+}(p , v) \}$ 
and $w , \x \in \Rn{2}$ with $|w| \leq 1$ and $|\x| \leq 1$, 
we have 
$$
\left| \pder{|\Ups|}{(p , v)}((p , v) , x) \act (w , \x) \right| 
\leq 
\tau \L_{1} + \L_{2}. 
$$

\smallskip 

Now, since $\{ e^{-}(p , v) , e^{+}(p , v) \}$ is a set of zero measure 
with respect to $\mathrm{d} x$, we obtain from Lebesgue's dominated convergence theorem 
(see for example~\cite{Die68}, page~123) 
that the Finsler metric $\sF$ on $\cD$ 
is differentiable on $\cD(r) \times \{ v \in \Rn{2} \st 0 < |v| < \tau \}$. 

\smallskip 

As this holds for arbitrary $r \in (0 , R)$ and $\tau > 0$, we eventually get that 
$F$ is differentiable on $\cD \times (\Rn{2} \setmin \{ 0 \}) = T\cD \setmin \{ 0 \}$ 
with 
$$
\disp \pder{\sF}{(p , v)}(p , v) = \int{\bD}{}{\ \pder{|\Ups|}{(p , v)}((p , v) , x) \!}{x}. 
$$

\smallskip 

Finally, using the same reasonning as above, one can show by induction 
that for every $n \in \NN$ the Finsler metric $\sF$ is $n$ times differentiable 
on $T\cD \setmin \{ 0 \}$ with 
$$
\disp \frac{\partial^{n} \sF}{\partial (p , v)^{n}}(p , v) 
= 
\int{\bD}{}{\ \frac{\partial^{n} |\Ups|}{\partial (p , v)^{n}}((p , v) , x) \!}{x}.
$$ 
This proves that $\sF$ is smooth. 
\end{proof} 
\end{remark}

\bigskip

Now, using Theorem~\ref{thm:Arcostanzo}, 
our aim is to construct a `good' family $\fam{\gS}{\l}{(-\e , \e)}$ of admissible sets 
of parameterized curves for $\cD$ satisfying Property~(C), from which we shall 
later be able to get a family $\fam{\sF}{\l}{(-\e , \e)}$ of Finsler metrics on $\Hn{2}$ 
as needed in the Main Theorem. 
But, as we already mentionned, the main difficulty will be 
to ensure these Finsler metrics be $\emph{smooth}$ on the \emph{whole} space $\Hn{2}$ 
(and not only on the disk $\cD$) and \emph{coincide} with the Riemannian hyperbolic metric $g_{0}$ 
\emph{outside} $\cD$. 
Given any $\l \in (-\e , \e)$ and according to Theorem~\ref{thm:Arcostanzo}, 
it seems reasonable to ask all the parameterized curves in $\gS_{\l}$ to coincide 
with the geodesics for $g_{0}$ in a neighborhood of $\bD$ (note that since $(\Hn{2} , g_{0})$ 
has been chosen to be the Klein model of the hyperbolic plane, the images of the $g_{0}$-geodesics 
are affine segments, thus very easy to be dealt with). However, we also want the Finsler metric 
$\sF_{\l}$ \emph{not} to be Riemannian, and this will be the case if we choose $\gS_{\l}$ 
to be the set of the parameterized curves obtained as a `barycenter' of the geodesics for $g_{0}$ 
and the geodesics for some `good' Riemannian metric $g_{\l}$ conformal to $g_{0}$. 

\medskip

The advantage in constructing a family $\fam{\gS}{\l}{(-\e , \e)}$ in this way 
is that all the Finsler metrics of the associated family $\fam{\sF}{\l}{(-\e , \e)}$ 
obtained by Arcostanzo's theorem will satisfy the desired properties listed in the Main Theorem, 
but this construction will have a cost. 
Indeed, proving that the set $\gS_{\l}$ of parameterized curves 
is \emph{admissible} for $\cD$ and \emph{has} Property~(C) is not easy 
and will be done at the expense of great effort. 
This is why we will have to make very technical considerations just in order 
to ensure admissibility and Property~(C) for the family $\fam{\gS}{\l}{(-\e , \e)}$. 

\bigskip
\bigskip 

\subsection{Constructing a family of admissible sets of parameterized curves}
~

\medskip 


\subsubsection{\textsf{\emph{The setting}}} 
~

\smallskip 

We will now construct a family of admissible sets of parameterized curves for $\cD$ by interpolating
between the maximal geodesics for the hyperbolic metric $g_{0}$ on $\Hn{2}$
and those for a nearby Riemannian metric of non-constant curvature 
that is conformal to $g_{0}$. 

\medskip 

More precisely, let $\D$ be the Laplacian for $g_{0}$ and fix  
a regular eigenfunction $\p : \bar{\cD} \to \RR$ of $\D$ on $\bar{\cD}$
associated with the first eigenvalue $a$ of $\D$ and satisfying 
the Dirichlet condition $\rest{\p}{\bD} \equiv 0$.
It is then well known that $a > 0$ and that $\p$ can be chosen to be positive on $\cD$.
Furthermore, as $g_{0}$ is invariant under the group $\O(\Rn{2})$ of linear Euclidean isometries 
(\emph{i.e.}, $A^{*} g_{0} = g_{0}$ for all $A \in \O(\Rn{2})$) thanks to Equation~\ref{equ:Klein}, 
we get that $\p$ is $\O(\Rn{2})$-invariant. 

\medskip 

Next, let $\t : \Hn{2} \to \RR$ be any $\Cl{\infty}$ function that is $\O(\Rn{2})$-invariant
and such that $\t \equiv 1$ on $\bar{\cD(R / 4)}$ and $\t \equiv 0$ on $\Hn{2} \setmin \cD(R / 2)$. 
The new function $f : \Hn{2} \to \RR$ defined by 
$$ 
f(p) =
\begin{cases}
   \p(p) \t(p) & \mbox{if} \ p \in \cD \\
   0 & \mbox{if} \ p \in \Hn{2} \setmin \cD
\end{cases}
$$
is thus $\Cl{\infty}$ and $\O(\Rn{2})$-invariant, 
together with $f \equiv 0$ on $\Hn{2} \setmin \cD(R / 2)$ and $\D f = a f$ on $\cD(R / 4)$. \\
In particular, since $a > 0$ and $\p$ is positive on $\cD$, 
there exists a number $\d_{0} > 0$ such that 
$(\D f)(p) \geq 1 / \d_{0}^{2}$ for all $p \in \cD(R / 4)$.

\medskip 

\begin{proposition} \label{prop:conformal}
   The function
   $$
   \map{\a :}{(-\d_{0} , \d_{0}) \times \Hn{2}}{\RR}{(\l ,p)}{\a(\l , p) = \a_{\l}(p) := e^{2 \l^{2} f(p)}}
   $$
   is $\Cl{\infty}$ and satisfies the following:

   \begin{enumerate}
      \item $\a_{0} \equiv 1$; 

      \smallskip 

      \item for all $\l \in (-\d_{0} , \d_{0})$ and $p \in \Hn{2} \setmin \cD(R / 2)$, 
      we have $\a_{\l}(p) = 1$; and 

      \smallskip 

      \item for all $\l \in (-\d_{0} , \d_{0})$ with $\l \neq 0$, 
      the Riemannian metric $g_{\l} : \Hn{2} \to \mathrm{Sym}_{2}(T\Hn{2})$
      defined by $g_{\l}(p) = \a_{\l}(p) g_{0}(p)$ is $\Cl{\infty}$, complete, 
      and has \emph{non-constant negative} Gaussian curvature 
      on \emph{any} neighborhood about $0$ in $\Hn{2}$.
   \end{enumerate}
\end{proposition}

\medskip

\begin{proof}~\\ 
\indent The only two things to be proved deal with completeness and Gaussian curvature,
since all the other points are clear. 

\smallskip 

So, fix $\l \in (-\d_{0} , \d_{0}) \setmin \{ 0 \}$.

\medskip

$\bullet$ \textsf{Step~1:} 
To prove $g_{\l}$ is complete, we will use the Hopf-Rinow theorem. 

\smallskip 

Let $X$ be a closed set in $\Hn{2}$ that is bounded for $g_{\l}$, and prove it is compact. 

\smallskip 

We have $X = X_{1} \cup X_{2}$ with $X_{1} = X \cap \bar{\cD}$ and $X_{2} = X \cap (\Hn{2} \setmin \cD)$.
As $X_{1}$ is closed in the compact set $\bar{\cD}$, it is compact. 

\smallskip 

On the other hand, $X_{2}$ is included in the open set $(\Hn{2} \setmin \bar{\cD(R / 2)})$ of $\Hn{2}$
on which $g_{\l}$ coincide with $g_{0}$. So, $X_{2}$ is bounded for $g_{0}$, and hence compact
since the Klein metric $g_{0}$ is complete.

\smallskip 

Conclusion: $X = X_{1} \cup X_{2}$ is compact. 

\medskip

$\bullet$ \textsf{Step~2:}  
The Gaussian curvature $K_{\l}$ of the metric $g_{\l}$
depends on that $K_{0} \equiv -1$ of $g_{0}$ according to 
the formula $\a_{\l} K_{\l} = K_{0} - \D(\ln{\! (\a_{\l})}) / 2$
(see for example~\cite{GKM75}, page~97), 
which implies $K_{\l}(p) = -(1 + \l^{2} (\D f)(p)) e^{-2 \l^{2} f(p)}$ for all $p \in \Hn{2}$. 
Thus, for every $p \in \cD(R / 4)$, we have $K_{\l}(p) < 0$ 
since $1 + \l^{2} (\D f)(p) \geq 1 - \l^{2} / \d_{0}^{2} > 0$.

\smallskip 

On the other hand, given $r \in (0 , R / 4)$, if $K_{\l}$ were constant on $\bar{\cD(r)}$,
there would exist $C \in \RR$ such that for all $p \in \bar{\cD(r)}$,
$$
\ln{\! (1 + \l^{2} (\D f)(p))} = 2 \l^{2} f(p) + C,
$$
and hence 
\begin{equation} \label{equ:condition-1}
   \ln{\! (1 + a \l^{2} f(p))} = 2 \l^{2} f(p) + C
\end{equation}
since $\D f = a f$ on $\bar{\cD(r)} \inc \cD(R / 4)$ by construction of $f$. 

\smallskip 

Defining $t_{0} := \min{\! \{ f(p) \st p \in \bar{\cD(r)} \}}$ 
and $t_{1} := \max{\! \{ f(p) \st p \in \bar{\cD(r)} \}}$,
Equation~\ref{equ:condition-1} writes
\begin{equation} \label{equ:condition-2}
   \ln{\! (1 + a \l^{2} t)} = 2 \l^{2} t + C
\end{equation}
for all $t \in [t_{0} , t_{1}] = f(\bar{\cD(r)})$. 

\smallskip 

Since the function $f$ coincides with $\p > 0$ on $\cD(R / 4)$, 
it never vanishes on $\bar{\cD(r)}$, and hence it cannot be constant on $\bar{\cD(r)}$ 
(indeed, if $f$ were constant on $\bar{\cD(r)}$, 
then we would have $f = (\D f) / a \equiv 0$ on $\bar{\cD(r)} \inc \cD(R / 4)$). 
Therefore we have $t_{0} < t_{1}$, which makes sense to differentiate 
Equation~\ref{equ:condition-2} with respect to $t$ and get 
$$
\frac{a \l^{2}}{1 + a \l^{2} t} = 2 \l^{2}
$$
for all $t \in [t_{0} , t_{1}]$. 

But this is impossible since $a \neq 0$ and $\l \neq 0$. 

\smallskip 

Conclusion: the Gaussian curvature $K_{\l}$ cannot be constant on $\bar{\cD(r)}$. 
\end{proof}

\bigskip

Let us now show how we use Proposition~\ref{prop:conformal} to construct a family 
$\fam{\gS}{\l}{(-\d_{0} , \d_{0})}$ of sets of parameterized curves $\c : [0 , 1] \to \Rn{2}$ 
we will prove later they are admissible for $\cD$ and have Property~(C).  

\medskip 

For each $\l \in (-\d_{0} , \d_{0})$ and $x \in \Hn{2}$, denote by 
$\exp_{x}^{\l} : T_{x}\Hn{2} = \Rn{2} \to \Hn{2}$ the exponential map at $x$ associated with $g_{\l}$, 
and let $\exp^{\l} : T\Hn{2} = \Hn{2} \times \Rn{2} \to \Hn{2} \times \Hn{2}$ 
be defined by $\exp^{\l}(x , v) = (x , \exp_{x}^{\l}(v))$. 
Since $g_{\l}$ is negatively curved, it has no conjugate points and thus 
$\exp^{\l}$ is a $\Cl{\infty}$ diffeomorphism. In particular, $g_{\l}$ is uniquely geodesic.

\medskip 

We next fix a $\Cl{\infty}$ function $\r : \RR \to [0 , 1]$ such that 
\begin{equation} \label{equ:rho} 
\begin{split}
   (1) \ & \ \r \equiv 1 \ \mbox{on} \ [1 / 2 , 2 / 3]; \\ 
   (2) \ & \ \r \equiv 0 \ \mbox{on} \ [3 / 4 , +\infty); \ \mbox{and} \\ 
   (3) \ & \ \r(t) = \r(1 - t) \ \mbox{for all} \ t \in \RR. 
\end{split}
\end{equation}

\medskip 

Given any $\l \in (-\d_{0} , \d_{0})$ and $x \in \Hn{2}$, let $G_{x}^{\l} : \Hn{2} \times \RR
\to \Hn{2}$ and $\f_{x}^{\l} : \Hn{2} \times \RR \to \Rn{2}$ be defined by
$$
G_{x}^{\l}(y , t) := \exp^{\l}(x \, , \, t (\exp_{x}^{\l})^{-1}(y))
$$
and
$$
\f_{x}^{\l}(y , t) := (1 - \r(t)) G_{x}^{0}(y , t) + \r(t) G_{x}^{\l}(y , t).
$$

\smallskip 

Roughly speaking, we obtain the parameterized curve $\f_{x}^{\l}(y , \cdot) : \RR \to \Rn{2}$ 
as the `barycenter' in $\Rn{2}$ with `weights' $1 - \r$ and $\r$ of the unique maximal geodesics 
$G_{x}^{0}(y , \cdot)$ and $G_{x}^{\l}(y , \cdot)$ for $g_{0}$ and $g_{\l}$ respectively 
passing through $x$ at $t = 0$ and $y$ at $t = 1$ (see Figure~\ref{fig:barycenter}). 

\begin{figure}[h] 

\begin{pspicture}(8,6.5) 
   \psdots[dotstyle=*,dotscale=0.3](4,3) 
   \uput{0.1}[90](4,3){$0$} 
   \pscircle[linewidth=0.02](4,3){3} 
   \uput{0.1}[45](6.236,5){$\bD$} 
   \pscircle[linewidth=0.01](4,3){1.5} 
   \uput{1.5}[70](4,3){$\bD(\! R / 2)$} 
   \psline[linewidth=0.04,linestyle=dotted](0,2)(8,2) 
   \psline[linewidth=0.06]{c->}(2.605,2)(2.72,2) 
   \uput{0.15}[260](2.4,2){$G_{x}^{0}(y , \cdot)$} 
   \newgray{g5}{0.5} 
   \pscurve[linewidth=0.03,linecolor=g5](0.2,1.203)(1.172,2)(3,3.5) 
   \pscurve[linewidth=0.03,linecolor=g5](5,3.5)(6.828,2)(7.8,1.203) 
   \psline[linewidth=0.07,linecolor=g5]{c->}(1.85,2.557)(1.965,2.65) 
   \uput{0.18}[115]{39}(1.9,2.6){$G_{x}^{\l}(y , \cdot)$} 
   \psdots[dotstyle=*,dotscale=0.9](1.182,2)(6.818,2) 
   \uput{0.2}[140](1.172,2){$x$} 
   \uput{0.2}[40](6.828,2){$y$} 
   \psline[linewidth=0.04](0,2)(1.6,2) 
   \pscurve[linewidth=0.04](1.5,2)(1.55,2)(1.6,2)(1.65,2)(2,2.05)(2.5,2.4)(3.07,3.51)
   (3.5,3)(4,4)(4.5,3)(4.93,3.51)(5.5,2.4)(6,2.05)(6.35,2)(6.4,2)(6.45,2)(6.5,2) 
   \psline[linewidth=0.04](6.4,2)(8,2) 
   \psline[linewidth=0.08]{c->}(2.709,2.94)(2.749,3.1) 
   \uput{0.15}[310](2.7,3){$\f_{x}^{\l}(y , \cdot)$} 
\end{pspicture}  
   \caption{\label{fig:barycenter} Constructing $\f_{x}^{\l}(y , \cdot)$} 
\end{figure}
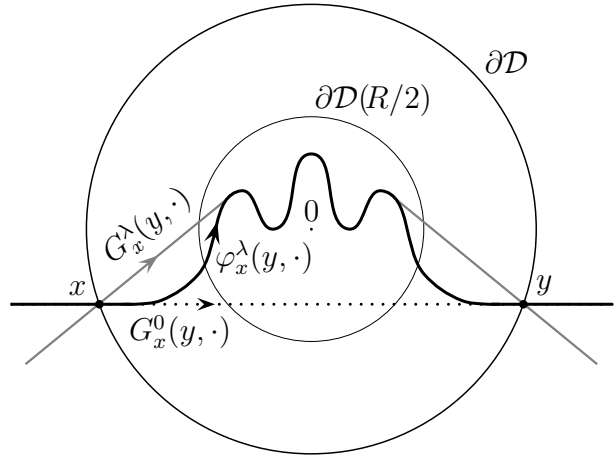

\bigskip 

In the rest of this section, we prove that if we shrink $\d_{0} > 0$,
then for each $\l \in (-\d_{0} , \d_{0})$, the
set of parameterized curves $\f_{x}^{\l}(y , \cdot) : [0 , 1] \to \Rn{2}$,
where $x$ and $y$ are distinct points in
$\bD$, is admissible for $\cD$ and satisfies Property~(C).

\medskip 

Then, in section~\ref{sec:end}, we prove
these parameterized curves have additional properties that will be used to ensure that the Finsler
metrics resulting from Theorem~\ref{thm:Arcostanzo} satisfy our Main Theorem. 

\bigskip 

In the following technical lemma, we show that for any $\l \in (-\d_{0} , \d_{0})$, 
if $C$ is a closed convex set in $\Rn{2}$ containing the open disk $\cD(R / 2)$, 
then $C \cap \Hn{2}$ is in some sense convex 
with respect to the set of parameterized curves
$G_{x}^{\l}(y , \cdot) : [0 , 1] \to \Hn{2}$ 
(respectively $\f_{x}^{\l}(y , \cdot) : [0 , 1] \to \Rn{2}$),
where $x , y \in C \cap \Hn{2}$.

\begin{lemma} \label{lem:basic-1}
   For each $\l \in (-\d_{0} , \d_{0})$, we have 
   
   \begin{enumerate} 
      \item $\f_{x}^{\l}(y , t) = \f_{y}^{\l}(x , 1 - t)$ for all $x , y \in \Hn{2}$ and $t \in \RR$, 

      \smallskip 

      \item if $C$ is any closed convex set in $\Rn{2}$ such that $\cD(R / 2) \inc C$,
      then $G_{x}^{\l}(y , t) \in C$ and $\f_{x}^{\l}(y , t) \in C$ 
      for all $x , y \in C \cap \Hn{2}$ and $t \in [0 , 1]$, 

      \smallskip 

      \item for all $x , y \in \bD$ and $t \in \RR$, 
      the equivalence $\f_{x}^{\l}(y , t) \in \cD \iff t \in (0 , 1)$ holds.
   \end{enumerate}
\end{lemma}

\medskip

\begin{proof}~\\ 
\indent Fix $\l \in (-\d_{0} , \d_{0})$.

\medskip

$\bullet$ \textsf{Point~(1):} 
Given any $x, y \in \Hn{2}$, the parameterized curves $t \in \RR \mapsto G_{x}^{\l}(y , t) \in \Hn{2}$ 
and $t \in \RR \mapsto G_{y}^{\l}(x , 1 - t) \in \Hn{2}$ 
are both $g_{\l}$-geodesics passing through $x$ at $t = 0$ and $y$ at $t = 1$. 
They are thus equal since $g_{\l}$ is uniquely geodesic. 
Then point~$(1)$ follows from property~$(1)$ in Equation~\ref{equ:rho} satisfied by the function $\r$.

\medskip

$\bullet$ \textsf{Point~(2):} 
Let $C$ be a closed convex set in $\Rn{2}$ such that $\cD(R/2) \inc C$. 
Let $x , y \in C \cap \Hn{2}$, and consider the $g_{\l}$-geodesic $\k : \RR \to \Hn{2}$ 
defined by $\k(t) := G_{x}^{\l}(y , t)$. 

\smallskip

We shall now prove by contradiction that the image of $\k$ is included in $C$. 
Then, since the image of the $g_0$-geodesic $G_{x}^{0}(y , \cdot)$ is in $C$, 
we will get that the image of the interpolated curve $\f_{x}^{\l}(y , \cdot)$ is also in $C$. 
So, let us suppose that there exists $t_{0} \in [0 , 1]$ such that $p_{0} = \k(t_{0}) \notin C$ 
and prove this is not possible.

\smallskip

Let $\tau_{0} := \max{\! \{ t \in [0 , t_{0}] \st \k(t) \in C \}}$ 
and $\tau_{1} := \min{\! \{t \in [t_{0} , 1] \st \k(t) \in C \}}$ 
(note that $\tau_{0}$ and $\tau_{1}$ exist since $\k(0) = x \in C$ and $\k(1) = y \in C$). 

\smallskip 

Then $\tau_{0} < t_{0} < \tau_{1}$, and for all $t \in (\tau_{0} , \tau_{1})$ we have $\k(t) \notin C$, 
which implies that $\k((\tau_{0} , \tau_{1}))$ is the affine segment $]\k(\tau_{0}) , \k(\tau_{1})[$
since the metric $g_{\l}$ coincides with $g_{0}$ on the open set $\Hn{2} \setmin C$ of $\Hn{2}$
(recall images of the $g_{0}$-geodesics are affine segments). 
But $]\k(\tau_{0}) , \k(\tau_{1})[ \ \inc [\k(\tau_{0}) , \k(\tau_{1})] \inc C$ 
since $\k(\tau_{0}) , \k(\tau_{1}) \in C$ and $C$ is convex.
As $p_{0} = \k(t_{0}) \in \k((\tau_{0} , \tau_{1})) = \ ]\k(\tau_{0}) , \k(\tau_{1})[$,
we get a contradiction.  
So, $\k([0 , 1]) \inc C$.

\smallskip

On the other hand, the image of $[0 , 1]$ under the $g_0$-geodesic $G_{x}^{0}(y , \cdot)$ 
is the affine segment $[x , y]$, which lies in $C$ since $x , y \in C$ and $C$ is convex.

\smallskip

Finally, for all $t \in [0 , 1]$, the barycenter point 
$(1 - \r(t)) G_{x}^{0}(y , t) + \r(t) G_{x}^{\l}(y , t) = \f_{x}^{\l}(y , t)$ 
is contained in the convex set $C$. 

\medskip

$\bullet$ \textsf{Point~(3):} 
Let $x , y \in \bD$ such that $x \neq y$ (the case $x = y$ is trivial). 

\smallskip

To prove the $\imp$ part, we show that $\f_{x}^{\l}(y , t) \notin \cD$ 
for all $t \in \RR \setmin (0 , 1)$. 
The idea consists here in saying that if the parameterized curve $\f_{x}^{\l}(y , \cdot)$ 
leaves the disk $\cD$, then it is equal to a $g_{0}$-geodesic. 
Hence, since the image of any geodesic for $g_{0}$ is an affine segment, 
$\f_{x}^{\l}(y , \cdot)$ can never go back into $\cD$. 

\smallskip 

So, let $c : \RR \to \Hn{2}$ be the $g_{0}$-geodesic defined by $c(t) := G_{x}^{0}(y , t)$.
As the images of $g_{0}$-geodesics are affine segments, we can write $c(t) = x + \t(t) (y - x)$ 
for all $t \in \RR$, 
where $\t : \RR \to \RR$ is a $\Cl{\infty}$ function. 
Since $c$ is a regular parameterized curve (it is a non-constant geodesic for a Riemannian metric)
satisfying $c(0) = x$ and $c(1) = y$, the derivative of $\t$ never vanishes and
we have $\t(0) = 0$ and $\t(1) = 1$.
Therefore $\t$ is an increasing homeomorphism with $\t([0 , 1]) = [0 , 1]$.

\smallskip 

This implies that $c(\RR \setmin (0 , 1))$ is equal to the complement of the affine segment $]x , y[$
in the intersection of the straight line $(x y)$ with $\Hn{2}$.
Since $]x , y[$ is the intersection of $(x y)$ with $\cD$, 
we get the inclusion $c(\RR \setmin (0 , 1)) \inc \Hn{2} \setmin \cD$.

\smallskip 

But $\f_{x}^{\l}(y , t) = c(t)$ for all $t \in \RR \setmin (0 , 1)$ 
since $\r \equiv 0$ on $\RR \setmin (0 , 1)$ by property~(2) in Equation~\ref{equ:rho}, 
and thus $\f_{x}^{\l}(y , \cdot)(\RR \setmin (0 , 1)) = c(\RR \setmin (0 , 1)) \inc \Hn{2} \setmin \cD$. 

\smallskip

This establishes the $\imp$~part in point~(3).

\medskip

To prove the $\con$~part, let $\n : \RR \to \Hn{2}$ be the $g_{\l}$-geodesic 
defined by $\n(t) := G_{x}^{\l}(y , t)$. 

\smallskip 

Applying point~(2) with $C = \bar{\cD}$, we already have $\n((0 , 1)) \in \bar{\cD}$.
Then suppose there exists $t_{0} \in (0 , 1)$ such that $p_{0} = \n(t_{0}) \in \bD$ 
and prove this is not possible.

\smallskip 

Since $p_{0}$ lies in the open set $\Hn{2} \setmin \bar{\cD(R / 2)}$ of $\Hn{2}$,
the continuity of $\n$ at $t_{0}$ implies there exists $\e > 0$
such that $[t_{0} - \e , t_{0} + \e] \inc [0 , 1]$ 
and $\n([t_{0} - \e , t_{0} + \e]) \inc \Hn{2} \setmin \bar{\cD(R / 2)}$.
But $g_{\l}$ agrees with $g_{0}$ on $\Hn{2} \setmin \bar{\cD(R / 2)}$,
so $\n(t) = G_{x}^{0}(y , t)$ for all $t \in [t_{0} - \e , t_{0} + \e]$,
and thus $\n([t_{0} - \e , t_{0} + \e])$ is the affine segment $[\n(t_{0} - \e) , \n(t_{0} + \e)]$. 
Hence $\n(t_{0} - \e) , \n(t_{0} + \e) \in \bar{\cD}$ with $p_{0} \in [\n(t_{0} - \e) , \n(t_{0} + \e)]$, 
which is impossible since $p_{0} \in \bD$ is an extreme point for the convex set $\bar{\cD}$. 
This shows that $\n((0 , 1)) \inc \cD$.

\smallskip

On the other hand, the image of $(0 , 1)$ under the $g_0$-geodesic $G_{x}^{0}(y , \cdot)$ 
is the affine segment $]x , y[$, which lies in $\cD$ since $x , y \in \bD$ and $\cD$ is strictly convex.

\smallskip

Finally, for all $t \in (0 , 1)$, the barycenter point 
$(1 - \r(t)) G_{x}^{0}(y , t) + \r(t) G_{x}^{\l}(y , t) = \f_{x}^{\l}(y , t)$ 
is contained in the convex set $\cD$. 
\end{proof}

\bigskip

We now consider the $\Cl{\infty}$ map 
$\F : (-\d_{0} , \d_{0}) \times \bD \times \bD \times \RR \to \bD \times \Rn{2}$ defined by
$$
\F(\l , (x , y , t)) = \F_{\l}(x , y , t) := (x , \f_{x}^{\l}(y , t))
$$
and denote by $\D := \{ (x , x) \st x \in \bD \}$ the diagonal of $\bD \times \bD$.

\smallskip 

Using this map $\F$, we shall prove that for all $\l \in (-\d_{0} , \d_{0})$, the set 
$$
\gS_{\l} := \{ \c^{\l}_{(x , y)} \st (x , y) \in (\bD \times \bD) \setmin \D \}
$$
of $\Cl{\infty}$ parameterized curves $\c^{\l}_{(x , y)} : [0 , 1] \to \Rn{2}$
defined by $\c^{\l}_{(x , y)}(t) := \f_{x}^{\l}(y , t)$ is \emph{admissible} for $\cD$
and \emph{satisfies} Property~(C) provided $\d_{0} > 0$ is \emph{sufficiently} small.

\bigskip


\subsubsection{\textsf{\emph{Diffeomorphism property for $\F_{\l}$}}} 
~

\smallskip 

Let $\cT := \{ (x , p) \in \bD \times \Hn{2} \st p - x \in T_{x}\bD \} 
= \{ (x , p) \in \bD \times \Hn{2} \st \scal{x}{p - x} = 0 \}$, 
$M := ((\bD \times \bD) \setmin \D) \times (\RR \setmin \{ 0 \})$ 
and $N := (\bD \times \Hn{2}) \setmin \cT$.

\smallskip 

The aim of this section is to prove the following:

\begin{proposition} \label{prop:diffeo}
   ~
   \begin{enumerate}
      \item For every $\l \in (-\d_{0} , \d_{0})$, we have $\F_{\l}(M) \inc N$. 

      \smallskip 

      \item There is $a \in (0 , \d_{0})$ such that $\F_{\l} : M \to N$ 
      is a diffeomorphism for each $\l \in (-a , a)$.
   \end{enumerate}
\end{proposition}

\medskip 

Thanks to this \emph{key} proposition and Corollary~\ref{cor:basic-2} below, we will be able to prove 
that the set $\gS_{\l}$ satisfies properties~(1), (3) and~(4) in Definition~\ref{def:admissible} 
(admissibility) together with point~(2) in Definition~\ref{def:property-C} (Property~(C)) 
after a suitable shrink of $a > 0$.  
Then, since property~(2) in Definition~\ref{def:admissible} is obvious by construction of $\gS_{\l}$ 
and since point~(1) in Definition~\ref{def:property-C} will be a consequence of 
Proposition~\ref{prop:end-point} below, $\gS_{\l}$ will finally be a set of parameterized curves 
that is admissible for $\cD$ and has Property~(C). 

\medskip

Now, the argument to prove Proposition~\ref{prop:diffeo} consists in saying that since 
it is obviously true for $\l = 0$, it still remains true for any $\l$ that is very close to $0$. 

\smallskip 

In order to apply this perturbation argument in a rigorous way, 
we will make use of two classical results in algebraic and
differential topology we recall here:

\begin{lemma}[Covering maps. See \cite{God71}, page~109] \label{lem:covering}
   Let $X$ and $Y$ be Hausdorff topological spaces such that $X$ is compact and $Y$ is connected. 
   Then any local homeomorphism $f : X \to Y$ is a covering map with a finite number of sheets.
\end{lemma}

\medskip

and 

\medskip 

\begin{lemma}[Regular points. See \cite{GuiPol74}, page~35] \label{lem:reg-pts}
   Let $\L$, $M$ and $N$ be $\Cl{1}$ manifolds, and let
   $$
   \map{F :}{\L \times M}{N}{(\l , x)}{F(\l , x) = f_{\l}(x)}
   $$
   be a $\Cl{1}$ map. 
   Let $\l_{0} \in \L$, $y_{0} \in N$ and $K \inc M$ be a compact set. 
   Then, if every $x \in f_{\l_{0}}^{-1}(y_{0}) \cap K$ is a regular point of $f_{\l_{0}}$,
   there exists an open neighborhood $U$ of $\l_{0}$ in $\L$ such that for each $\l \in U$,
   any $x \in f_{\l}^{-1}(y_{0}) \cap K$ is a regular point of $f_{\l}$.
\end{lemma}

\bigskip

\begin{proof}[Proof of Proposition~\ref{prop:diffeo}]~\\ 
\indent The proof will consist in five steps. 

\smallskip 

After showing that $\F_{0}(M) \inc N$ and $\F_{0} : M \to N$ is a bijection,
we first prove that $\F_{0}$ is a local diffeomorphism. 
We use this and Lemma~\ref{lem:reg-pts} (Regular points Lemma) to find a value $a \in (0 , \d_{0})$ 
such that for each $\l \in (-a , a)$, $\F_{\l}(M) \inc N$ 
and $\F_{\l} : M \to N$ is also a local diffeomorphism. 
Next, we use Lemma~\ref{lem:covering} (Covering maps Lemma) to obtain that $\F_{\l} : M \to N$ is a
finite sheeted covering map.  
Finally we prove there is a point in $N$ at which the number of pre-images
for this covering map is~1, and conclude $\F_{\l} : M \to N$ is a diffeomorphism.

\medskip

$\bullet$ \textsf{Step~1:}  
We begin by showing that $\F_{0}(M) \inc N$ and $\F_{0} : M \to N$ is a bijection.

\smallskip 

Since the images of $g_{0}$-geodesics are affine segments, 
for each $(x , y , t) \in ((\bD \times \bD) \setmin \D) \times \RR$,
there is a unique real number $\om(x , y , t)$ 
such that $\f_{x}^{0}(y , t) = G_{x}^{0}(y , t) = x + \om(x , y , t) (y - x)$,
and thus $\F_{0}(x , y , t) = (x \, , \, x + \om(x , y , t) (y - x))$. 
The function $\om : ((\bD \times \bD) \setmin \D) \times \RR \to \RR$ 
is therefore $\Cl{\infty}$ by smoothness of $\F_{0}$,
and satisfies the two following properties for each $(x , y) \in (\bD \times \bD) \setmin \D$:

\begin{enumerate}[$(i)$]
   \item $\om(x , y , 0) = 0$ and $\om(x , y , 1) = 1$; 

   \item for any $t \in \RR$, $\disp \pder{\om}{t}(x , y , t) \neq 0$.
\end{enumerate}

\smallskip 

For $x , y \in \bD$ with $x \neq y$ and $t \in \RR \setmin \{ 0 \}$, we then have 
$\scal{x \,}{\, x + \om(x , y , t) (y - x) - x} = \om(x , y , t) \scal{x}{y - x} \neq 0$ 
since $\scal{x}{y - x} \neq 0$ ($\bD$ is a Euclidean circle), $\om(x , y , 0) = 0$ (point~$(i)$ above) 
and $\om(x , y , \cdot) : \RR \to \RR$ is injective (point~$(ii)$ above). 
This shows $\F_{0}(M) \inc N$.

\smallskip 

Now, given any $(x , p) \in N$, let $y$ be the intersection point of the straight line $(x p)$ with $\bD$.
We have $y \neq x$, and thus we can write $p = x + \om(x , y , t) (y - x)$ with a unique $t \in \RR$
($\om(x , y , \cdot) : \RR \to \RR$ is injective) which is not equal to $0$ since $p \neq x$.
This proves there is a unique $(x , y , t) \in M$ such that $\F_{0}(x , y , t) = (x , p)$.
Hence $\F_{0} : M \to N$ is a bijection.

\medskip

$\bullet$ \textsf{Step~2:}  
Let us prove $\F_{0} : M \to N$ is a local diffeomorphism. 

\smallskip 

Given any $(x , y , t) \in M$, it suffices to show that the linear tangent map
$T_{(x , y , t)}\F_{0} : T_{(x , y , t)}M \to T_{\F_{0}(x , y , t)}N$
is injective since the manifolds $M$ and $N$ have the same dimension (equal to three).

\smallskip 

But for all $(u , v , s) \in T_{(x , y , t)}M = T_{x}\bD \times T_{y}\bD \times \RR$, we compute
\begin{multline*}
   \Tg{(x , y , t)}{\F_{0}}{(u , v , s)} = \\
   \left( \! u \, , \, u + \left\{ \pder{\om}{x}(x , y , t) \act u 
   + \pder{\om}{y}(x, y , t) \act v 
   + s \pder{\om}{t}(x , y , t) \right\} \! (y - x) + \om(x , y , t) (v - u) \! \right).
\end{multline*}

\smallskip 

So, if $\Tg{(x , y , t)}{\F_{0}}{(u , v , s)} = (0 , 0) \in T_{\F_{0}(x , y , t)}N = T_{x}\bD \times \Rn{2}$,
we get 
$$
u = 0 
\quad \mbox{and} \quad 
u + \left\{ \pder{\om}{x}(x , y , t) \act u + \pder{\om}{y}(x, y , t) \act v 
+ s \pder{\om}{t}(x , y , t) \right\} \! (y - x) + \om(x , y , t) (v - u) = 0.
$$
Hence 
$\disp \left\{ \pder{\om}{y}(x, y , t) \act v + s \pder{\om}{t}(x , y , t) \right\} \! (x - y) 
= \om(x , y , t) v$.

\smallskip 

As $v \in T_{y}\bD$, the first member of this equality lies in $T_{y}\bD$ too,
which implies
\begin{equation} \label{equ:kernel}
   \pder{\om}{y}(x, y , t) \act v + s \pder{\om}{t}(x , y , t) = 0 
\end{equation}
since $x - y \notin T_{y}\bD$. 
Thus $\om(x , y , t) v = 0$, that is $v = 0$ since $\om(x , y , t) \neq 0$ 
(use $t \neq 0$ and points~$(i)$ and~$(ii)$ in Step~1).

\smallskip 

Finally, replacing $v = 0$ in Equation~\ref{equ:kernel}, we obtain 
$\disp s \pder{\om}{t}(x , y , t) = 0$ and deduce $s = 0$ from point~$(ii)$ in Step~1. 

\smallskip 

Thus, $\F_{0} : M \to N$ is a local diffeomorphism.

\medskip

$\bullet$ \textsf{Step~3:} 
Now we fix $\l \in (-\d_{0} , \d_{0})$ and show by contradiction that $\F_{\l}(M) \inc N$.

\smallskip 

Let $(x_{0} , p) \in \cT$ and suppose there exist 
$y_{0} \in \bD \setmin \{ x_{0} \}$ and $t_{0} \in \RR \setmin \{ 0 \}$ 
such that $\F_{\l}(x_{0} , y_{0} , t_{0}) = (x_{0} , p)$.
Denoting by $L$ the tangent line to $\bD$ at $x_{0}$, we have $p \in L$.

\smallskip 

As $\r \equiv 0$ on $(-\infty , 1/4] \cup [3/4 , +\infty)$ by properties~(2) and~(3) 
in Equation~\ref{equ:rho}, we have 
$\f_{x_{0}}^{\l}(y_{0} , t) = G_{x_{0}}^{0}(y_{0} , t)$ 
for all $t \in (-\infty , 1/4] \cup [3/4 , +\infty)$. 
Thus, if we had $t_{0} \in (-\infty , 1/4] \cup [3/4 , +\infty)$, we would get 
$p = \f_{x_{0}}^{\l}(y_{0} , t_{0}) = G_{x_{0}}^{0}(y_{0} , t_{0})$, 
and thus $p$ would lie on the straight line $(x_{0} y_{0})$. 
But this implies $p = x_{0}$ since $(x_{0}y_{0}) \cap L = \{ x_{0} \}$, 
and hence $x_{0} = G_{x_{0}}^{0}(y_{0} , t_{0})$, which means that
$\F_{0}(x_{0} , y_{0} , t_{0}) = (x_{0} , x_{0}) = \F_{0}(x_{0} , y_{0} , 0)$.

\smallskip 

Since $\F_{0} : M \to N$ is injective, we then get $t_{0} = 0$,
which is not possible. 
Therefore, we have $t_{0} \in [1/4 , 3/4]$. 
But this is also impossible since by point~(3) of Lemma~\ref{lem:basic-1}, $t_{0} \in (0,1)$ implies
that $p = \f_{x_{0}}^{\l}(y_{0} , t_{0}) \in \cD$, and $L \cap \cD$ is empty.

\medskip

$\bullet$ \textsf{Step~4:} 
Now, let $\ell := R / \sqrt{3} > 0$. 
Then any chord of $\bD$ that is tangent to $\bD(R/2)$ has a Euclidean length equal to $3 \ell$. 

\smallskip 

Define $\Om := \{ (x , y) \in \bD \times \bD : |y - x| < \ell \}$ 
and consider the compact set $K := ((\bD \times \bD) \setmin \Om) \times [1/4 , 3/4] \inc M$. 

The complement of $K$ in $M$ is the disjoint union of the open sets 
$U_{1} := (\Om \setmin \D) \times (\RR \setmin \{ 0 \})$ and 
$U_{2} := ((\bD \times \bD) \setmin \Om) \times ((-\infty , 0) \cup (0 , 1/4) \cup (3/4 , +\infty))$ 
of $M$.  

\smallskip

We will first show that for each $\l \in (-\d_{0} , \d_{0})$, $\F_{\l} : M \to N$ induces a diffeomorphism 
from $M \setmin K$ onto an open set in $N$.
Then we will use Lemma~\ref{lem:reg-pts} to find a number $a \in (0 , \d_{0})$ such that 
for all $\l \in (-a , a)$, $\F_{\l} : M \to N$ is a local diffeomorphism at any point in $K$.

\smallskip

Fix $\l \in (-\d_{0} , \d_{0})$. 

\smallskip 

For every $(x , y) \in \Om \setmin \D$, 
the image of the $g_{0}$-geodesic $G_{x}^{0}(y , \cdot) : \RR \to \Hn{2}$ 
lies in the open set $\Hn{2} \setmin \bar{\cD(R/2)}$ of $\Hn{2}$ since this image is equal to the
intersection of the straight line $(x y)$ with $\Hn{2}$. 
As the Riemannian metrics $g_{\l}$ and $g_{0}$ coincide on $\Hn{2} \setmin \bar{\cD(R/2)}$, we
get that the $g_{\l}$-geodesic $G_{x}^{\l}(y , \cdot) : \RR \to \Hn{2}$ is actually
equal to $G_{x}^{0}(y , \cdot) : \RR \to \Hn{2}$. 
Hence, $\rest{(\F_{\l})}{U_{1}} = \rest{(\F_{0})}{U_{1}}$. 

\smallskip 

Next, using again the fact that $\r \equiv 0$ on $(-\infty , 1/4] \cup [3/4 , +\infty)$, 
we have $\f_{x}^{\l}(y , t) = G_{x}^{0}(y , t)$ 
for all $(x , y , t) \in \bD \times \bD \times ((-\infty , 1/4] \cup [3/4 , +\infty))$. 
Thus, $\rest{(\F_{\l})}{U_{2}} = \rest{(\F_{0})}{U_{2}}$. 

\smallskip 

We conclude that we have
$\rest{(\F_{\l})}{(M \backslash K)} = \rest{(\F_{0})}{(M \backslash K)}$, and hence
$\F_{\l} : M \to N$ induces a diffeomorphism from $M \setmin K$ 
onto $\F_{\l}(M \setmin K) = \F_{0}(M \setmin K)$,
which is an open set of $N$ since $\F_{0} : M \to N$ is an open map by Step~2.

\smallskip 

On the other hand, fixing $(x , p) \in N$, we have that the unique point $(x , y , t) \in M$ 
satisfying $\F_{0}(x , y , t) = (x , p)$ is regular for the diffeomorphism $\F_{0} : M \to N$, 
thus any point in $\F_{0}^{-1}((x , p)) \cap K$ is regular for $\F_{0}$. 
We can then apply Lemma~\ref{lem:reg-pts} to $\F : (-\d_{0} , \d_{0}) \times M \to N$ 
and get the existence of a number $a \in (0 , \d_{0})$ 
such that for each $\l \in (-a , a)$, all the points in $\F_{\l}^{-1}((x , p)) \cap K$ 
are regular for $\F_{\l} : M \to N$. 
As $M$ and $N$ have the same dimension, $\F_{\l} : M \to N$ is a local diffeomorphism at any point in $K$.

\smallskip 

Summing up, we proved that $\F_{\l} : M \to N$ is a local diffeomorphism for every $\l \in (-a , a)$.

\medskip

$\bullet$ \textsf{Step~5:} 
From now on, fix $\l \in (-a , a)$. 

\smallskip 

As $\F_{\l} : M \setmin K \to \F_{\l}(M \setmin K)$ is a diffeomorphism, 
what remains is for us to show that the map $\F_{\l} : K \to \F_{\l}(K)$ is one-to-one. 

\smallskip 

Since $\F_{\l} : K \to \F_{\l}(K)$ is a local homeomorphism by Step~4
and $K$ is compact and connected, we can apply Lemma~\ref{lem:covering} 
with $X = K$ and $Y = \F_{\l}(K)$ in order to get that $\F_{\l} : K \to \F_{\l}(K)$ 
is a covering map with a finite number of sheets. 
We complete the argument by finding a point in the image of this covering map 
at which the number of pre-images is~1.

\smallskip 

Choose $(x_{0} , y_{0} , t_{0}) \in K$ with $|y_{0} - x_{0}| = 2 \ell$ 
and let $p := \f_{x_{0}}^{\l}(y_{0} , t_{0})$. 

\smallskip 

Since any chord of $\bD$ that is tangent to $\bD(R/2)$ has a Euclidean length equal to $3 \ell$ 
and since $|y_{0} - x_{0}| < 3 \ell$, 
the straight line $(x_{0} y_{0})$ does not intersect with $\bar{\cD(R/2)}$. 
Then we have $\f_{x_{0}}^{\l}(y_{0} , t) = G_{x_{0}}^{0}(y_{0} , t)$ for all $t \in \RR$, 
and thus $p = G_{x_{0}}^{0}(y_{0} , t_{0})$.  
Consider any $y_{1} \in \bD \setmin \{ x_{0} \}$ and $t_{1} \in [1/4 , 3/4]$ 
such that $p = \f_{x_{0}}^{\l}(y_{1} , t_{1})$,
and let us prove that $y_{1} = y_{0}$ and $t_{1} = t_{0}$.

\smallskip 

Fix a closed half cone $C$ in $\Rn{2}$ whose vertex is $x_{0}$ 
and that contains $\cD(R/2)$ with $y_{0} \notin C$ (see Figure~\ref{fig:one-to-one}). 

\begin{figure}[h] 

\begin{pspicture}[-0.03](8,7) 
   \pscustom[linewidth=0.02]
   {
   \psline(4,7)(1.172,2)(7.5,0.5) 
   \gsave 
      \pscurve[liftpen=1](7.5,0.5)(7.1,3.6)(5.5,6.9) 
      \newgray{g9}{0.9} 
      \fill[fillstyle=solid,fillcolor=g9] 
   \grestore 
   } 
   \uput{0.1}[0](4.5,6.7){\darkgray $C$} 
   \psdots[dotstyle=*,dotscale=0.3](4,3) 
   \uput{0.1}[90](4,3){$0$} 
   \pscircle[linewidth=0.02](4,3){3} 
   \uput{0.1}[45](6.236,5){$\bD$} 
   \pscircle[linewidth=0.01](4,3){1.5} 
   \uput{1.5}[70](4,3){$\bD(\! R / 2)$} 
   \psdots[dotstyle=*,dotscale=0.9](1.182,2)(5,0.182)(2.5,1.366)(6.818,2) 
   \uput{0.15}[240](1.172,2){$x_{0}$} 
   \uput{0.18}[268](5,0.172){$y_{0}$} 
   \uput{0.15}[215](2.5,1.366){$p$} 
   \uput{0.17}[300](6.828,2){$y_{1}$} 
   \psline[linewidth=0.04,linestyle=dotted](0.672,2.239)(5.5,-0.067) 
   \psline[linewidth=0.06]{c->}(3.5,0.888)(3.6,0.84) 
   \uput{0.12}[250]{335}(3.55,0.86){$G_{\! x_{0}}^{0} \! (y_{0} , \cdot)$} 
   \psline[linewidth=0.04](0.6,2)(1.6,2) 
   \pscurve[linewidth=0.04](1.5,2)(1.55,2)(1.6,2)(1.65,2)(2,2.05)(2.5,2.4)(3.07,3.51)
   (3.5,3)(4,4)(4.5,3)(4.93,3.51)(5.5,2.4)(6,2.05)(6.35,2)(6.4,2)(6.45,2)(6.5,2) 
   \psline[linewidth=0.04](6.4,2)(8,2) 
   \psline[linewidth=0.08]{c->}(2.709,2.94)(2.749,3.1) 
   \uput{0.15}[310](2.7,3){$\f_{\! x_{0}}^{\l} \! (y_{1} , \cdot)$} 
   \psline(1.172,2)(7.5,0.5) 
   \psline(1.172,2)(4,7) 
\end{pspicture}  
   \caption{\label{fig:one-to-one} The map $\F_{\l} : K \to \F_{\l}(K)$ is one-to-one} 
\end{figure}
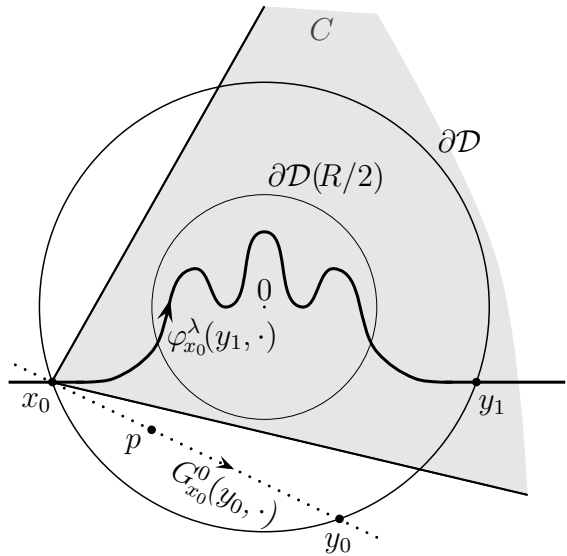

\smallskip 

We show by contradiction that $y_{1} \notin C$.  
If we assume $y_{1}$ is in $C$, then point~(2) in Lemma~\ref{lem:basic-1} 
implies $p = \f_{x_{0}}^{\l}(y_{1} , t_{1}) \in C$ since $t_{1} \in [0 , 1]$. 
But this is not possible since $p \notin C$ (indeed, we have $y_{0} \notin C$ 
and $p = G_{x_{0}}^{0}(y_{0} , t_{0})$ lies on the affine segment $]x_{0} , y_{0}[$). 
We conclude that we necessarily have $y_{1} \notin C$.

\smallskip

It follows that the straight line $(x_{0} y_{1})$ does not meet $\bar{\cD(R/2)}$, 
and thus $\f_{x_{0}}^{\l}(y_{1} , t) = G_{x_{0}}^{0}(y_{1} , t)$ for all $t \in \RR$.  
Therefore, 
$G_{x_{0}}^{0}(y_{0} , t_{0}) = p = \f_{x_{0}}^{\l}(y_{1} , t_{1}) = G_{x_{0}}^{0}(y_{1} , t_{1})$, 
or equivalently 
$\F_{0}(x_{0} , y_{0} , t_{0}) =(x_{0} , p) = \F_{0}(x_{0} , y_{1} , t_{1})$,
which implies $y_{1} = y_{0}$ and $t_{1} = t_{0}$ since $\F_{0} : M \to N$ is injective. 

In other words, we showed that $\F_{\l}^{-1}((x_{0} , p)) = 
\{ (x_{0} , y_{0} , t_{0}) \}$ with $(x_{0} , y_{0} , t_{0}) \in K$. 
Hence $(x_{0} , p) \in \F_{\l}(K)$ and there is a unique point 
in the fiber of $\F_{\l} : K \to \F_{\l}(K)$ over $(x_{0} , p)$. 
This proves that the covering map $\F_{\l} : K \to \F_{\l}(K)$ has only one sheet, 
which implies it is bijective. 

\smallskip 

But on the other hand, as we have seen that 
$\rest{(\F_{\l})}{(M \setmin K)} = \rest{(\F_{0})}{(M \setmin K)}$ 
and $\F_{0} : M \to N$ is a bijection, the map $\F_{\l} : M \setmin K \to \F_{\l}(M \setmin K)$ 
is also a bijection. 

\smallskip 

Hence $\F_{\l} : M \to N$ is bijective.  

\smallskip 

As we showed this map is a local diffeomorphism in Step~4, 
it is finally a diffeomorphism and this ends the proof of Proposition~\ref{prop:diffeo}.
\end{proof}

\bigskip

For each $\l \in (-a , a)$, we can now define the map $\s_{\l} : N \to \bD$ by
$\s_{\l}(x , p) := y$, where $y \in \bD$ is such that $(x , y , t)$ is the unique point in $M$
that satisfies $\F_{\l}(x , y , t) = (x , p)$ according to Proposition~\ref{prop:diffeo} 
(see Figure~\ref{fig:end-point-2}). 

\begin{figure}[h] 

\begin{pspicture}(8,6.5) 
   \psdots[dotstyle=*,dotscale=0.3](4,3) 
   \uput{0.1}[90](4,3){$0$} 
   \pscircle[linewidth=0.02](4,3){3} 
   \uput{0.1}[45](6.236,5){$\bD$} 
   \pscircle[linewidth=0.01](4,3){1.5} 
   \uput{1.5}[70](4,3){$\bD(\! R / 2)$} 
   \psdots[dotstyle=*,dotscale=0.9](1.182,2)(6.818,2) 
   \uput{0.15}[180](1.172,2){$x$} 
   \uput{0.15}[0](6.828,2){$y =: \s_{\l}(x , p)$} 
   \psline[linewidth=0.04](1.172,2)(1.6,2) 
   \pscurve[linewidth=0.04](1.5,2)(1.55,2)(1.6,2)(1.65,2)(2,2.05)(2.5,2.4)(3.07,3.51)
   (3.5,3)(4,4)(4.5,3)(4.93,3.51)(5.5,2.4)(6,2.05)(6.35,2)(6.4,2)(6.45,2)(6.5,2) 
   \psline[linewidth=0.04](6.4,2)(6.828,2) 
   \psline[linewidth=0.08]{c->}(2.709,2.94)(2.749,3.1) 
   \uput{0.15}[310](2.7,3){$\f_{x}^{\l}(y , \cdot)$} 
   \psdots[dotstyle=*,dotscale=0.9](5.82,2.14) 
   \uput{0.2}[265](5.82,2.14){$p$} 
\end{pspicture}  
   \caption{\label{fig:end-point-2} The `end point map' $\s_{\l}$} 
\end{figure}

\medskip 

\begin{remark} \label{rem:end-point} 
By the first point in Lemma~\ref{lem:basic-1}, for any $x , y , p \in \Hn{2}$, we have the equivalence
$$
((x , p) \in N \ \ \mbox{and} \ \ \s_{\l}(x , p) = y) 
\iff 
((y , p) \in N \ \ \mbox{and} \ \ \s_{\l}(y , p) = x).
$$
\end{remark}

\bigskip

Let us now prove the following useful result:  

\begin{lemma} \label{lem:diff-perturb}
   Let $\L$, $M$ and $N$ be $\Cl{k}$ manifolds ($k \geq 1$ integer), 
   and let $(f_{\l})_{\l \in \L}$ be a family of $\Cl{k}$ diffeomorphisms from $M$ to $N$. \\
   If $\map{\t :}{\L \times M}{N}{(\l , x)}{f_{\l}(x)}$ is of class $\Cl{k}$, 
   then the map $\map{h :}{\L \times N}{\L \times M}{(\l , x)}{(\l , f_{\l}^{-1}(x))}$ 
   is a $\Cl{k}$ diffeomorphism. 
   
   In particular, $\map{}{\L \times N}{M}{(\l , x)}{f_{\l}^{-1}(x)}$ is of class $\Cl{k}$. 
\end{lemma}

\medskip

\begin{proof}~\\ 
\indent Since the map $\map{h :}{\L \times M}{\L \times N}{(\l , x)}{(\l , f_{\l}(x)) =: (\l , \t(\l , x))}$
is of class $\Cl{k}$ and bijective, it suffices to show it is a local diffeomorphism.
But this is equivalent to showing that for any $(\l , x) \in \L \times M$,
the linear tangent map $T_{(\l , x)}h : T_{(\l , x)}(\L \times M) \to T_{h(\l , x)}(\L \times N)$
is injective since the manifolds $M$ and $N$ have the same dimension by hypothesis.

\smallskip 

Now, for all $(\x , v) \in T_{(\l , x)}(\L \times M) = T_{\l}\L \times T_{x}M$, we have
$$
\Tg{(\l , x)}{h}{(\x , v)} 
= 
(\x \, , \, \Tg{(\l , x)}{\t}{(\x , v)}) 
= 
\left( \! \x \, , \, \pder{\t}{\l}(\l , x) \act \x + \pder{\t}{x}(\l , x) \act v \! \right)
= 
\left( \! \x \, , \, \pder{\t}{\l}(\l , x) \act \x + \Tg{x}{f_{\l}}{v} \! \right) \! .
$$

\smallskip 

So, if $\Tg{(\l , x)}{h}{(\x , v)} = (0 , 0) \in T_{h(\l , x)}(\L \times N) = T_{\l}\L \times T_{f_{\l}(x)}N$,
we get
$$
\x = 0 \quad \mbox{and} \quad \pder{\t}{\l}(\l , x) \act \x + \Tg{x}{f_{\l}}{v} = 0.
$$ 
Hence $\Tg{x}{f_{\l}}{v} = 0$, which implies $v = 0$ since $f_{\l}$ is a diffeomorphism. 

\smallskip 

Conclusion: $h$ is a $\Cl{k}$ diffeomorphism. 

\smallskip 

Therefore, if $\pi : \L \times M \to M$ denotes the natural projection, 
$\pi \circ h^{-1} : \L \times N \to M$ is of class $\Cl{k}$, 
which establishes Lemma~\ref{lem:diff-perturb}. 
\end{proof}

\medskip

This lemma then implies

\begin{proposition} \label{prop:end-point} 
   The map 
   $$
   \map{}{(-a , a) \times N}{\bD}{(\l , (x , p))}{\s_{\l}(x , p)} \ \mbox{is} \ \Cl{\infty}. 
   $$
\end{proposition}

\medskip

\begin{proof}~\\ 
\indent If we introduce the natural projection $\pi : \Rn{2} \times \Rn{2} \times \RR \to \Rn{2}$ 
onto the second factor, we can write 
$\s_{\l}(x , p) = \pi(\F_{\l}^{-1}(x , p))$ for all $(x , p) \in N$. 
Then, applying Lemma~\ref{lem:diff-perturb} with $\L := (-a , a)$ and 
$f_{\l} := \F_{\l}$ (which is a diffeomorphism by point~(2) in Proposition~\ref{prop:diffeo}), 
we get the result. 
\end{proof}

\bigskip

A direct consequence of Proposition~\ref{prop:diffeo} is the following:

\begin{corollary} \label{cor:basic-2}
   Given any $\l \in (-a , a)$, we have 
   
   \begin{enumerate}[(1)]
      \item for every $(x , y) \in (\bD \times \bD) \setmin \D$,
      the $\Cl{\infty}$ parameterized curve $\f_{x}^{\l}(y , \cdot) : \RR \to \Rn{2}$ 
      is regular and injective, and 

      \smallskip 

      \item for every $(x , p) \in N$ and $V \in \Rn{2} \setmin \{ 0 \}$,
      $$
      \pder{\s_{\l}}{p}(x , p) \act V = 0 
      \iff 
      V \ \mbox{and} \ \pder{\f_{x}^{\l}}{t}(y , t) \ \mbox{are parallel vectors},
      $$
      where $p := \f_{x}^{\l}(y , t)$.
   \end{enumerate}
\end{corollary}

\medskip

\begin{proof}~\\ 
\indent $\bullet$ \textsf{Point~(1):}  
Let $(x , y) \in (\bD \times \bD) \setmin \D$.

\smallskip 

For any $t \in \RR$, we have
$$
\left( \! 0 \, , \, \pder{\f_{x}^{\l}}{t}(y , t) \! \right) 
=
\pder{\F_{\l}}{t}(x , y , t) 
= 
\Tg{(x , y , t)}{\F_{\l}}{(0 , 0 , 1)} \neq (0 , 0) 
$$ 
since $T_{(x , y , t)}\F_{\l} : T_{(x , y , t)}M \to T_{\F_{\l}(x , y , t)}N$ is one-to-one 
by Proposition~\ref{prop:diffeo}, 
and therefore $\disp \pder{\f_{x}^{\l}}{t}(y , t) \neq 0$.
Hence $\f_{x}^{\l}(y , \cdot) : \RR \to \Rn{2}$ is regular.

\smallskip

Let $t_{0}, t_{1} \in \RR$ such that $\f_{x}^{\l}(y , t_{0}) = \f_{x}^{\l}(y , t_{1})$. 
Then $\F_{\l}(x , y , t_{0}) = \F_{\l}(x , y , t_{1})$. 
If $t_{0} \neq 0$ and $t_{1} \neq 0$, we have $(x , y , t_{0}), (x , y , t_{1}) \in M$, 
and thus $t_{0} = t_{1}$ since $\F_{\l} : M \to N$ is injective 
by Proposition~\ref{prop:diffeo}. 
If $t_{0} \neq 0$ and $t_{1} = 0$, we get $\f_{x}^{\l}(y , t_{0}) = \f_{x}^{\l}(y , 0)$ 
which also writes $\f_{y}^{\l}(x , 1 - t_{0}) = \f_{y}^{\l}(x , 1)$ by point~(1) in Lemma~\ref{lem:basic-1}.
Since $x \neq y$, we have $1 - t_{0} \neq 0$, and thus $1 - t_{0} = 1$ in the same way as previously, 
\emph{i.e.}, $t_{0} = 0 = t_{1}$.
 
This shows that $\f_{x}^{\l}(y , \cdot) : \RR \to \Rn{2}$ is injective.

\medskip

$\bullet$ \textsf{Point~(2):} 
Let $(x , p) \in N$ and $V \in \Rn{2} \setmin \{ 0 \}$. 

\smallskip 

By Proposition~\ref{prop:diffeo}, there are unique elements $(x , y , t) \in M$,
$v \in T_{y}\bD$ and $s \in \RR$ such that $(x , p) = \F_{\l}(x , y , t)$ 
and $(0 , V) = \Tg{(x , y , t)}{\F_{\l}}{(0 , v , s)}$.

\smallskip 

Then we have the equivalences
\begin{eqnarray*}
   \pder{\s_{\l}}{p}(x , p) \act V = 0 
   & \iff & 
   \Tg{(x , p)}{\s_{\l}}{(0 , V)} = 0 \\
   & \iff & 
   \Tg{(x , y , t)}{(\s_{\l} \circ \F_{\l})}{(0 , v , s)} = 0 \\
   & \iff & 
   v = 0 \\ 
   & & \mbox{(since $\s_{\l}(\F_{\l}(x , y , t)) = y$)} \\
   & \iff & 
   (0 , V) = \Tg{(x , y , t)}{\F_{\l}}{(0 , 0 , s)} 
   = s \pder{\F_{\l}}{t}(x , y , t) = \left( \! 0 \, , \, s \pder{\f_{x}^{\l}}{t}(y , t) \! \right) \\
   & \iff & 
   V = s \pder{\f_{x}^{\l}}{t}(y , t),
\end{eqnarray*}
and we are done.
\end{proof}

\bigskip


\subsubsection{\textsf{\emph{Admissibility and Property~(C) for the set $\gS_{\l}$}}} 
~

\smallskip 

The following two propositions allow us to shrink $a > 0$ so that for \emph{each} $\l \in (-a , a)$, 
the set of parameterized curves 
$\gS_{\l} = \{ \c^{\l}_{(x , y)} \st (x , y) \in (\bD \times \bD) \setmin \D \}$ 
be admissible for $\cD$ and have Property~(C).

\begin{proposition} \label{prop:p-q}
   There exists a number $b \in (0 , a)$ such that for all $\l \in (-b , b)$ 
   and $p , q \in \bar{\cD}$ with $p \neq q$,
   there is a unique 
   $(x , y , t_{0} , t_{1}) \in ((\bD \times \bD) \setmin \D) \times [0 , 1] \times [0 , 1]$
   such that $p = \f_{x}^{\l}(y , t_{0})$, $q = \f_{x}^{\l}(y , t_{1})$ and $t_{0} < t_{1}$.
\end{proposition}

\smallskip 

This proposition will imply that $\gS_{\l}$ satisfies property~(3) 
in Definition~\ref{def:admissible} (admissibility) for every $\l \in (-b , b)$. 

\medskip

\begin{proposition} \label{prop:p-v}
   There exists a number $c \in (0 , a)$ such that for all $\l \in (-c , c)$, 
   $p \in \cD$ and $V \in \Rn{2} \setmin \{ 0 \}$,
   there is a unique $(x , y , t) \in ((\bD \times \bD) \setmin \D) \times (0 , 1)$
   such that $p = \f_{x}^{\l}(y , t)$
   and $\disp \pder{\f_{x}^{\l}}{t}(y , t)$ is parallel to $V$ with the same direction.
\end{proposition}

\smallskip 

This proposition will imply that $\gS_{\l}$ satisfies property~(4) 
in Definition~\ref{def:admissible} (admissibility) for every $\l \in (-c , c)$. 

\medskip

In order to prove these two results, we need the following classical lemma from
differential topology. This lemma allows us to show certain properties that are
true for $\l = 0$ continues to hold for $\l \in (-a , a)$ close enough to $0$.

\begin{lemma}[Regular value. See~\cite{Hir76}, Theorem~2.7] \label{lem:reg-val}
   Let $\L$, $M$ and $N$ be $\Cl{1}$ manifolds, and let
   $$
   \map{F :}{\L \times M}{N}{(\l , x)}{F(\l , x) = f_{\l}(x)}
   $$
   be a $\Cl{1}$ map.
   Given $y_{0} \in N$, we have
   
   \begin{enumerate}
      \item if $y_{0}$ is a regular value of $f_{\l}$ for all $\l \in \L$, 
      then $y_{0}$ is a regular value of $F$, 

      \smallskip 

      \item if $y_{0}$ is a regular value of $F$, 
      then $W = F^{-1}(y_{0})$ is a $\Cl{1}$ submanifold of $\L \times M$
      with dimension equal to $\dim{\L} + \dim{M} - \dim{N}$, and we have the equivalence 
      $$
      \forall \l \in \L,~(y_{0} \ \mbox{is a regular value of} \ f_{\l})
      \iff 
      (\l \ \mbox{is a regular value of} \ \rest{\pi}{W} : W \to \L),
      $$
      where $\pi : \L \times M \to \L$ is the natural projection.
   \end{enumerate}
\end{lemma}

\bigskip

\begin{proof}[Proof of Proposition~\ref{prop:p-q}]~\\ 
\indent Fix two distinct points $p , q \in \bar{\cD}$. 

\medskip 

$\bullet$ \textsf{Case~1:} 
Suppose $p \in \bD$ (the case $q \in \bD$ is similar).

\smallskip 

Then for each $\l \in (-a , a)$, 
there exist a unique $y \in \bD$ with $y \neq p$ and a unique $t_{1} \in \RR \setmin \{ 0 \}$
such that $\F_{\l}(p , y , t_{1}) = (p , q)$ by point~(2) in Proposition~\ref{prop:diffeo} 
since $(p , q) \in N$.
Hence we have $p = \f_{p}^{\l}(y , 0)$ and $q = \f_{p}^{\l}(y , t_{1})$ with 
$(p , y , 0) \in ((\bD \times \bD) \setmin \D) \times [0 , 1]$
and $t_{1} \in (0 , 1]$ by point~(2) in Lemma~\ref{lem:basic-1} since $q \in \bar{\cD}$.

\smallskip 

Now, let $(x' , y' , t'_{0} , t'_{1}) \in ((\bD \times \bD) \setmin \D) \times [0 , 1] \times [0 , 1]$
be such that $p = \f_{x'}^{\l}(y' , t'_{0})$ and $q = \f_{x'}^{\l}(y' , t'_{1})$ with $t'_{0} < t'_{1}$.
If we had $t'_{0} > 0$, then we would have $t'_{0} \in (0 , 1)$, and therefore $p \in \cD$ by point~(3)
in Lemma~\ref{lem:basic-1}. But this is not true.
So $t'_{0} = 0$, which implies $x' = \f_{x'}^{\l}(y' , 0) = p$,
and hence $\F_{\l}(p , y' , t'_{1}) = (p , q)$.

\smallskip 

But we already had $\F_{\l}(p , y , t_{1}) = (p , q)$, thus $y' = y$ and $t'_{1} = t_{1}$ 
since $t_{1} , t'_{1} \in \RR \setmin \{ 0 \}$ and $\F_{\l} : M \to N$ is injective
by point~(2) in Proposition~\ref{prop:diffeo}. 

\medskip

$\bullet$ \textsf{Case~2:} 
Suppose that both $p$ and $q$ are in $\cD$.

\smallskip 

Consider the function $F : (-a , a) \times \bD \to \RR$ 
defined by $F(\l , x) := f_{\l}(x) = \det{V_{p}^{\l}(x) , V_{q}^{\l}(x)}$, 
where $V_{p}^{\l}(x) := \s_{\l}(x , p) - x$ and $V_{q}^{\l}(x) := \s_{\l}(x , q) - x$ 
(see Figure~\ref{fig:prop-p-q}). 
Thanks to Proposition~\ref{prop:end-point}, this function is $\Cl{\infty}$. 

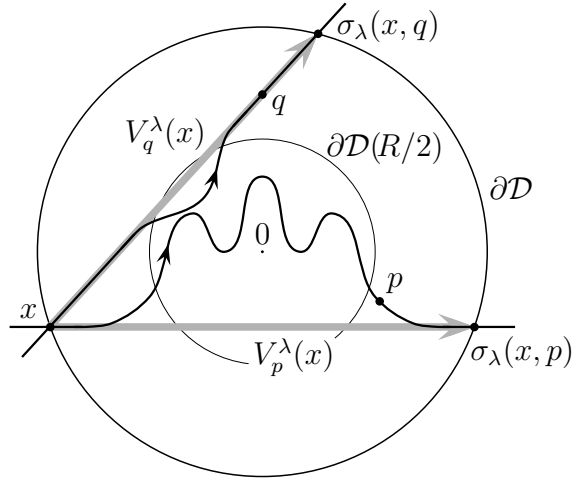
\begin{figure}[h] 

\begin{pspicture}(8,6.5) 
   \psdots[dotstyle=*,dotscale=0.3](4,3) 
   \uput{0.1}[90](4,3){$0$} 
   \pscircle[linewidth=0.02](4,3){3} 
   \uput{3.1}[15](4,3){$\bD$} 
   \psarc[linewidth=0.01](4,3){1.5}{-50}{263} 
   \uput{1.45}[50](4,3){$\bD(\! R / 2)$} 
   \newgray{g7}{0.7} 
   \psline[linewidth=0.1,linecolor=g7]{c->}(1.172,2)(6.828,2) 
   \psline[linewidth=0.15,linecolor=g7]{c->}(6.53,2)(6.828,2) 
   \uput{1.15}[284](4,3){$V_{p}^{\l}(x)$} 
   \psline[linewidth=0.1,linecolor=g7]{c->}(1.172,2)(4.743,5.907) 
   \psline[linewidth=0.15,linecolor=g7]{c->}(4.54,5.685)(4.743,5.907) 
   \uput{1.5}[125](4,3){$V_{q}^{\l}(x)$} 
   \psdots[dotstyle=*,dotscale=0.9](1.182,2)(6.818,2)(4,5.094)(4.743,5.901) 
   \uput{0.2}[140](1.172,2){$x$} 
   \uput{0.12}[308](6.828,2){$\s_{\l}(x , p)$} 
   \uput{0.12}[330](4,5.094){$q$} 
   \uput{0.25}[12](4.743,5.907){$\s_{\l}(x , q)$} 
   \psline[linewidth=0.03](0.64,2)(1.6,2) 
   \pscurve[linewidth=0.03](1.5,2)(1.55,2)(1.6,2)(1.65,2)(2,2.05)(2.5,2.4)(3.07,3.51)
   (3.5,3)(4,4)(4.5,3)(4.93,3.51)(5.5,2.4)(6,2.05)(6.35,2)(6.4,2)(6.45,2)(6.5,2) 
   \psline[linewidth=0.03](6.4,2)(7.36,2) 
   \psline[linewidth=0.07]{c->}(2.713,2.954)(2.749,3.1) 
   \psdots[dotstyle=*,dotscale=0.9](5.56,2.34) 
   \uput{0.15}[47](5.56,2.34){$p$} 
   \psline[linewidth=0.03](0.806,1.6)(2.299,3.233) 
   \pscurve[linewidth=0.03](2.116,3.033)(2.208,3.133)(2.299,3.233)(2.401,3.324)
   (3.218,3.715)(3.533,4.564)(3.615,4.673)(3.706,4.773)(3.798,4.873) 
   \psline[linewidth=0.03](3.615,4.673)(5.108,6.306) 
   \psline[linewidth=0.07]{c->}(3.355,3.976)(3.397,4.11) 
\end{pspicture}  
   \caption{\label{fig:prop-p-q} Proof of Proposition~\ref{prop:p-q}} 
\end{figure}

\smallskip 

Let $x_{0}$ and $y_{0}$ be the two intersection points of the strait line $(p q)$ with $\bD$. 
As the images of $g_{0}$-geodesics are affine segments, we have
$\s_{0}(x_{0} , p) = y_{0} = \s_{0}(x_{0} , q)$, which shows that
$V_{p}^{0}(x_{0}) = y_{0} - x_{0} = V_{q}^{0}(x_{0})$ and thus $f_{0}(x_{0}) = 0$. 
Similarly $f_{0}(y_{0}) = 0$, and we actually have $f_{0}^{-1}(0) = \{ x_{0} , y_{0} \}$. 
So, in order to prove $0$ is a regular value of $f_{0}$, we
have to show both $x_{0}$ and $y_{0}$ are regular points of $f_{0}$.

\smallskip 

Let us prove $x_{0}$ is a regular point of $f_{0}$ (arguments are the same for $y_{0}$). 

For any $u \in T_{x_{0}}\bD$, we have 
\begin{eqnarray*}
   \Tg{x_{0}}{f_{0}}{u} 
   & = & 
   \det{V_{p}^{0}(x_{0}) \, , \, \Tg{x_{0}}{V_{q}^{0}}{u}} 
   + \det{\Tg{x_{0}}{V_{p}^{0}}{u} \, , \, V_{q}^{0}(x_{0})} \\
   & = & 
   \det{y_{0} - x_{0} \, , \, \Tg{x_{0}}{V_{q}^{0}}{u} - \Tg{x_{0}}{V_{p}^{0}}{u}} \\
   & = & 
   \det{\! y_{0} - x_{0} \, , \, \pder{\s_{0}}{x}(x_{0} , q) \act u 
   - \pder{\s_{0}}{x}(x_{0} , p) \act u \!}.
\end{eqnarray*} 

\smallskip 

Now $\s_{0}(x , p) = p + \a(x) (p - x)$ and $\s_{0}(x , q) = q + \b(x) (q - x)$
for any $x \in \bD$, where $\a : \bD \to \RR$ and $\b : \bD \to \RR$ are
functions that are $\Cl{\infty}$ (since $\s_{0}$ is). 

\smallskip 

For any $u \in T_{x_{0}}\bD$ with $u \neq 0$, we have
$$
\pder{\s_{0}}{x}(x_{0} , p) \act u - \pder{\s_{0}}{x}(x_{0} , p) \act u
= 
\left( \Tg{x_{0}}{\a}{u} \right) \! (p - x_{0}) 
- \left( \Tg{x_{0}}{\b}{u} \right) \! (q - x_{0}) 
+ (\b(x_{0}) - \a(x_{0})) u,
$$ 
and thus
$$
\Tg{x_{0}}{f_{0}}{u} = (\b(x_{0}) - \a(x_{0})) \det{\left( y_{0} - x_{0} \, , \, u \right)}.
$$

\smallskip 

As $u$ is not parallel to $y_{0} - x_{0}$ by strict convexity of $\cD$, 
the point $x_{0}$ will be regular for $f_{0}$ if $\b(x_{0}) \neq \a(x_{0})$. 
But if $\b(x_{0})$ were equal to $\a(x_{0})$, we would get 
$y_{0} = \s_{0}(x_{0} , p) = p + \a(x_{0}) (p - x_{0})$ and 
$y_{0} = \s_{0}(x_{0} , q) = q + \b(x_{0}) (q - x_{0})$. 
Therefore $(1 - \a(x_{0})) (q - p) = 0$, and hence $\a(x_{0}) = \b(x_{0}) = 1$ since $p \neq q$. 
This would then imply $2 p = x_{0} + y_{0} = 2 q$, 
contradicting the fact that $p$ and $q$ are distinct points.
Thus $0$ is a regular value of $f_{0}$. 

\smallskip

Then, from Lemma~\ref{lem:reg-pts} with $K := \bD$, there exists $b \in (0 , a / 2)$ 
such that $0$ is a regular value of $f_{\l}$ for all $\l \in (-2 b , 2 b)$. 
This implies by Lemma~\ref{lem:reg-val} that $0$ is a regular value of $\rest{F}{(-2 b , 2 b) \times \bD}$, 
and hence that any $\l \in (-2 b , 2 b)$ is a regular value of $\rest{\pi}{W} : W \to (-2 b , 2 b)$, 
where $\pi : (-2 b , 2 b) \times \bD \to (-2 b , 2 b)$ is the natural projection 
and $W = (\rest{F}{(-2 b , 2 b) \times \bD})^{-1}(0)$.

\smallskip 

But $\dim{W} = 1 = \dim{(-2 b , 2 b)}$, so $\rest{\pi}{W}$ is a local diffeomorphism, 
thus a local homeomorphism, which implies that $\pi : W_{0} \to [-b , b]$ is also a local homeomorphism, 
where $W_{0} := W \cap ([-b , b] \times \bD)$. 

\smallskip 

Next, as $W_{0}$ is compact (since $W$ is closed in $(-2 b , 2 b) \times \bD$ 
and $[-b , b] \times \bD$ is compact) 
and $[-b , b]$ is connected, we get that $\pi : W_{0} \to [-b , b]$ is a covering map 
with a finite number of sheets by Lemma~\ref{lem:covering}.

\smallskip 

Since $\pi^{-1}(0) \cap W_{0} = \{ (0 , x_{0}) , (0 , y_{0}) \}$, 
we have $\card{\pi^{-1}(\l) \cap W_{0}} = 2$ for all $\l \in [-b , b]$.

Hence, given $\l \in [-b , b]$, there are exactly two distinct points $x, y \in \bD$ such that
$$
\det{V_{p}^{\l}(x) , V_{q}^{\l}(x)} = 0 
\quad \mbox{and} \quad  
\det{V_{p}^{\l}(y) , V_{q}^{\l}(y)} = 0.
$$
But this means that $\s_{\l}(x , p) = \s_{\l}(x , q) = y$ by Remark~\ref{rem:end-point}.  
So, there exist $t_{0}, t_{1} \in \RR$ such that $p = \f_{x}^{\l}(y , t_{0})$ 
and $q = \f_{x}^{\l}(y , t_{1})$ with $t_{0} \leq t_{1}$ after a suitable labelling of $x$ and $y$. 
As $\f_{x}^{\l}(y , \cdot) : \RR \to \Rn{2}$ is injective by point~(1) in Corollary~\ref{cor:basic-2}, 
such $t_{0}$ and $t_{1}$ are unique with $0 < t_{0} < t_{1} < 1$. 

\smallskip 

This proves Proposition~\ref{prop:p-q}.
\end{proof}

\bigskip

\begin{proof}[Proof of Proposition~\ref{prop:p-v}]~\\ 
\indent Fix a point $p \in \cD$ and a vector $V \in \Rn{2} \setmin \{ 0 \}$, 
and consider the $\Cl{\infty}$ function $F : (-a , a) \times \bD \to \RR$ defined by
$\disp F(\l , x) = f_{\l}(x) := \det{\! V \, , \, \pder{\f_{x}^{\l}}{t}(y , t) \!}$ 
with $\F_{\l}(x , y , t) = (x , p) \in N$. 
In other words, $\disp f_{\l}(x) := \det{\! V \, , \, \pder{\F_{\l}}{t}(\F_{\l}^{-1}(x , p)) \!}$ 
for all $(\l , x) \in (-a , a) \times \bD$ (see Figure~\ref{fig:prop-p-v}). 

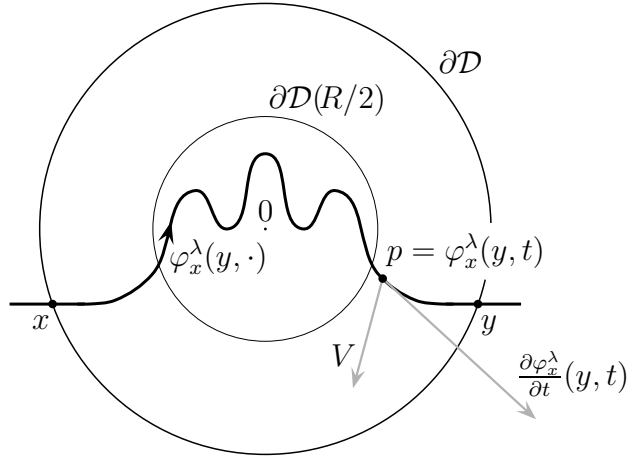
\begin{figure}[h] 

\begin{pspicture}(8,6.5) 
   \psdots[dotstyle=*,dotscale=0.3](4,3) 
   \uput{0.1}[90](4,3){$0$} 
   \psarc[linewidth=0.02](4,3){3}{1.5}{347} 
   \uput{0.1}[45](6.236,5){$\bD$} 
   \pscircle[linewidth=0.01](4,3){1.5} 
   \uput{1.5}[70](4,3){$\bD(\! R / 2)$} 
   \psdots[dotstyle=*,dotscale=0.9](1.172,2)(6.828,2) 
   \uput{0.15}[240](1.172,2){$x$} 
   \uput{0.15}[300](6.828,2){$y$} 
   \psline[linewidth=0.04](0.6,2)(1.6,2) 
   \pscurve[linewidth=0.04](1.5,2)(1.55,2)(1.6,2)(1.65,2)(2,2.05)(2.5,2.4)(3.07,3.51)
   (3.5,3)(4,4)(4.5,3)(4.93,3.51)(5.5,2.4)(6,2.05)(6.35,2)(6.4,2)(6.45,2)(6.5,2) 
   \psline[linewidth=0.04](6.4,2)(7.4,2) 
   \psline[linewidth=0.08]{c->}(2.709,2.94)(2.749,3.1) 
   \uput{0.15}[310](2.7,3){$\f_{x}^{\l}(y , \cdot)$} 
   \newgray{g7}{0.7} 
   \psline[linewidth=0.03,linecolor=g7]{c-}(5.56,2.34)(5.2,1) 
   \psline[linewidth=0.07,linecolor=g7]{c->}(5.204,1.02)(5.165,0.88) 
   \uput{0.17}[115](5.2,1){$V$} 
   \psline[linewidth=0.03,linecolor=g7]{c-}(5.56,2.34)(7.56,0.5) 
   \psline[linewidth=0.07,linecolor=g7]{c->}(7.5,0.557)(7.6,0.467) 
   \uput{0.28}[63](7.56,0.5){$\pder{\f_{x}^{\l}}{t}(y , t)$} 
   \psdots[dotstyle=*,dotscale=0.9](5.56,2.34) 
   \uput{0.15}[47](5.56,2.34){$p = \f_{x}^{\l}(y , t)$} 
\end{pspicture}  
   \caption{\label{fig:prop-p-v} Proof of Proposition~\ref{prop:p-v}} 
\end{figure}

\medskip 

Let $x_{0}$ and $y_{0}$ be the two intersection points of the straight line $p + \RR V$ with $\bD$, 
and write $G_{x}^{0}(y , t) = x + \om(x , y , t) (y - x)$
for all $(x , y , t) \in ((\bD \times \bD) \setmin \D) \times \RR$, 
where $\om$ is the function introduced in Step~1 in the proof of Proposition~\ref{prop:diffeo}. 
As we have $f_{0}^{-1}(0) = \{ x_{0} , y_{0} \}$, the value $0$ will be regular for $f_{0}$ 
if $x_{0}$ and $y_{0}$ are regular points of $f_{0}$.

\smallskip 

So, let us prove $x_{0}$ is a regular point of $f_{0}$ (arguments are the same for $y_{0}$). 

\medskip 

Using the diffeomorphism $\F_{0} : M \to N$ (see point~(2) in Proposition~\ref{prop:diffeo}), 
we can write 
$$
f_{0}(x) 
= 
\det{\! V \, , \, \left\{ \pder{\om}{t}(\F_{0}^{-1}(x , p)) \right\} \! (\s_{0}(x , p) - x) \!}
$$
for all $x \in \bD$.

Thus, for any $u \in T_{x_{0}}\bD$, we have 
\begin{eqnarray*}
   \Tg{x_{0}}{f_{0}}{u} 
   & = & 
   \det{\! V \, , \, \left\{ \Tg{\F_{0}^{-1}(x_{0} , p)}{\!\! \left( \! \pder{\om}{t} \! \right) \!}
   {\{ \Tg{(x_{0} , p)}{\F_{0}^{-1}}{u} \}} \right\} \! (\s_{0}(x_{0} , p) - x_{0}) \!} \\
   & & 
   + \det{\! V \, , \, \pder{\om}{t}(\F_{0}^{-1}(x_{0} , p)) \! 
   \left\{ \pder{\s_{0}}{x}(x_{0}, p) \act u - u \right\} \!} \\
   & = & 
   \det{\! V \, , \, \left\{ \Tg{\F_{0}^{-1}(x_{0} , p)}{\!\! \left( \! \pder{\om}{t} \! \right) \!}
   {\{ \Tg{(x_{0} , p)}{\F_{0}^{-1}}{u} \}} \right\} \! (y_{0} - x_{0}) \!} \\
   & & 
   + \pder{\om}{t}(\F_{0}^{-1}(x_{0} , p)) 
   \det{\! V \, , \, \pder{\s_{0}}{x}(x_{0} , p) \act u - u \!}, \\
\end{eqnarray*} 
that is,
$$
\Tg{x_{0}}{f_{0}}{u} 
= 
\pder{\om}{t}(\F_{0}^{-1}(x_{0} , p)) 
\det{\! V \, , \, \pder{\s_{0}}{x}(x_{0} , p) \act u - u \!}
$$
since $y_{0} - x_{0}$ is parallel to $V$.

\medskip 

As $\disp \pder{\om}{t}(\F_{0}^{-1}(x_{0} , p)) \neq 0$ 
(from property~$(ii)$ for $\om$ in Step~1 in the proof of Proposition~\ref{prop:diffeo}), 
$x_{0}$ will be a regular point of $f_{0}$ if the vector $\disp \pder{\s_{0}}{x}(x_{0} , p) \act u - u$ 
is not parallel to $V$ for $u \in T_{x_{0}}\bD$ with $u \neq 0$. 
In order to prove this, just write $\s_{0}(x , p) = p + \a(x) (p - x)$ for any $x \in \bD$, where
$\a : \bD \to \RR$ is a function that is $\Cl{\infty}$ (since $\s_{0}$ is) and positive 
(since $p$ is in the affine segment $]x , \s_{0}(x , p)[$). 
Then, for $u \in T_{x_{0}}\bD$ with $u \neq 0$, we have
$$
\pder{\s_{0}}{x}(x_{0} , p) \act u - u 
= 
\left( \Tg{x_{0}}{\a}{u} \right) \! (p - x_{0}) - (1 + \a(x_{0})) u,
$$
which is not parallel to $V$ since $p - x_{0}$ is parallel to $V$ 
and $u$ is not (by strict convexity of $\cD$). 

\medskip 

Hence we have shown that $0$ is a regular value of $f_{0}$, 
and we conclude exactly the same as in the end 
of the proof of Proposition~\ref{prop:p-q} with $c$ instead of $b$.
\end{proof}

\bigskip

We can now use all we proved in this section to eventually obtain what we wanted: 

\begin{theorem} \label{thm:admissible}
   There exists a number $\e_{0} \in (0 , a)$ such that for all $\l \in (-\e_{0} , \e_{0})$, 
   the set $\gS_{\l} := \{ \c^{\l}_{(x , y)} \st (x , y) \in (\bD \times \bD) \setmin \D \}$ 
   of parameterized curves $\c^{\l}_{(x , y)} : [0 , 1] \to \Rn{2}$ defined by 
   $\c^{\l}_{(x , y)}(t) := \f_{x}^{\l}(y , t)$ is admissible for $\cD$ and satisfies Property~(C).
\end{theorem}

\medskip

\begin{proof}~\\ 
\indent Define $\disp \e_{0} := \frac{1}{2} \min{\{ b , c \}} > 0$, where $b$ and $c$ are given respectively 
by Proposition~\ref{prop:p-q} and Proposition~\ref{prop:p-v}, and let $\l \in (-\e_{0} , \e_{0})$.

\smallskip 

The fact that the set $\gS_{\l}$ is admissible for $\cD$ follows from point~(1) 
in Corollary~\ref{cor:basic-2}, 
Proposition~\ref{prop:p-q} and Proposition~\ref{prop:p-v}.

\smallskip 

Property~(C) for $\gS_{\l}$ is a consequence of Proposition~\ref{prop:end-point} 
and point~(2) in Corollary~\ref{cor:basic-2}.
\end{proof}

\bigskip
\bigskip 


\subsection{Towards the Main Theorem} \label{sec:end} 
~

\medskip 

At this stage of the paper and following Arcostanzo's construction in~\cite{Arc94}, 
let us define for each $\l \in (-\e_{0} , \e_{0})$ 
the function $\sF_{\l} : T\cD = \cD \times \Rn{2} \to \RR$ by setting
\begin{equation*}
   \sF_{\l}(p , v) 
   :=
   \frac{1}{4} \! \int{\bD}{}{\ \ppder{d_{g_{0}}}{x}{y}(x , \s_{\l}(x , p)) \! 
   \left| \pder{\s_{\l}}{p}(x , p) \act v \right| \!\!}{x}
\end{equation*}
for all $(p , v) \in T\cD$.

\smallskip 

Since the distance function $d_{g_{0}}$ is $\Cl{\infty}$ on $(\bD \times \bD) \setmin \D$ 
with $\ppder{d_{g_{0}}}{x}{y}(x , y) > 0$ for all $(x , y) \in (\bD \times \bD) \setmin \D$ 
(point~(2) in Remark~\ref{rem:Arcostanzo}), 
we get from Theorem~\ref{thm:Arcostanzo} and Theorem~\ref{thm:admissible} that 
$\sF_{\l}$ is a \emph{smooth} Finsler metric on $\cD$ such that $d_{\sF_{\l}} = d_{g_{0}}$.

\smallskip

On the other hand, since we have $\c^{0}_{(x , y)}(t) = \f_{x}^{0}(y , t) = G_{x}^{0}(y , t)$ 
for all $(x , y) \in (\bD \times \bD) \setmin \D$ and $t \in [0 , 1]$, the set $\gS_{0}$ 
coincides with the set of maximal geodesics of $g_{0}$ in $\bar{\cD}$
after reparametrization by $[0 , 1]$. 
Thus, as mentionned in point~(4) of Remark~\ref{rem:Arcostanzo}, 
$\sF_{0}$ \emph{equals} the restriction to $T\cD$ of the Finsler metric $F_{0}$ on $\Hn{2}$ 
associated with $g_{0}$. 

\bigskip

We are now going to give some properties about $\sF_{\l}$ that will lead to the Main Theorem.

\medskip

The first one shows that $\sF_{\l}$ agrees with $F_{0}$ near the boundary $\bD$ of $\cD$, 
which is not a surprise since our construction of $\sF_{\l}$ has especially been made for this. 
Moreover, we prove that the region in $\cD$ near the boundary of $\bD$ 
on which $\sF_{\l}$ agrees with $F_{0}$ can actually be chosen in such a way 
that it does \emph{not} depend on the parameter $\l$. 
This uniformity will later ensure that the family of Finsler metrics we will obtain 
in the Main Theorem is \emph{smooth} with respect to $\l$. 

\begin{proposition} \label{prop:hyp}
   There exists $R_{0} \in (R/2 , R)$ such that for every $\l \in (-\e_{0} , \e_{0})$, 
   the Finsler metric $\sF_{\l}$ coincides with $F_{0}$ on $(\cD \setmin \bar{\cD(R_{0})}) \times \Rn{2}$.
\end{proposition}

\medskip

In order to establish this fact, we will need the following useful lemma which proves 
that $\sF_{\l}$ is invariant under the Euclidean isometries 
since all the objects we constructed so far have as much symmetry as the Euclidean circle $\bD$ has. 

\begin{lemma} \label{lem:rad-sym}
   For any $\l \in (-\e_{0} , \e_{0})$ and any linear Euclidean isometry $A \in \O(\Rn{2})$,
   we have $\sF_{\l}(A(p) , A(v)) = \sF_{\l}(p , v)$ for all $p \in \cD$ and $v \in \Rn{2}$.
\end{lemma}

\medskip

\begin{remark} \label{rem:Klein} 
Before proving this lemma, recall for the reader's convenience 
that the Klein metric $g_{0}$ whose associated Finsler metric is $F_{0}$ is given by
\begin{equation*}  
   g_{0}(p) \act (v , v) 
   = 
   F_{0}(p , v)^{2} 
   = 
   \frac{|v|^{2}}{1 - |p|^{2}} + \frac{\scal{v}{p}^{2}}{(1 - |p|^{2})^{2}}
\end{equation*}
for all $p \in \Hn{2}$ and $v \in \Rn{2}$.
\end{remark}

\bigskip

\begin{proof}[Proof of Lemma~\ref{lem:rad-sym}]~\\ 
\indent Fix $\l \in \RR$ and $A \in \O(\Rn{2})$. 

\smallskip 

For every $p \in \Hn{2}$, we have $|A(p)| = |p|$, and thus $\a_{\l}(|A(p)|) = \a_{\l}(|p|)$.
Since the Klein metric $g_{0}$ on $\Hn{2}$ is invariant under $A$ 
(\emph{i.e.}, $A^{*}g_{0} = g_{0}$) by the formula for $g_{0}$ in Remark~\ref{rem:Klein},
we get that $g_{\l}$ is $A$-invariant too from the very definition of $g_{\l}$. 

\medskip 

Hence, given any $x , y \in \Hn{2}$ and any $t \in \RR$, 
we have $G_{A(x)}^{0}(A(y) , t) = A(G_{x}^{0}(y , t))$
and $G_{A(x)}^{\l}(A(y) , t) = A(G_{x}^{\l}(y , t))$,
which implies
\begin{equation} \label{equ:rad-sym} 
\begin{split} 
   \f_{A(x)}^{\l}(A(y) , t) 
   & =  
   (1 - \r(t)) G_{A(x)}^{0}(A(y) , t) + \r(t) G_{A(x)}^{\l}(A(y) , t) \\
   & = 
   (1 - \r(t)) A(G_{x}^{0}(y , t)) + \r(t) A(G_{x}^{\l}(y , t)) \\
   & = 
   A \! \Big( \! (1 - \r(t)) G_{x}^{0}(y , t) + \r(t) G_{x}^{\l}(y , t) \! \Big) \\
   & =  
   A(\f_{x}^{\l}(y , t))
\end{split}
\end{equation}
for all $t \in \RR$.

\medskip 

So, for every $\l \in (-\e_{0} , \e_{0})$ and $(x , p) \in N$, 
we have $\s_{\l}(A(x) , A(p)) = A(\s_{\l}(x , p))$,
and therefore, for every $\l \in (-\e_{0} , \e_{0})$, $p \in \cD$ and $v \in \Rn{2}$, one has
\begin{eqnarray*}
   \sF_{\l}(A(p) , A(v)) 
   & = &
   \frac{1}{4} \! \int{\bD}{}{\ \ppder{d_{g_{0}}}{x}{y}(x , \s_{\l}(x , A(p))) \! 
   \left| \pder{\s_{\l}}{p}(x , A(p)) \act A(v) \right| \!\!}{x} \\
   & = &
   \frac{1}{4} \! \int{\bD}{}{\ \ppder{d_{g_{0}}}{x}{y}(A(x) , \s_{\l}(A(x) , A(p))) \! 
   \left| \pder{\s_{\l}}{p}(A(x) , A(p)) \act A(v) \right| \!\!}{x} \\
   & & 
   \mbox{(since the canonical Euclidean measure $\textrm{d}x$ on $\bD$ is $A$-invariant)} \\
   & = &
   \frac{1}{4} \! \int{\bD}{}{\ \ppder{d_{g_{0}}}{x}{y}(A(x) , A(\s_{\l}(x , p))) \! 
   \left| A(\pder{\s_{\l}}{p}(x , p) \act v) \right| \!\!}{x} \\
   & = &
   \frac{1}{4} \! \int{\bD}{}{\ \ppder{d_{g_{0}}}{x}{y}(x , \s_{\l}(x , p)) \! 
   \left| \pder{\s_{\l}}{p}(x , p) \act v \right| \!\!}{x} \\ 
   & & 
   \mbox{(since $d_{g_{0}}$ and $|\cdot|$ are $A$-invariant)} \\
   & = &
   \sF_{\l}(p , v). 
\end{eqnarray*}

\smallskip 

This ends the proof of Lemma~\ref{lem:rad-sym}.
\end{proof}

\bigskip

\begin{proof}[Proof of Proposition~\ref{prop:hyp}]~\\ 
\indent Applying Lemma~\ref{lem:diff-perturb} with $\L := (-a , a)$ 
and $f_{\l} := \F_{\l}$ (which is a diffeomorphism by point~(2) in Proposition~\ref{prop:diffeo}), 
we get that $h : (-a , a) \times M \to (-a , a) \times N$ defined by
$h(\l , (x , y , t)) := (\l , (x , \f_{x}^{\l}(y , t)))$ is a diffeomorphism,
hence a homeomorphism.

\medskip

So, consider the open set 
$U := \{ (x , y) \in \bD \times \bD \st |x - y| > R \} \times (3/4 , +\infty)$ in $M$, 
fix $x_{0} \in \bD$, and define the compact set 
$K := \{ y \in \bD \st |x_{0} - y| \geq \sqrt{3} R \} \inc \bD$. 

\smallskip 

Since $\{ x_{0} \} \times K \times \{ 1 \} \inc U$, 
the compact set 
$$
[-a/2 , a/2] \times \{ x_{0} \} \times K 
= 
h([-a/2 , a/2] \times \{ x_{0} \} \times K \times \{ 1 \})
$$ 
is included in the open set $\cU := h((-a , a) \times U)$ of $(-a , a) \times N$. 
Thus, there exists a number $\tau_{0} \in (0 , R/2)$ such that
$[-a/2 , a/2] \times \{ x_{0} \} \times \S \inc \cU$,
where $\S := \{ (1 + \tau) y \st y \in K \ \ \mbox{and} \ \ \tau \in (-\tau_{0} , \tau_{0}) \}$.

\medskip 

But Lemma~\ref{lem:rad-sym} implies that for any $(\l , (x , p)) \in \cU$ and $A \in \O(\Rn{2})$,
we have $(\l , (A(x) , A(p))) \in \cU$.
Hence, if 
$$
E := \{ (x , (1 + \tau) y) \st x , y \in \bD 
\ \ \mbox{and} \ \ 
|x - y| \geq \sqrt{3} R 
\ \ \mbox{and} \ \ 
\tau \in (-\tau_{0} , \tau_{0}) \},
$$
we get
\begin{equation} \label{equ:inclusion}
   [-a/2 , a/2] \times E \ 
   =
   \!\!\! \bigcup_{A \in \O(\Rn{2})} \!\!\! [-a/2 , a/2] \times \{ A(x_{0}) \} \times A(\S) \inc \cU.
\end{equation}

\medskip 

Now define $R_{0} := R - \tau_{0} \in (R/2 , R)$, 
and pick $\l \in [-a/2 , a/2]$, $x \in \bD$ and $p \in \cD \setmin \bar{\cD(R_{0})}$. 

\smallskip 

Let $z \in \bD$ be the intersection point between $\bD$ and the open half line $x + \RR_{+}^{*} (p - x)$.

\medskip 

If $|x - z| \geq \sqrt{3} R$, then $(x , p) \in E$, 
and thus $(\l , (x , p)) \in \cU$ by Equation~\ref{equ:inclusion}.
This means that there are $y \in \bD$ and $t \in (3/4 , +\infty)$ satisfying $p = \f_{x}^{\l}(y , t)$,
which implies that $\s_{\l}(x , p) = y$.
But, since $\r \equiv 0$ on $(3/4 , +\infty)$ by property~(2) in Equation~\ref{equ:rho}, we also
have $p = \f_{x}^{\l}(y , t) = G_{x}^{0}(y , t) = \f_{x}^{0}(y , t)$,
and hence $\s_{0}(x , p) = y$.

\medskip 

If $|x - z| < \sqrt{3} R$, then the image of the
$g_{0}$-geodesic $G_{x}^{0}(z , \cdot) : \RR \to \Hn{2}$ lies in the open set
$\Hn{2} \setmin \bar{\cD(R/2)}$ of $\Hn{2}$, since this image is equal to the
intersection of the straight line $(xz)$ with $\Hn{2}$ and since 
any chord of $\bD$ that is tangent to
$\bD(R/2)$ has a Euclidean length equal to $\sqrt{3} R$.
Since the Riemannian metrics $g_{\l}$ and $g_{0}$ coincide on $\Hn{2} \setmin \bar{\cD(R/2)}$, we
get that the $g_{\l}$-geodesic $G_{x}^{\l}(z , \cdot) : \RR \to \Hn{2}$ is actually
equal to the $g_{0}$-geodesic $G_{x}^{0}(z , \cdot) : \RR \to \Hn{2}$.
Thus, $\f_{x}^{\l}(z , t) = G_{x}^{0}(z , t) = \f_{x}^{0}(y , t)$ for all $t \in \RR$.
But the definition of $z$ says that $p \in (xz)$, 
which means there is $t_{0} \in \RR$ such that $p = G_{x}^{0}(z , t_{0})$. 
So, $p = \f_{x}^{\l}(z , t_{0}) = \f_{x}^{0}(z , t_{0})$,
and therefore $\s_{\l}(x , p) = z = \s_{0}(x , p)$.

\medskip 

Conclusion: for every $\l \in [-a/2 , a/2]$, $x \in \bD$ and $p \in \cD \setmin \bar{\cD(R_{0})}$, we have
$\s_{\l}(x , p) = \s_{0}(x , p)$.

\medskip 

Hence, for any $\l \in (-\e_{0} , \e_{0}) \inc [-a/2 , a/2]$, $p \in \cD \setmin \bar{\cD(R_{0})}$ 
and $v \in \Rn{2}$, we can write
\begin{eqnarray*}
   \sF_{\l}(p , v) 
   & = &
   \frac{1}{4} \! \int{\bD}{}{\ \ppder{d_{g_{0}}}{x}{y}(x , \s_{\l}(x , p)) \! 
   \left| \pder{\s_{\l}}{p}(x , p) \act v \right| \!\!}{x} \\
   & = &
   \frac{1}{4} \! \int{\bD}{}{\ \ppder{d_{g_{0}}}{x}{y}(x , \s_{0}(x , p)) \! 
   \left| \pder{\s_{0}}{p}(x , p) \act v \right| \!\!}{x} \\
   & = &
   \sF_{0}(p , v) = F_{0}(p , v).
\end{eqnarray*} 

\smallskip 

This proves Proposition~\ref{prop:hyp}.
\end{proof}

\bigskip

From now on, for each $\l \in (-\e_{0} , \e_{0})$ and thanks to Proposition~\ref{prop:hyp}, 
we extend the Finsler metric $\sF_{\l}$ on the whole $\Hn{2}$ 
by setting $\sF_{\l}(p , v) = F_{0}(p , v)$ 
for all $(p , v) \in (\Hn{2} \setmin \bar{\cD(R_{0})}) \times \Rn{2}$. 

\medskip

Then we have

\begin{proposition} \label{prop:properties}
   The family of Finsler metrics $(\sF_{\l})_{\l \in (-\e_{0} , \e_{0})}$ on $\Hn{2}$ 
   satisfies the following:
    
   \begin{enumerate}
      \item the function $\F : (-\e_{0} , \e_{0}) \times T\Hn{2} \to \RR$ 
      defined by $\F(\l , \cdot) := \sF_{\l}(\cdot)$ for all $\l \in (-\e_{0} , \e_{0})$ is continuous 
      and $\Cl{\infty}$ on $(-\e_{0} , \e_{0}) \times \Hn{2} \times (\Rn{2} \setmin \{ 0 \})$; and 
      
      \smallskip 

      \item there exists $\e \in (0 , \e_{0})$ such that for each $\l \in (-\e , \e)$, 
      the smooth Finsler metric $\sF_{\l}$ is strongly convex and has no conjugate points. 
   \end{enumerate} 
\end{proposition}

\medskip

\begin{proof}~\\ 
\indent $\bullet$ \textsf{Point~(1):}   
Consider the map $\Ups : (-\e_{0} , \e_{0}) \times \Hn{2} \times \Rn{2} \times \bD \to \Rn{2}$  
defined by 
$$
\Ups((\l , p , v) , x) 
:= 
\frac{1}{4} \ppder{d_{g_{0}}}{x}{y}(x , \s_{\l}(x , p)) \pder{\s_{\l}}{p}(x , p) \act v.
$$

\smallskip 

Since $d_{g_{0}}$ is $\Cl{\infty}$ on $(\bD \times \bD) \setmin \{ (x , x) \st x \in \bD \}$ and 
$\disp (\l , x , p) \mapsto \s_{\l}(x , p)$ is a $\Cl{\infty}$ map 
from $(-\e_{0} , \e_{0}) \times \bD \times \cD$ to $\bD$ 
by Proposition~\ref{prop:end-point} 
which satisfies $\s_{\l}(x , p) \neq x$ for all $\l \in (-\e_{0} , \e_{0})$ 
and $(x , p) \in \bD \times \cD$, 
the positive function $\disp (\l , x , p) \mapsto \ppder{d_{g_{0}}}{x}{y}(x , \s_{\l}(x , p))$ 
is $\Cl{\infty}$ on $(-\e_{0} , \e_{0}) \times \bD \times \cD$, 
and therefore $\Ups$ is $\Cl{\infty}$ on $(-\e_{0} , \e_{0}) \times \cD \times \Rn{2} \times \bD$. 

\smallskip 

Then, using the same arguments as in Remark~\ref{rem:smooth}, we get that 
the function $\F$ is continuous on $(-\e_{0} , \e_{0}) \times \cD \times \Rn{2}$ 
and $\Cl{\infty}$ on $(-\e_{0} , \e_{0}) \times \cD \times (\Rn{2} \setmin \{ 0 \})$ 
since we have 
$$
\F(\l , (p , v)) = \int{\bD}{}{|\Ups((\l , p , v) , x)| \!}{x} 
$$
for all $\l \in (-\e_{0} , \e_{0})$ and $(p , v) \in T\Hn{2} = \Hn{2} \times \Rn{2}$. 

\smallskip 

On the other hand, since $\F(\l , (p , v)) = F_{0}(p , v)$ 
for all $\l \in (-\e_{0} , \e_{0})$ and $(p , v) \in (\Hn{2} \setmin \bar{\cD(R_{0})}) \times \Rn{2}$
by construction, the function $\F$ is continuous 
on $(-\e_{0} , \e_{0}) \times (\Hn{2} \setmin \bar{\cD(R_{0})}) \times \Rn{2}$ and $\Cl{\infty}$ 
on $(-\e_{0} , \e_{0}) \times (\Hn{2} \setmin \bar{\cD(R_{0})}) \times (\Rn{2} \setmin \{ 0 \})$. 

\smallskip

Conclusion: $\F$ is a continuous function that is $\Cl{\infty}$ 
on $(-\e_{0} , \e_{0}) \times \Hn{2} \times (\Rn{2} \setmin \{ 0 \})$. 

\medskip

$\bullet$ \textsf{Point~(2):}  
As a consequence of the first point, the map 
$\l \mapsto \F(\l , \cdot) = \sF_{\l}(\cdot)$ from $(-\e_{0} , \e_{0})$ 
to $\Cl{2}(T\Hn{2} \setmin \{ 0 \} , \RR)$ 
is continuous when $\Cl{2}(T\Hn{2} \setmin \{ 0 \} , \RR)$ is endowed with the $\Cl{2}$-topology. 

\smallskip 

This first implies that $\frac{\partial^{2} \sF_{\l}^{2}}{\partial v^{2}}$ is close to 
$\frac{\partial^{2} F_{0}^{2}}{\partial v^{2}}$ in $\Cl{0}(T\Hn{2} \setmin \{ 0 \} , \RR)$ 
with respect to the $\Cl{0}$-topology whenever $\l \in (-\e_{0} , \e_{0})$ is sufficiently small. 
Hence, there exists $\e_{1} \in (0 , \e_{0})$ such that $\sF_{\l}$ is strongly convex 
for all $\l \in (-\e_{1} , \e_{1})$ since the hyperbolic Finsler metric $F_{0}$ is. 

\smallskip 

Furthermore, if $\cV \inc T(T\Hn{2})$ is the vertical vector bundle over $T\Hn{2}$ 
(the kernel of the differential of the natural projection $T\Hn{2} \to \Hn{2}$) 
and $\f_{\l} = \flowR{\f_{\l}}{t}$ is the geodesic flow of $\sF_{\l}$ on $T\Hn{2} \setmin \{ 0 \}$ 
for any $\l \in (-\e_{1} , \e_{1})$ (\emph{i.e.}, the Euler-Lagrange flow of the non-degenerate 
Lagrangian $\cL_{\l} := \frac{1}{2} \sF_{\l}^{2} : T\Hn{2} \setmin \{ 0 \} \to \RR$), 
the map $\l \mapsto \f_{\l}$ from $(-\e_{1} , \e_{1})$ 
to $\Cl{1}(\RR \times (T\Hn{2} \setmin \{ 0 \}) , T\Hn{2} \setmin \{ 0 \})$ 
is continuous when $\Cl{1}(\RR \times (T\Hn{2} \setmin \{ 0 \}) , T\Hn{2} \setmin \{ 0 \})$ 
is endowed with the $\Cl{1}$-topology. 

\smallskip 

Since the hyperbolic Finsler metric $F_{0}$ has no conjugate points, we have 
$$
\cV_{(p , v)} \cap T_{\f_{0}^{t}(p , v)} \f_{0}^{-t} (\cV_{\f_{0}^{t}(p , v)}) \neq \{ 0 \} 
\ \ \mbox{for all} \ \ (t , (p , v)) \in \RR \times (T\Hn{2} \setmin \{ 0 \}). 
$$ 
Thus, there exists $\e \in (0 , \e_{1})$ such that 
$$
\cV_{(p , v)} \cap T_{\f_{\l}^{t}(p , v)} \f_{\l}^{-t} (\cV_{\f_{\l}^{t}(p , v)}) \neq \{ 0 \} 
\ \ \mbox{for all} \ \ (t , (p , v)) \in \RR \times (T\Hn{2} \setmin \{ 0 \}) 
\ \ \mbox{and all} \ \ \l \in (-\e , \e). 
$$
But this is equivalent to saying that the Finsler metric $\sF_{\l}$ has no conjugate points 
whenever $\l \in (-\e , \e)$. 
\end{proof}

\bigskip

\begin{proposition} \label{prop:non-Riemann}
   For any $\l \in (-\e , \e)$, the Finsler metric $\sF_{\l}$ is \emph{not Riemannian} 
   whenever $\l \neq 0$.    
\end{proposition}

\medskip

Before proving this result, we will need to establish the following:

\begin{lemma} \label{lem:geodesics}
   There exists $r_{0} \in (0 , R/2)$ such that for every $\l \in (-\e , \e)$, 
   all the geodesics of the restriction of the Riemannian metric $g_{\l}$ to $\cD(r_{0})$
   are geodesics for $\sF_{\l}$.
\end{lemma}

\medskip

\begin{proof}~\\ 
\indent It will consists in four technical steps. 

\smallskip 

We first show that for any $\l \in \RR$ and $x \in \bD$, the parameterized curve 
$\f_{x}^{\l}(-x , \cdot) : \RR \to \Rn{2}$ passes through the origin $0$ at $t = 1 / 2$ 
(here $y = -x \in \bD$ is the symmetric of $x$ about $0$). 
Then, remembering that 
$\f_{x}^{\l}(y , \cdot) = (1 - \r) G_{x}^{0}(y , \cdot) + \r G_{x}^{\l}(y , \cdot)$ for all $y \in \bD$ 
and using the fact that $\r \equiv 1$ on $[1 / 3 , 2 / 3]$, 
we deduce the lemma. 

\medskip

$\bullet$ \textsf{Step~1:}  
Fix arbitrary $\l \in \RR$ and $x \in \Hn{2} \setmin \{ 0 \}$.

\smallskip 

If $A \in \O(\Rn{2})$ is the Euclidean reflection 
through the line $x^{\perp}$, we have 
$$
\f_{x}^{\l}(-x , 1/2) = \f_{A(-x)}^{\l}(A(x) , 1/2) = A(\f_{-x}^{\l}(x , 1/2))
$$
on the one hand by Equation~\ref{equ:rad-sym}, and $\f_{-x}^{\l}(x , 1/2) = \f_{x}^{\l}(-x , 1/2)$ 
on the other hand by the first point in Lemma~\ref{lem:basic-1}.
Therefore $\f_{x}^{\l}(-x , 1/2) \in x^{\perp}$.

\smallskip 

Next, if $B \in \O(\Rn{2})$ is the reflection 
through the line $\RR x$, we have 
$$
\f_{x}^{\l}(-x , 1/2) = \f_{B(x)}^{\l}(B(-x) , 1/2) = B(\f_{x}^{\l}(-x , 1/2))
$$
by Equation~\ref{equ:rad-sym}, and hence $\f_{-x}^{\l}(x , 1/2) \in \RR x$.

\smallskip 

This shows that $\f_{x}^{\l}(-x , 1/2) = 0$.

\medskip

$\bullet$ \textsf{Step~2:}  
Now, as in the proof of Proposition~\ref{prop:hyp}, we will make use of the map 
$h : (-a , a) \times M \to (-a , a) \times N$ defined by
$h(\l , (x , y , t)) = (\l , (x , \f_{x}^{\l}(y , t)))$. 

\smallskip 

Consider the open set $V := ((\bD \times \bD) \setmin \D) \times (1 / 3 , 2 / 3)$ in $M$, 
and define the compact set $L := \{ (x , -x) \st x \in \bD \} \inc \bD \times \bD$. 

\smallskip 

Since $L \times \{ 1 / 2 \} \inc V$, the set 
$h([-a / 2 , a / 2] \times L \times \{ 1 / 2 \})$ 
is included in the open set $\cV := h((-a , a) \times V)$ of $(-a , a) \times N$. 
But Step~1 implies that we have 
$[-a / 2 , a / 2] \times \bD \times \{ 0 \} \inc h([-a / 2 , a / 2] \times L \times \{ 1 / 2 \})$. 
So, using the compactness of $[-a / 2 , a / 2] \times \bD \times \{ 0 \}$, 
there exists a number $r_{0} \in (0 , R/2)$ such that 
$[-a / 2 , a / 2] \times \bD \times \cD(r_{0}) \inc \cV$,
which means that for every $\l \in [-a / 2 , a / 2]$, $x \in \bD$ and $p \in \cD(r_{0})$,
there are $z \in \bD$ and $\tau \in (1 / 3 , 2 / 3)$ satisfying $p = \f_{x}^{\l}(z , \tau)$.

\medskip

$\bullet$ \textsf{Step~3:} 
For any $\l \in (-\e , \e) \inc [-a / 2 , a / 2]$ and $x , y \in \bD$ with $x \neq y$, 
we have $\{ t \in (0 , 1) \st \f_{x}^{\l}(y , t) \in \cD(r_{0}) \} \inc (1 / 3 , 2 / 3)$.

\smallskip 

Indeed, if $t \in (0 , 1)$ satisfies $p := \f_{x}^{\l}(y , t) \in \cD(r_{0})$,
then by Step~2 there exist $z \in \bD$ and $\tau \in (1 / 3 , 2 / 3)$ 
such that $\f_{x}^{\l}(y , t) = \f_{x}^{\l}(z , \tau)$.
So $\F_{\l}(x , y , t) = \F_{\l}(x , z , \tau)$, and therefore $t = \tau \in (1 / 3 , 2 / 3)$ 
since $\F_{\l} : M \to N$ is injective by point~(2) in Proposition~\ref{prop:diffeo}.

\medskip

$\bullet$ \textsf{Step~4:}  
For this last step, fix $\l \in (-\e , \e)$, let $c : I \to \Hn{2}$ be a $g_{\l}$-geodesic 
such that $c(I) \inc \cD(r_{0})$, where $I \inc \RR$ is an interval, 
and prove that $c$ is also a $\sF_{\l}$-geodesic. 

\smallskip 

For doing this, choose arbitrary $s_{0} , s_{1} \in I$ with $s_{0} < s_{1}$, 
and define $p_{0} := c(s_{0})$ and $p_{1} := c(s_{1})$. 

By Proposition~\ref{prop:p-q}, there exists 
$(x , y , t_{0} , t_{1}) \in ((\bD \times \bD) \setmin \D) \times (0 , 1) \times (0 , 1)$
such that $p_{0} = \f_{x}^{\l}(y , t_{0})$ and $p_{1} = \f_{x}^{\l}(y , t_{1})$ with $t_{0} \leq t_{1}$. 

\smallskip 

Then, by Theorem~\ref{thm:Arcostanzo}, the parameterized curve 
$\k : [t_{0} , t_{1}] \inc (0 , 1) \to \Hn{2}$ defined by 
$\k(t) := \f_{x}^{\l}(y , t)$ is a $\sF_{\l}$-geodesic. 
This implies that the reparametrized curve $\a : [s_{0} , s_{1}] \to \Hn{2}$ defined by 
$$
\a(s) := \k(t_{0} + (s - s_{0}) (t_{1} - t_{0}) / (s_{1} - s_{0}))
$$ 
is a $\sF_{\l}$-geodesic too. 

\smallskip 

Now, since $p_{0} = c(s_{0})$ and $p_{1} = c(s_{1})$ are in $\cD(r_{0})$, 
we have $t_{0} , t_{1} \in (1 / 3 , 2 / 3)$ by Step~3,
and hence $[t_{0} , t_{1}] \inc (1 / 3 , 2 / 3)$. 
But $\r \equiv 1$ on $[1 / 3 , 2 / 3]$ by properties~(1) and~(3) in Equation~\ref{equ:rho}, 
so we get 
\begin{equation} \label{equ:geodesics}
   \k(t) = G_{x}^{\l}(y , t) \ \ \mbox{for all} \ \ t \in [t_{0} , t_{1}].
\end{equation}

\smallskip 

This leads to considering the reparameterized curve $\bar{c} : [s_{0} , s_{1}] \to \Hn{2}$ 
defined by 
$$ 
\bar{c}(s) := G_{x}^{\l}(y , t_{0} + (s - s_{0}) (t_{1} - t_{0}) / (s_{1} - s_{0})) 
$$ 
which is a $g_{\l}$-geodesic that satisfies 
$$ 
\bar{c}(s_{0}) = G_{x}^{\l}(y , t_{0}) = \k(t_{0}) = p_{0} = c(s_{0}) 
\quad \mbox{and} \quad 
\bar{c}(s_{1}) = G_{x}^{\l}(y , t_{1}) = \k(t_{1}) = p_{0} = c(s_{1}) 
$$
by Equation~\ref{equ:geodesics}. 

\smallskip 

Thus, $\bar{c} = \rest{c}{[s_{0} , s_{1}]}$ since $g_{\l}$ has no conjugate points. 
This writes 
\begin{eqnarray*} 
   c(s) = \bar{c}(s) 
   & = & 
   G_{x}^{\l}(y , \underbrace{t_{0} 
   + (s - s_{0}) (t_{1} - t_{0}) / (s_{1} - s_{0})}_{\in [t_{0} , t_{1}]}) \\ 
   & = & 
   \k(t_{0} + (s - s_{0}) (t_{1} - t_{0}) / (s_{1} - s_{0})) \\ 
   & & 
   \mbox{(by Equation~\ref{equ:geodesics})} \\ 
   & = & \a(s) 
\end{eqnarray*}
for all $s \in [s_{0} , s_{1}]$. 
Hence $\rest{c}{[s_{0} , s_{1}]} = \a$, which shows that $\rest{c}{[s_{0} , s_{1}]}$ 
is a $\sF_{\l}$-geodesic (since $\a$ is). 

\smallskip 

As this holds for arbitrary $s_{0} , s_{1} \in I$ with $s_{0} < s_{1}$, 
we have proved that $c : I \to \Hn{2}$ is a $\sF_{\l}$-geodesic. 

\smallskip 

This establishes Lemma~\ref{lem:geodesics}. 
\end{proof}

\bigskip

\begin{proof}[Proof of Proposition~\ref{prop:non-Riemann}]~\\ 
\indent Let $\l \in (-\e , \e)$ with $\l \neq 0$. 

\smallskip 

As in~\cite{Arc94}, we use the following theorem of Beltrami to verify that within $\cD$ the
$\sF_{\l}$-geodesics do not arise as geodesics for a metric diffeomorphic to $g_{0}$:

\begin{theorem}[Beltrami. See~\cite{Spi75}, Chapter~7, page~26] 
   If $(X , g)$ is a connected Riemannian manifold such that for every point $p \in X$,
   there is a chart about $p$ that maps the $g$-geodesics onto straight lines, 
   then $(X , g)$ has constant sectional curvature.
\end{theorem}

\medskip 

Now, if $\sF_{\l}$ were Riemannian, then by the boundary rigidity 
of $(\bar{\cD} , \rest{{g_{0}}}{\bar{\cD}})$
given in Theorem~\ref{thm:Riemann-boundary-rigid}, $\sF_{\l}$ would
be isometric to $F_{0}$ in restriction to $\cD$, which would imply that the $\sF_{\l}$-geodesics
within $\cD$ are diffeomorphically mapped onto straight lines in $\Rn{2}$.

\medskip 

In particular, this would be true for all the $\sF_{\l}$-geodesics
within the open set $\cD(r_{0})$ defined in Lemma~\ref{lem:geodesics}.
But this lemma says that every geodesic of the restriction of $g_{\l}$ to $\cD(r_{0})$
is a geodesic for $\sF_{\l}$, and therefore the $g_{\l}$-geodesics within $\cD(r_{0})$ 
would be diffeomorphically mapped onto straight lines in $\Rn{2}$.

\medskip 

Hence the curvature of $g_{\l}$ would be constant on $\cD(r_{0})$ by Beltrami's theorem,
which is impossible by point~$(3)$ in Propositon~\ref{prop:conformal}.
\end{proof}

\bigskip
\bigskip
\bigskip


\bibliographystyle{acm}
\bibliography{cnv-7}

\end{document}